\documentclass[11pt,reqno]{amsart}
\numberwithin{equation}{section}
\usepackage{amsmath,amsthm}
\usepackage{mathptmx}  
\usepackage{amssymb}
\usepackage{multicol}
\DeclareFontFamily{U}{BOONDOX-calo}{\skewchar\font=45 }
\DeclareFontShape{U}{BOONDOX-calo}{m}{n}{
  <-> s*[1.05] BOONDOX-r-calo}{}
\DeclareFontShape{U}{BOONDOX-calo}{b}{n}{
  <-> s*[1.05] BOONDOX-b-calo}{}
\DeclareMathAlphabet{\mathcalboondox}{U}{BOONDOX-calo}{m}{n}
\SetMathAlphabet{\mathcalboondox}{bold}{U}{BOONDOX-calo}{b}{n}
\DeclareMathAlphabet{\mathbcalboondox}{U}{BOONDOX-calo}{b}{n}

\usepackage{mathtools}
\usepackage{xfrac}

\newcommand{\mcb}[1]{{\mathcalboondox #1}}
\usepackage{amssymb}
\usepackage{amsmath}
\usepackage{amsthm}
\usepackage{diagbox}
\usepackage{booktabs}
\usepackage{enumitem}
\usepackage[sans]{dsfont} 
\usepackage{bbm}
\usepackage{changebar}
\usepackage[colorlinks]{hyperref}
\usepackage{a4wide}
\usepackage[usenames,dvipsnames,svgnames,table]{xcolor}
\usepackage{xcolor}

\usepackage{tikz}
\usetikzlibrary{arrows,shapes,automata,petri,positioning,calc}
\tikzset{
    place/.style={
        circle,
        thick,
        draw=black,
        fill=gray!50,
        minimum size=20mm,
    },
        state/.style={
        circle,
        thick,
        draw=blue!75,
        fill=blue!20,
        minimum size=20mm,
    },
}

\tikzset{
    cross/.pic = {
    \draw[rotate = 45] (-0.2,0) -- (0.2,0);
    \draw[rotate = 45] (0,-0.2) -- (0, 0.2);
    }
}

\usepackage{ulem}
\usepackage{setspace}

\newtheorem{thm}{Theorem}[section]
\newtheorem{lem}[thm]{Lemma}

\newtheorem{prop}[thm]{Proposition}

\newtheorem{definition}[thm]{Definition}

\newtheorem{rem}[thm]{Remark}

\newcommand\ve{\varepsilon}

\usepackage{comment}

\title[{ Diffusive fluctuations of Long-range symmetric exclusion with a slow barrier}
]{Diffusive fluctuations of Long-range\\ symmetric exclusion with a slow barrier }
\author{Pedro Cardoso, Patr\'icia   Gon\c calves, Byron Jim\'enez-Oviedo}

\begin{document}
\begin{abstract}
In this article we obtain the equilibrium fluctuations of a symmetric exclusion process in $\mathbb{Z}$ with long jumps. The transition probability of the jump from $x$ to $y$ is proportional to $|x-y|^{-\gamma-1}$. Here we restrict to the choice  $\gamma \geq 2$ so that the system has a diffusive behavior. Moreover, when particles move between $\mathbb{Z}_{-}^{*}$ and $\mathbb N$, the jump rates are slowed down by a factor $\alpha n^{-\beta}$, where $\alpha>0$, $\beta\geq 0$ and $n$ is the  scaling parameter. Depending on the values of $\beta$ and $\gamma$, we obtain several  stochastic partial differential equations, corresponding to a heat equation without boundary conditions, or with Robin boundary conditions or Neumann boundary conditions.
\end{abstract}
\maketitle

\section{Introduction}

Since  Spitzer \cite{spitzer} introduced in the mathematical community the subject of  Interacting Particle Systems (IPS), this has become quite an active field of research involving mathematicians   and also theoretical physicists from different fields. The physical motivation to study this type of systems comes from Statistical Mechanics, where the goal is to  study the global behavior of some thermodynamic quantity (ies) (for example, the density of a fluid) from the microscopic interactions between its constituent molecules/particles. Since the number of molecules is of the order of Avogrado's number, one cannot give a complete description of the microscopic state of the system. Alternatively, the goal is to understand the macroscopic behavior of the system from the microscopic interactions,  which are assumed to be random, i.e. each molecule/particle behaves as a continuous-time random walk, which evolves according to some prescribed dynamical rule. In this way, one can make a probabilistic analysis of the whole system. One of the goals in the literature of IPS is to describe the macroscopic evolution equations of some quantity(ies) by a scaling limit procedure. We fix then a scaling parameter $n$ which connects the macroscopic space, a continuous space where the solutions of the evolution equations will be defined, to the microscopic space, a discrete space where the  particles will evolve (randomly).

In this setting two questions naturally appear: first, how to deduce the space-time evolution of the thermodynamic quantity of interest? This  is known as the \textit{hydrodynamic limit}, which gives  a partial differential equation (PDE),  a deterministic limit,   for the space-time evolution of that quantity of interest. Second, how to describe the \textit{fluctuations} around this deterministic limit? This is described by a stochastic PDE (SPDE). In the latter case,  typically ones assumes to start the system from a stationary state, since otherwise the analysis is usually extremely intricate. 

One of the most classical IPS is the exclusion process that we denote by $\eta_t$. The dynamics of this process can be defined as follows. At each site of the microscopic space one can have at most one particle (exclusion rule) so that if for a site $x\in\mathbb Z$ and a time $t$, the quantity $\eta_t(x)$ denotes the quantity of particles at site $x$ then $\eta_t(x)\in\{0,1\}$.  The dynamics is defined as follows: each particle waits a random time (which is exponentially distributed) after which it jumps to another position of the  discrete space according to some transition probability $p(\cdot)$. The process is said to be simple if jumps are restricted to nearest-neighbour sites. In this dynamics the number of particles is conserved and therefore it is the relevant quantity to  investigate.

In this article,  we analyse an exclusion process evolving on the discrete set $\mathbb{Z}$ and    particles jump according to the transition probability $p: \mathbb{Z} \mapsto [0,1]$, given  by
\begin{equation} \label{prob}
p(x) =
\begin{cases}
0, \; \; & x=0; \\
c_{\gamma} |x|^{-\gamma-1} , & x \neq 0.
\end{cases}
\end{equation}
Above $c_{\gamma}$ is a normalizing constant that turns $p(\cdot)$ into a probability measure. 
There are two important regimes to distinguish  the value of $\gamma$. When $\gamma\in[2,\infty)$ the behavior of the system is diffusive,  while for $\gamma\in(0,2)$, it is superdiffusive (in this work we only consider the former regime). This is a consequence of the fact that for $\gamma \in(2,\infty)$, the transition probability  $p(\cdot)$ has finite variance since  $\sum_{x\in\mathbb Z} x^2 p(x) < \infty$ and therefore, to observe a non-trivial evolution, we will speed up the process in the diffusive time scale $tn^2$. Note also that for $\gamma=2$, despite the variance of $p(\cdot)$ being infinite, by taking the time  scale $\frac{n^2}{\log(n)}$ we still see a diffusive behaviour.

On the exclusion dynamics defined above, we add a \textit{slow} barrier and the goal is to understand its macroscopic effect at the level of the evolution equations for the density of particles. 
To properly introduce the barrier, we denote the  set of bonds  by $\mcb B$, and we consider a set of slow bonds given by 
\begin{equation*}
    \mcb S \subset \mcb S_0:=\big\{\{x,y\} \in \mcb B: x \in\mathbb Z_-^*, y \in\mathbb N \big\},
\end{equation*}
and the complement of $\mcb S$ with respect to $\mcb B$ will be denoted by $\mcb F$, the set of  fast bonds. Now we fix two parameters $\alpha >0$ and $\beta \geq 0$. 
At a bond $\{x,y\}\in\mcb  F$,  particles swap positions  according to the transition probability $p(y-x)$, but when $\{x,y\}\in\mcb S$, then the jump rate becomes  equal to $\alpha n^{-\beta}p(y-x)$, see the figure below for a scheme of the dynamics.

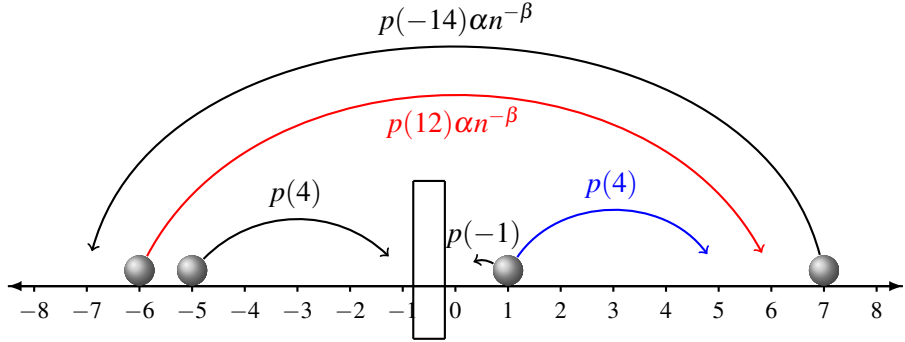
\begin{figure}[htb]
	\begin{center}
		\begin{tikzpicture}[thick, scale=0.7]
			\draw[latex-] (-8.5,0) -- (8.5,0) ;
			\draw[-latex] (-8.5,0) -- (8.5,0) ;
			\foreach \x in  {-8,-7,-6,-5,-4,-3,-2,-1,0,1,2,3,4,5,6,7,8}
			\draw[shift={(\x,0)},color=black] (0pt,0pt) -- (0pt,-3pt) node[below] 
			{};
			
  \draw[] (-8,-0.1) node[below] {\footnotesize{$-8$}};			
			\draw[] (-7,-0.1) node[below] {\footnotesize{$-7$}};
			\draw[] (-6,-0.1) node[below] {\footnotesize{$-6$}};
			\draw[] (-5,-0.1)  node[below] {\footnotesize{$-5$}};
			\draw[] (-4,-0.1)  node[below] {\footnotesize{$-4$}};
			\draw[] (-3,-0.1)  node[below] {\footnotesize{$-3$}};
			\draw[] (-2,-0.1)  node[below] {\footnotesize{$-2$}};
			\draw[] (-1,-0.1)  node[below] {\footnotesize{$-1$}};
			\draw[] (0,-0.1)  node[below] {\footnotesize{$0$}};
			\draw[] (1,-0.1)  node[below] {\footnotesize{$1$}};
			\draw[] (2,-0.1) node[below] {\footnotesize{$2$}};
			\draw[] (3,-0.1) node[below] {\footnotesize{$3$}};
			\draw[] (4,-0.1) node[below] {\footnotesize{$4$}};
			\draw[] (5,-0.1)  node[below] {\footnotesize{$5$}};
			\draw[] (6,-0.1)  node[below] {\footnotesize{$6$}};
			\draw[] (7,-0.1)  node[below] {\footnotesize{$7$}};
			\draw[] (8,-0.1)  node[below] {\footnotesize{$8$}};
			\draw [thick, black]  (-0.8,-1) -- (-0.8,2) ; 
			\draw [thick, black]  (-0.8,2) -- (-0.2,2) ;
			\draw [thick, black]  (-0.2,2) -- (-0.2,-1) ;
			\draw [thick, black]  (-0.2,-1) -- (-0.8,-1) ;

			\node[shape=circle,minimum size=0.5cm] (F) at (-7,0.3) {};
			\node[shape=circle,minimum size=0.5cm] (G) at (0,0.3) {};
			\node[shape=circle,minimum size=0.5cm] (H) at (6,0.3) {};
			\node[shape=circle,minimum size=0.5cm] (I) at (5,0.3) {};
			\node[shape=circle,minimum size=0.5cm] (J) at (-1,0.3) {};

			\node[ball color=black!30!, shape=circle, minimum size=0.4cm] (A) at (7,0.3) {};
			\node[ball color=black!30!, shape=circle, minimum size=0.4cm] (B) at (-6,0.3) {};
			\node[ball color=black!30!, shape=circle, minimum size=0.4cm] (C) at (-5.,0.3) {};
			\node[ball color=black!30!, shape=circle, minimum size=0.4cm] (D) at (1,0.3) {};

			\path [->] (A) edge[bend right =75] node[above] {\textcolor{black}{$p(-14) \alpha n^{-\beta}  $}} (F);           
			\path [->] [red]  (B) edge[bend left=62] node[below] {\textcolor{red}{$p(12) \alpha n^{-\beta}$}} (H);        
			\path [->] [blue] (D) edge[bend left=55] node[above]  {\textcolor{blue}{$p(4)$}} (I);
			\path [->] (C) edge[bend left=45] node[above]  {\textcolor{black}{$p(4) $}} (J);
			\path [->] (D) edge[bend right=25] node[above]  {\textcolor{black}{$ p(-1) $}} (G);		
		\end{tikzpicture}
		\bigskip
		\caption{Exclusion process with long-jumps and a slow barrier for $\mcb S = \mcb S_0$.}\label{figure6int}
	\end{center}
\end{figure}

In \cite{casodif} the hydrodynamic limit in the diffusive regime  was obtained. To properly state it, let $\mcb g: \mathbb{R} \rightarrow [0,1]$ be a measurable function, which corresponds to the initial condition of the PDE. Since the density is the relevant quantity to look, then we  consider  the empirical measure given  by \begin{equation}\label{eq:emp_mea_n}\pi^n_t(\eta,du):=\frac 1n\sum_x\eta_t(x)\delta_{\frac xn}(du).\end{equation} 
Let us now assume that the initial distribution of the process, denoted by $\mu_n$, is associated to $\mcb g(\cdot)$, i.e. the initial empirical measure $\pi_0^n(\eta,du)$ converges to the deterministic measure $\mcb g(u)du$ in probability with respect to $\mu_n$. This means that for every continuous function $G:\mathbb R\to\mathbb R $ with compact support and for every $\delta >0$, it holds  
\begin{equation*}
\lim_{n \rightarrow \infty} \mu_n \left( \eta: \Big| \langle \pi_0^n, G \rangle - \int_{\mathbb{R}} G(u) \mcb g(u) du \Big| > \delta \right) =0. 
\end{equation*}
The goal in the hydrodynamic limit consists in showing that the same result above holds at any time $t$, i.e. for any $t$, the sequence of random measures $\pi_t^n(\eta,du)$ converges, as $n \rightarrow \infty$, to the deterministic measure $\rho(t,u) du$ in probability with respect to the distribution of the system at time $t$, where $\rho(t,u)$ is the unique weak solution of  a PDE, the hydrodynamic equation. 
 The results of \cite{casodif} can be summarized as follows. 

\begin{itemize}

\item [I)] if particles can move between $\mathbb{Z}_{-}^{*}$ {and}  $\mathbb N$ using a path without slow bonds (this can only happen if $\mcb S \subsetneq \mcb S_0$) then one gets the heat equation as hydrodynamic equation:
\begin{equation} \label{eqhyddifreal}
\begin{cases}
\partial_t \varrho (t,u) = \frac{\sigma^2}{2} \Delta \varrho (t,u), (t,u) \in [0,T] \times \mathbb{R}, \\
\varrho(0,u) =\mcb g(u), u \in \mathbb{R};
\end{cases}
\end{equation}

\item [II)] if particles  cannot move between $\mathbb{Z}_{-}^{*}$  and $\mathbb N$ without using at least one slow bond, then the hydrodynamic equation depends on the value of $\beta$: 
\item for $0\leq \beta<1$,  it is the heat equation given in \eqref{eqhyddifreal};

\item  for $\beta=1$, it is the heat equation with Robin (linear) boundary conditions:
\begin{equation} \label{eqhyddifrob}
\begin{cases}
\partial_t \varrho(t,u) = \frac{\sigma^2}{2} \Delta \varrho(t,u), (t,u) \in [0,T] \times \mathbb{R}, \\
\partial_u \varrho(t,0^+) = \partial_u \varrho(t,0^-)= \kappa[ \varrho(t,0^+) - \varrho(t,0^{-})], t \in (0,T], \\
\varrho(0,u) =\mcb  g(u), u \in \mathbb{R}; 
\end{cases}
\end{equation}

\item  for $\beta>1$, it is the heat equation with  Neumann boundary conditions, i.e.  \eqref{eqhyddifrob} with $\kappa=0$.
\end{itemize}
We refer the reader to \cite{casodif} for the proper notion of weak solutions that are obtained. In \cite{superdif} it was analysed the  superdiffusive case, corresponding to the regime $\gamma\in(0,2)$. There, the equations obtained at the hydrodynamic limit have   the Laplacian operator replaced by the fractional Laplacian operator  or a regional fractional Laplacian operator defined on an unbounded domain. We refer the interested reader to \cite{superdif} for the exact statements.  

In this article we are now interested in investigating the fluctuations around the deterministic limits \eqref{eqhyddifreal} and \eqref{eqhyddifrob}, i.e. in the diffusive regime corresponding to $\gamma\in[2,\infty)$. In order to observe some dependence on $\beta$ and obtain analogous results to II) above, we assume that $\mcb S = \mcb S_0$, so exchanges of particles between $\mathbb{Z}_{-}^{*}$ and $\mathbb N$ always occur through slow bonds. Moreover,  we assume that our system starts from an invariant measure, which is the Bernoulli product measure with a 
constant parameter $b \in(0,1)$, that we denote by $\nu_b$. Therefore our quantity of interest is the density fluctuation field, which is given on a test function $f$  by 
\begin{equation*}
\mathcal{Y}_t^n (f) :=  \frac{1}{\sqrt{n}} \sum_{x} f ( \tfrac{x}{n} ) \big[ \eta_t^n(x)- b \big].
\end{equation*}

Observe that the field above is centred with respect to the invariant measure since $E_{\nu_b}[\eta_t(x)]=b$ for all $x\in\mathbb Z$ and $t>0$.  Typically, the density fluctuation field is described in terms of a solution to a SPDE, which depends strongly on the dynamics of the underlying microscopic system. To give some examples,  the generalized Ornstein-Uhlenbeck process has been obtained in \cite{tertufluc, flucstefano}; the stochastic Burgers equation has been obtained in \cite{goncalves2017stochastic}; and fractional versions  of those equations have been obtained in \cite{jarafluc}. In \cite{tertufluc} it was studied the same problem that we address here, but there the exclusion process was simple, since jumps were only allowed to nearest-neighbors and only a slow bond at $\{-1,0\}$ was present. Depending on the strength of the slow bond in \cite{tertufluc} it was derived an Ornstein-Uhlenbeck process with several boundary conditions. Our goal is to extend those results to the exclusion with  long-jumps  and with a slow barrier. 

In this work our space of test functions has to be carefully chosen, due to the slow barrier. Indeed, we need to impose conditions on these functions in order to control the evolution equations of the density field when the scaling parameter $n \rightarrow \infty$. On the other hand, we cannot impose too strong assumptions, since this would lead to a very small space of test functions, and we would not be able to obtain the uniqueness of the solutions for the aforementioned equations.

Our results can be stated as follows. When  the barrier is  not very slow (more precisely, when $0 \leq  \beta < 1$), we do not see any effect at the level of the fluctuations and they are described by a generalized Ornstein-Uhlenbeck process. For $\gamma >2$, when the intensity of the barrier attains a critical level (corresponding to $\beta=1$) we see an effect at the level of the Ornstein-Uhlenbeck process, which now has Robin  boundary  conditions; and those boundary conditions are replaced by Neumann boundary conditions when the barrier is very slow (corresponding to  $\beta>1$). This behaviour is reminiscent from the choice of the space of test functions where the fluctuation field acts; and it has similarities to the results that were obtained at the level of the  hydrodynamic limits. On the other hand, for $\gamma=2$,  there is no value of $\beta$ for which Robin boundary conditions are present, and we have Neumann boundary conditions for every $\beta \geq 1$.

  Now we give some comments about the proof. Since we are working with a model that is evolving in the full lattice  $\mathbb Z$, our arguments in the proofs become lengthy and technical, since we have to control many tails from several series that appear involving the transition probability $p(\cdot)$. Due to the fact that we have introduced a slow barrier, many boundary conditions have to be obtained from the microscopic level and this obliges us to use several replacement lemmas that allow closing the equations and take limits to reach the notions of solutions at the  macroscopic level. We therefore successfully extend the results of \cite{tertufluc} to the case when the exclusion process has long jumps and a slow barrier. Naturally, if we choose a transition probability which allows only jumps of size $1$ to the left and to the right with probability $1/2$, we recover the results of \cite{tertufluc}.

We leave some open problems for the future. The first one has to do with the understanding the limit of other observables of the process such as  the  fluctuations of either the current of the system or a tagged particle. In  \cite{tertufluc}, the aforementioned results were obtained as a consequence of the fluctuations of the density of particles. The strategy  was to use an identity relating the current of particles with the density fluctuation field evaluated  at a proper test function (more precisely, an Heaviside function). Then since the density fluctuations are known for a large collection of test functions, by an approximation argument, one can obtain the density fluctuations for the Heaviside function and from this the current fluctuations. For the tagged particle,  its position can be related with both the current and the density of particles, therefore the result for the tagged particle is a consequence of the corresponding results for the current and the density. We observe however that this procedure can only be used when all jumps are  nearest-neighbor  because it is crucial to have an order between particles. However, if jumps of size greater that $1$ are allowed, the order of the particles can be destroyed, so one requires an alternative strategy.  Another problem that can be addressed is the extension of our results to the superdiffusive  regime, i.e. for  $\gamma\in(0,2)$.  

This article is divided as  follows. In Section \ref{secstat}, we define precisely  our model and we state our main result, namely Theorem \ref{clt}. In Section \ref{tightfluc}, we prove tightness of the sequence of the density fluctuation fields. In Section \ref{characfluc}, we characterize the limit points of the aforementioned sequence in terms of a martingale problem. In Section \ref{bgibsgeral}, we obtain some useful $L^2$ estimates that  are needed to prove our main theorem. Finally, in the Appendix \ref{secdiscconv}, we show a multitude of results related to the convergence of discrete operators to continuous ones for each of the space of test functions. 

\section{Statement of results} \label{secstat}

\subsection{The model}\label{sec:model_hydro}

Our goal is to study the evolution of the exclusion process in $\mathbb{Z}$: this is an IPS which allows at most one particle per site, therefore our space state is $\Omega:=\{0,1\}^{\mathbb{Z}}$. The elements of $\mathbb{Z}$ are called \textit{sites} and will be denoted by Latin letters, such as $x$, $y$ and $z$. The elements of $\Omega$ are called \textit{configurations} and will be denoted by Greek letters, such as $\eta$. Moreover, we denote the number of particles at a site $x$ according to a configuration $\eta$ by $\eta(x)$; this means that the site $x$ is \textit{ empty } if $\eta(x)=0$ and it is \textit{ occupied } if $\eta(x)=1$. 

Recall the expression for $p$ given in \eqref{prob}. We denote $m:= \sum_{x \geq 1} xp(x) < \infty$. Moreover, in the case $\gamma >2$, we denote $\sigma^2:= \sum_{x} x^2 p(x) < \infty$.

According to the \textit{exclusion rule}, given $x_1,x_2 \in \mathbb{Z}$, a particle jumps from $x_1$ to $x_2$ if $x_2$ is empty and $x_1$ is occupied, i.e. $\eta(x_2)=0$ and $\eta(x_1)=1$. This means that a movement between $x_1$ and $x_2$ (in one of the two directions) is only possible if $\eta(x_1)[1-\eta(x_2)]+ \eta(x_2)[1-\eta(x_1)]=1.$

Therefore, the movement of a particle between $x$ and $y$ acts on an initial configuration $\eta \in \Omega$ and transforms it into $\eta^{x,y} \in \Omega$, which is defined by
\begin{equation*}
\eta^{x,y}(z)=
\begin{cases}
\eta(y), \quad & z= x; \\
\eta(x), \quad & z= y; \\
\eta(z),\quad & z \notin \{x, y\}.
\end{cases}
\end{equation*}
The elements of $\mcb {B}:=\big\{ \{x, y \}: x \neq y \in \mathbb{Z} \big\} $ are called \textit{bonds}. We denote $\mathbb{N}:=\{0, 1, 2, \ldots \}$ and $\mathbb{Z}_{-}^{*}:= \mathbb{Z} - \mathbb{N} = \{-1, -2, \ldots \}$. We also denote $\mcb{S}:=\big\{ \{x, y \} \in \mcb B: x \in \mathbb{Z}_{-}^{*}, y \in \mathbb{N}  \big\}$ and $\mcb{F}:= \mcb{B} - \mcb{S}$. 

Now we describe the effect of our slow barrier. Hereafter we fix $\alpha > 0$ and $\beta \geq 0$. Starting from a configuration $\eta$, a particle jumps from $x$ to $y$ with rate $\eta(x) [ 1 - \eta (y) ] r_{x,y}^n$, where 
\begin{equation}
r_{x,y}^n:=
\begin{cases}
\alpha n^{-\beta}, \quad & \{x, y\} \in \mcb{S}; \\
1, \quad & \{x, y\} \in \mcb{F}.
\end{cases}
\end{equation}
When $\beta >0$, we have $\lim_{n \rightarrow \infty} r_{x,y}^n=0$ for every $\{x,y\}$ in $\mcb{S}$, creating a physical barrier between $\mathbb{Z}_{-}^{*}$ and $\mathbb{N}$. Then $\mcb{S}$ and $\mcb {F}$ are the sets of slow and fast bonds, respectively.  

We say that a function $f: \Omega \mapsto \mathbb{R}$ is  \textit{local} if there exists a finite set $\Lambda \subset \mathbb{Z}$ such that
$
\forall x \in \Lambda, \eta_1(x) = \eta_2(x) \Rightarrow f(\eta_1) = f(\eta_2).
$
The exclusion process with slow bonds is described by the infinitesimal generator $ \mcb{L}_n$, defined by
\begin{align}\label{eq:generator}
( \mcb{L}_n f ) (\eta) =& \sum_{x,y} p( y - x) \eta(x) [1 - \eta (y) ] r_{x,y}^n \nabla_{x,y}f(\eta)   \\
= & \frac{1}{2} \sum_{x,y} p( y - x)   r_{x,y}^n \nabla_{x,y}f(\eta),
\end{align}
where $f$ is a local function and for $x,y\in\mathbb Z$ we have $\nabla_{x,y}f(\eta): = [ f ( \eta^{x,y} ) - f (\eta ) ]$. From now on, unless it is stated differently, we will assume that our discrete variables in a summation always range over $\mathbb{Z}$. 

\subsection{Notation}

If $G \in C^{\infty}(\mathbb{R})$ (i.e., $G: \mathbb{R} \rightarrow \mathbb{R}$ has continuous derivatives of all orders), $u \in \mathbb{R}$ and $k \geq 1$, we denote the $k$-th derivative of $G$ at $u$ by $G^{(k)}(u)$. For $k=0$, we denote $G^{(0)}(u)$ by $G(u)$. Following Section V.3 of \cite{ReedSimon}, the Schwartz space $\mathcal{S}( \mathbb{R})$ is the subspace of functions $G \in C^{\infty}(\mathbb{R})$ such that for every $(k,\ell) \in \mathbb{N}^2$, it holds
\begin{equation*} 
\| G \|_{k, \ell, \mathcal{S}( \mathbb{R})}:= \sup_{u \in \mathbb{R}}  |u^{k} G^{(\ell)} (u)| < \infty.
\end{equation*}
Moreover, we denote by $\mathcal{S}( \mathbb{R}^{*})$ the space of functions $G: \mathbb{R} \rightarrow \mathbb{R}$ such that there exist $G_{-},G_{+}$ in $\mathcal{S}( \mathbb{R})$ satisfying $G(u)= \mathbbm{1}_{ \{ u <0 \}}G_{-}(u) + \mathbbm{1}_{ \{ u \geq 0 \} }G_{+}(u)$, hence $G^{(k)}(u) \mathbbm{1}_{ \{ u \neq 0 \} } =\mathbbm{1}_{ \{ u < 0 \} } G_{-}^{(k)}(u) + \mathbbm{1}_{ \{ u > 0 \} } G_{+}^{(k)}(u) $ for every $k \geq 0$. By convention, we say that $G^{(k)}(0)=G_{+}^{(k)}(0)$ for every $k \geq 0$. We denote by $\mathcal{S}_{Neu}(\mathbb{R}^{*})$ the closed subspace of functions $G \in \mathcal{S}( \mathbb{R}^{*})$ such that
\begin{align*}
\forall k \geq 0, \quad G_{-}^{(2k+1)}(0) = G_{+}^{(2k+1)}(0) =0.
\end{align*}
For the critical case $\gamma>2$ and $\beta=1$, it will be useful to denote
\begin{equation} \label{alphahat}
\hat{\alpha}:= 2  \frac{\alpha m}{ \sigma^2} = \frac{\alpha m}{  \kappa_{\gamma}},
\end{equation}
where $\kappa_{\gamma}$ is given by
\begin{equation} \label{defkdifgamma}
\kappa_{\gamma}:=
\begin{cases}
c_2, & \gamma=2; \\
\sigma^2 , & \gamma > 2.
\end{cases}
\end{equation}
Moreover, we denote by $\mathcal{S}_{Rob}(\mathbb{R}^{*})$ the closed subspace of functions $G \in \mathcal{S}( \mathbb{R}^{*})$ such that
\begin{align*}
 \forall k \geq 0, \quad G_{-}^{(2k+1)}(0) = G_{+}^{(2k+1)}(0) = \hat{\alpha}\Big[G_{+}^{(2k)}(0) - G_{-}^{(2k)}(0)\Big].
\end{align*}
\begin{definition} \label{defsbetagamma}
For every $(\beta, \gamma) \in [0, \infty) \times [2, \infty)$, we denote our space of test functions by $\mathcal{S}_{\beta,\gamma}$, where $\mathcal{S}_{\beta,\gamma}:=\mathcal{S}( \mathbb{R})$ for $(\beta, \gamma) \in [0,1) \times [2, \infty)$, $\mathcal{S}_{\beta,\gamma}:=\mathcal{S}_{Rob}( \mathbb{R}^{*})$ for $(\beta, \gamma) \in \{1\} \times (2,\infty)$ and $\mathcal{S}_{\beta,\gamma}:=\mathcal{S}_{Neu}( \mathbb{R}^{*})$ otherwise. In particular, $\mathcal{S}_{\beta,\gamma}$ is always a closed subspace of $\mathcal{S}( \mathbb{R}^{*})$.

Furthermore, in order to state explicitly that certain results hold for all the possible values for $(\beta,\gamma)$, it will be convenient to define $R_0 :=[0, \infty) \times [2, \infty)$. 
\end{definition}
\begin{rem} \label{remnucfre}
First, we observe that $\mathcal{S}( \mathbb{R})$ is a nuclear Fr\'echet space; see Chapter 2, Section 4 (resp. Chapter 3, Section 7) of \cite{nuclear} for more details regarding the definition of Fr\'echet space (resp. nuclear space). 

From Chapter 2, Section 4 (resp. Chapter 3, Section 7) of \cite{nuclear} we have that the every countable product of Fr\'echet spaces is also a Fr\'echet space (resp. every arbitrary product of nuclear spaces is also a nuclear space). Thus, since $\mathcal{S}( \mathbb{R}^{*})$ can be identified with $\mathcal{S}( \mathbb{R}) \times \mathcal{S}( \mathbb{R})$, we have that $\mathcal{S}( \mathbb{R}^{*})$ is a nuclear Fr\'echet space.

Finally, from Chapter 2, Section 4 (resp. Chapter 3, Section 7) of \cite{nuclear} we have that the every closed subspace of a Fr\'echet space is also a Fr\'echet space (resp. every arbitrary subspace of a nuclear space is also a nuclear space). Combining this with Definition \ref{defsbetagamma}, we have that $\mathcal{S}_{\beta,\gamma}$ is a nuclear Fr\'echet space, for every $(\beta, \gamma) \in R_0$.
\end{rem}
For every $(\beta, \gamma) \in R_0$ we define the operator $\Delta_{\beta, \gamma}: \mathcal{S}_{\beta, \gamma} \rightarrow \mathcal{S}(\mathbb{R}^{*})$ by
\begin{equation} \label{defdeltafluc}
(\Delta_{\beta, \gamma} G) (u) = 
\begin{cases}
G_{-}^{(2)}(u), & u<0; \\
G_{+}^{(2)}(u), & u \geq 0;
\end{cases}
\end{equation}
and the operator $\nabla_{\beta, \gamma}: \mathcal{S}_{\beta, \gamma} \rightarrow \mathcal{S}(\mathbb{R}^{*})$ by
\begin{equation} \label{defnablafluc}
(\nabla_{\beta, \gamma} G) (u) = 
\begin{cases}
G_{-}^{(1)}(u),   & u<0; \\
G_{+}^{(1)}(u), & u \geq 0.
\end{cases}
\end{equation}
In particular, given $(\beta, \gamma) \in R_0$ it holds $\Delta_{\beta, \gamma} G  \in \mathcal{S}_{\beta,\gamma}$ for every $G  \in \mathcal{S}_{\beta,\gamma}$. For every function $G \in \mathcal{S}_{\beta,\gamma}$, we define $\| \cdot \|_{2,\beta,\gamma}$ by
\begin{equation*}
\| G \|^2_{2,\beta,\gamma}:=2 \kappa_{\gamma} \Big(  \int_{\mathbb{R}} [G(u)]^2 du  +  \frac{ 1} {\hat{\alpha}}  \mathbbm{1}_{\{ \gamma > 2 \} } \mathbbm{1}_{\{ \beta = 1 \} }  [  G(0)]^2 \Big),
\end{equation*}
where $\kappa_{\gamma}$ is given in \eqref{defkdifgamma}. For every $t>0$ and $x \in \mathbb{R}$, let $\phi_t(x): = \frac{e^{-\frac{x^2}{4t}}}{\sqrt{4 \pi t}}$ be the heat kernel. Then $\phi_t \in \mathcal{S}(\mathbb{R})$ for every $t>0$.

For every $ t \geq 0$, we define the following semigroups:
\begin{enumerate}
\item
\begin{align*}
P_t g(x) :=
\begin{cases}
g(x), & t=0; \\
  (\phi_t * g)(x),  &  t>0.
 \end{cases}
\end{align*}
\item
\begin{align*}
P_t^{Neu} g (x):= 
\begin{cases}
g(x), & (t,x) \in \{0\} \times \mathbb{R}; \\
\frac{1}{\sqrt{4 \pi t}} \int_0^{\infty} \big[ e^{- \frac{(x-y)^2}{4t}} + e^{- \frac{(x+y)^2}{4t}}  ] g(y) dy, &  (t,x) \in (0, \infty) \times [0, \infty); \\
\frac{1}{\sqrt{4 \pi t}} \int_0^{\infty} \big[ e^{- \frac{(x-y)^2}{4t}} + e^{- \frac{(x+y)^2}{4t}}  ] g(-y) dy,   & (t,x) \in (0, \infty) \times(- \infty,0). 
\end{cases}
\end{align*}
\item
\begin{align*}
P_t^{Dir} g(x):= 
\begin{cases}
g(x), &  (t,x) \in \{0\} \times  \mathbb{R}; \\
 \int_0^{\infty} [ \phi_t(x-y) - \phi_t(x+y)] g(y) dy, & (t,x) \in (0, \infty) \times \mathbb{R}.
 \end{cases}
\end{align*}
\item
\begin{align*}
P_t^{ \hat{\alpha} } g(x): = e^{2  \hat{\alpha}  x} \int_x^{\infty} e^{-2  \hat{\alpha}  y} P_t^{Dir} ( 2  \hat{\alpha}  g - g^{(1)}) (y) dy, \quad  (t,x) \in [0, \infty) \times \mathbb{R}.
\end{align*}
\item
\begin{align*}
P_t^{Rob} g  (x):= 
\begin{cases}
 P_t ( \mathcal{E} g)(x) + \tilde{P}_t^{\hat{\alpha}} ( \mathcal{O} g)(x),  & (t,x) \in [0, \infty) \times [0, \infty); \\
 - P_t ( \mathcal{E} g)(-x) - \tilde{P}_t^{\hat{\alpha}} ( \mathcal{O} g)(x-), & (t,x) \in [0, \infty) \times (- \infty,0),
\end{cases}
\end{align*}
where for any function $g: \mathbb{R} \rightarrow \mathbb{R}$, $\mathcal{E} g: \mathbb{R} \rightarrow \mathbb{R}$ and $\mathcal{O} g: \mathbb{R} \rightarrow \mathbb{R}$ are defined by
\begin{align*}
 (\mathcal{E} g ) (x) = \frac{g(x)+g(-x)}{2},  \quad \quad   (\mathcal{O} g ) (x) = \frac{g(x)-g(-x)}{2},
\end{align*}
for every $x \in \mathbb{R}$.
\end{enumerate}
Finally, for every $(\beta, \gamma) \in R_0$, we define the semigroup $P_t^{\beta,\gamma}$ acting on $\mathcal{S}_{\beta,\gamma}$ by $P_t^{\beta,\gamma}: = P_t$ for $(\beta, \gamma) \in [0,1) \times [2, \infty)$, $P_t^{\beta,\gamma}: = P_t^{Rob}$ for $(\beta,\gamma) \in \{1 \} \times (2,\infty)$ and $P_t^{\beta}: = P_t^{Neu}$ otherwise. From \cite{corrigendum}, we know that given $t \geq 0$ and $(\beta, \gamma) \in [0, \infty) \times [2, \infty)$, it holds $P_t^{\beta,\gamma} G  \in \mathcal{S}_{\beta,\gamma}$ for every $G \in \mathcal{S}_{\beta,\gamma}$.

Moreover, for every $p \in [1, \infty]$ let $L^p(\mathbb{R}):=L^p(\mathbb{R}, \mu)$, where $\mu$ is the Lebesgue measure in $\mathbb{R}$. We also denote the norm in $L^p(\mathbb{R})$ by $\| \cdot \|_{p}$ for $p< \infty$ and by $\| \cdot \|_{\infty}$ for $p = \infty$.

In order to observe a non-trivial macroscopic limit, our Markov process will be accelerated in time by a factor of $\Theta(n)$, given by
\begin{equation} \label{timescale}
\Theta(n):=
\begin{cases}
n^2, & \gamma > 2; \\
\frac{n^2}{  \log(n)}, & \gamma = 2.
\end{cases}
\end{equation}
We study the process $\eta_t^n(\cdot):=\eta_{t \Theta(n)}(\cdot)$, whose infinitesimal generator is $\Theta(n) \mcb {L}_n$. We fix an arbitrary $T>0$, which leads to a finite time horizon $[0,T]$. In particular, $(\eta^n_t)_{t \in [0,T]} \in \mcb {D} [0,T], \Omega )$, the space of the c\`adl\`ag (right-continuous and with left limits) trajectories in $\Omega$. The goal of this work is to study the fluctuations of the \textit{empirical measure}, denoted by $\pi_t^n$ and defined by \eqref{eq:emp_mea_n}. We denote the integral of a function $ G \in L^1(\mathbb{R})$  with respect to the empirical measure by $\langle \pi_t^n, G \rangle$,  so that 
\begin{align} \label{intmedemp}
\forall t \in [0,T], \quad \langle \pi_t^n, G \rangle: = \frac{1}{n} \sum_x \eta_t^n(x) G (\tfrac{x}{n} ).
\end{align} 
Next we introduce a measure on $\Omega$ which will be relevant in this work. Hereinafter, we fix $b \in (0,1)$ and denote by $\nu_b$ the Bernoulli product measure with marginals given by 
\begin{align} \label{defberprod}
\forall x \in  \mathbb{Z}, \quad  1 - \nu_b \big( \eta: \eta(x) = 0 \big) = \nu_b \big(\eta: \eta(x) = 1 \big) = b.
\end{align}
Since $\nu_b(\eta^{x,y})=\nu_b(\eta)$ for every $\eta \in \Omega$ and every $x,y \in \mathbb{Z}$, we conclude that $\nu_b$ is reversible (and in particular invariant) with respect to $\mcb L_n$. In particular,
\begin{equation} \label{invmeas}
\forall t \in [0,T], \forall x \in \mathbb{Z} \quad  E_{\nu_b} [ \eta_t^n(x) ]= b \quad \textrm{and} \quad  E_{\nu_b} \big[\big( \eta_t^n(x) - b \big)^2 \big]= \chi(b):=b(1-b).
\end{equation}
Moreover, the random variables $\big((\eta_t^n(x)\big)_{x \in \mathbb{Z}}$ are independent under $\nu_b$ for every $t \in [0,T]$.

\subsection{Main result}

Hereinafter, we fix $b \in (0,1)$. Before stating the main result of this article, we characterize the generalized Ornstein-Uhlenbeck process, which is a solution of
\begin{equation} \label{spde}
d \mathcal{Y}_t = \kappa_{\gamma} \Delta_{\beta,\gamma} \mathcal{Y}_t dt + \sqrt{2 \chi(b)} \nabla_{\beta,\gamma} d \mathcal{W}_t,
\end{equation}
 in terms of a martingale problem. Above $\mathcal{W}_t$ is a space-time white noise of variance $1$. In spite of having a dependence of $\mathcal{Y}_t$ on $\beta$ and $\gamma$, in order to keep notation simpler, we do not index on these parameters.

In what follows, $\mathcal{S}_{\beta,\gamma}'$ denotes the space of bounded linear functionals $f: \mathcal{S}_{\beta,\gamma} \rightarrow \mathbb{R}$. Moreover, $ \mcb{D} ( [0,T], \mathcal{S}_{\beta,\gamma}')$ (respectively $ \mcb{C} ( [0,T],    \mathcal{S}_{\beta,\gamma}')$  is the space of c\`adl\`ag functions (respectively continuous) $\mathcal{S}_{\beta,\gamma}-'$ valued functions endowed with the Skorokod topology.

The proof of the  next result is presented in Section 5 of \cite{tertufluc} and Section 4 of \cite{corrigendum}, therefore we do not prove it here and we refer the interested reader to those articles.
\begin{prop} \label{uniqou}
There exists a unique random element $\mathcal{Y}$ taking values in $ \mcb{C} ( [0,T], \mathcal{S}_{\beta,\gamma}')$ such that:
\begin{itemize}
\item
For every function $G \in \mathcal{S}_{\beta,\gamma}$, the processes $\mathcal{M}_t(G)$ and $\mathcal{N}_t(G)$ given by
\begin{equation} \label{martMou}
\mathcal{M}_t(G) = \mathcal{Y}_t(G) - \mathcal{Y}_0(G) - \kappa_{\gamma} \int_0^{t} \mathcal{Y}_s( \Delta_{\beta,\gamma} G)ds
\end{equation}
and
\begin{equation} \label{martNou}
\mathcal{N}_t(G) =[ \mathcal{M}_t(G)]^2  - 2 \chi(b) t   \| \nabla_{ \beta, \gamma} G \|_{2, \beta, \gamma}^2 
\end{equation}
are $\mathcal{F}_t$-martingales, where for each $t \in [0,T]$, $\mathcal{F}_{t}:= \sigma ( \mathcal{Y}_s (G): (s,G) \in [0,t] \times \mathcal{S}_{\beta,\gamma})$.
\item
$\mathcal{Y}_0$ is a mean zero Gaussian field with covariance given on $G_1,G_2 \in \mathcal{S}_{\beta,\gamma}$ by
\begin{equation} \label{covinifie}
\mathbb{E}_{\nu_{b}} [ \mathcal{Y}_0(G_1) \mathcal{Y}_0(G_2) ] = \chi(b) \int_{\mathbb{R}}  G_1(u) G_2(u) du. 
\end{equation}
Moreover for every $G \in \mathcal{S}_{\beta,\gamma}$, the stochastic process $\{ \mathcal{Y}_t (G); t \geq 0 \}$ is Gaussian, being the distribution of $\mathcal{Y}_t(G)$ conditionally to $\mathcal{F}_{s}$, for $s<t$, normal of mean $\mathcal{Y}_{s} ( P_{t-s}^{\beta,\gamma} G)$ and variance $\int_0^{t-s} \| \nabla_{\beta,\gamma} P_r^{\beta,\gamma} G \|_{2, \beta, \gamma}^2 dr$.  
\end{itemize}
\end{prop}
The random element $\mathcal{Y}$ is called the generalized Ornstein-Uhlenbeck process of characteristics $\Delta_{\beta,\gamma}$ and $\nabla_{\beta,\gamma}$, see \eqref{defdeltafluc} and \eqref{defnablafluc}. From \eqref{martNou} and Levy's Theorem on the martingale characterization of Brownian motion, the process
\begin{equation*}
\mathcal M_t (G) \sqrt{2 \chi(b)}  \| \nabla_{\beta,\gamma}  G \|_{2, \beta, \gamma}
\end{equation*} 
is a standard Brownian motion, for every $G \in \mathcal{S}_{\beta,\gamma}$ fixed. Therefore, in view of Proposition \ref{uniqou}, it makes sense to say that $\mathcal{Y}$ is the formal solution of \eqref{spde}. 

Recall the definition of the $\nu_{b}$ in \eqref{defberprod}. In order to state the Central Limit Theorem (CLT) for the empirical measure $\pi_t^n$ defined in \eqref{eq:emp_mea_n}, we introduce the \textit{density fluctuation field} $\mathcal{Y}_t^n$. It is defined as the $\mathcal{S}_{\beta,\gamma}'$-valued process $\{ \mathcal{Y}_t^n: t \in [0,T] \}$ given on $f \in \mathcal{S}_{\beta, \gamma}$ by
\begin{equation*}
\mathcal{Y}_t^n (f) :=  \frac{1}{\sqrt{n}} \sum_{x} f ( \tfrac{x}{n} ) \bar{\eta}_t^n(x),
\end{equation*}
for any $t \in [0,T]$ and any $n \geq 1$. Above $\bar{\eta}_t^n(x):= \eta_t^n(x) - \mathbb{E}_{\nu_b}[\eta_t^n(x) ] = \eta_t^n(x)-b$.

Finally we state the main result of this article.
\begin{thm} \textbf{(CLT for the density of particles)} \label{clt}
Assume  $(\beta, \gamma ) \in R_0$. Consider the Markov process $\{ \eta_{t}^n : t \in [0,T] \}:=\{ \eta_{t \Theta(n)} : t \in [0,T] \}$  with generator given by \eqref{eq:generator} and with $\Theta(n)$  given by \eqref{timescale}. Suppose that it  starts  from the invariant state $\nu_{b}$. Then, the sequence of processes $( \mathcal{Y}_t^n)_{n \geq 1}$ converges in distribution, as $n \rightarrow \infty$, with respect to the Skorohod topology of $ \mcb{D} ( [0,T],  \mathcal{S}_{\beta,\gamma}' )$ to $\mathcal{Y}_t$ in $ \mcb{C} ( [0,T],  \mathcal{S}_{\beta,\gamma}' )$, the generalized Ornstein-Uhlenbeck process $\mathcal{Y}_{t}$ defined in Proposition \ref{uniqou}.
\end{thm}
We denote by $\mathbb{P} _{\nu_b}$ be the probability measure on $\mcb D ([0,T],\Omega)$ induced by the Markov process $\{\eta_{t}^{n};{t\geq 0}\}$ and the initial distribution $\nu_{b}$. Furthermore, we denote the expectation with respect to $\mathbb{P}_{\nu_b}$ by $\mathbb{E}_{\nu_b}$. In Section \ref{tightfluc}, we prove tightness of $( \mathcal{Y}_t^n)_{n \geq 1}$, i.e., $( \mathcal{Y}_t^n)_{n \geq 1}$ converges (weakly) along subsequences. In Section \ref{characfluc}, we characterize the limit points of $( \mathcal{Y}_t^n)_{n \geq 1}$ as random elements satisfying the three conditions stated in Proposition \ref{uniqou}. The uniqueness established in Proposition \ref{uniqou} leads to the desired result. Finally, some $L^2(\mathbb{P}_{\nu_b})$ estimates that are used in the article are proved in Section \ref{bgibsgeral}.

\section{Tightness of the  sequence of fluctuation fields} \label{tightfluc}

From Dynkin's formula (Appendix 1.5 of \cite{kipnis1998scaling}), for every $G \in \mathcal{S}_{\beta,\gamma}$, the process $\{ \mathcal{M}_t^n(G); t \in [0,T] \}$ defined by
\begin{equation} \label{defMtngfluc}
 \mathcal{M}_t^n(G) := \mathcal{Y}_t^n (G) - \mathcal{Y}_0^n (G) - \int_{0}^t  \Theta(n)  \mcb{L}_n  \mathcal{Y}_s^n (G) ds
\end{equation}
is a martingale  with respect to the natural filtration $\mcb F_t^n = \sigma ( \eta_s^n: 0 \leq s \leq t)$, whose quadratic variation $\langle \mathcal M^n (G)  \rangle_t$ is given by
\begin{equation} \label{quadvarMtnG}
\begin{split}
\langle \mathcal M^n (G)  \rangle_t :=& \int_0^t \frac{\Theta(n)}{n}   \sum_{ \{x,y\} \in \mcb F} p(x-y) [G( \tfrac{y}{n}) - G( \tfrac{x}{n})]^2[ \eta_s^n(y) - \eta_s^n(x)]^2 ds   \\
+&\int_0^t  \alpha \Theta(n) n^{-1-\beta}     \sum_{  \{x,y\} \in \mcb S  } p(x-y) [G( \tfrac{y}{n}) - G( \tfrac{x}{n})]^2  [ \eta_s^n(y) - \eta_s^n(x)]^2 ds. 
\end{split}
\end{equation}
Thus, $\mathcal{N}_t^n(G):=[\mathcal{M}_t^n(G)]^2 - \langle \mathcal M^n (G)  \rangle_t$ is also a martingale with respect to $\mcb F_t^n$. The rightmost term on the right-hand side of \eqref{defMtngfluc} will be called the \textit{integral term} and it will be denoted by $\mathcal{I}_t^{n}$. Hence from \eqref{defMtngfluc}, we get
\begin{align*}
\mathcal{Y}_t^n (G)  = \mathcal{Y}_0^n (G) + \mathcal{M}_t^n(G)  + \mathcal{I}_t^n(G).
\end{align*}
Therefore, the tightness of $( \mathcal{Y}_t^n)_{n \geq 1}$ is a consequence of the tightness of the sequence of initial fields $\big( \mathcal{Y}_0^n (G) \big)_{n \geq 1}$, of the sequence of the martingale terms $ \big( \mathcal{M}_t^n(G) \big)_{n \geq 1}$ and of the sequence of integral terms $\big(\mathcal{I}_t^n(G) \big)_{n \geq 1}$. From Remark \ref{remnucfre}, we can use the Mitoma's criterion (see \cite{mitoma}) in the same way as it was done in \cite{tertufluc, corrigendum} since for every $(\beta,\gamma) \in R_0$, our space of test functions  $\mathcal{S}_{\beta,\gamma}$,  are nuclear and  Fr\'echet spaces, see Remark \ref{remnucfre}.

\begin{prop} \textbf{Mitoma's criterion}. Let $N$ be a nuclear Fr\'echet space. A sequence $\{ x_t^n; t \in [0,T] \}_{n \geq 1}$ in $\mcb{D} ([0,T], N' )$ of stochastic processes is tight with respect to the Skorohod topology if, and only if, the sequence of real-valued processes $\{ x_t^n (G); t \in [0,T] \}_{n \geq 1}$  is tight with respect to the Skorohod topology of $\mcb{D} ([0,T], \mathbb{R})$, for every $G \in  N$ fixed.
\end{prop}

\subsection{Tightness of the initial field}

Next result is analogous to Proposition 3.2 of \cite{tertufluc}, therefore its proof will be omitted.
\begin{prop} \textbf{(Convergence of the initial field)} \label{tightfluc1}
Let $(\beta, \gamma) \in R_0$. For every $G \in \mathcal{S}_{\beta,\gamma}$, for every $t>0$ and every $\lambda \in \mathbb{R}$, it holds
\begin{align*}
\lim_{n \rightarrow \infty}  \mathbb{E}_{\nu_b} [ exp \{  i\lambda  \mathcal{Y}_t^n (G) \} ] = \exp \Big\{ - \frac{ \lambda^2 \chi(b)}{2} \int_{\mathbb{R}} G^2(u) du  \Big\}.
\end{align*}
In particular, the sequence $\big( \mathcal{Y}_t^n (G) \big)_{n \geq 1}$ converges in distribution to a mean zero Gaussian variable with variance $\lambda^2 \chi(b) \| G \|_{2,\mathbb{R}}^2 $ and it is tight. Finally, $( \mathcal{Y}_0^n)_{n \geq 1}$ converges in distribution to $ \mathcal{Y}_0$, where $\mathcal{Y}_0$ is a mean zero Gaussian field with covariance given by \eqref{covinifie}.
\end{prop}
\begin{rem} 
The result above is true for $\mathcal{Y}_t$ for any $t \in [0,T]$. In particular, the Gaussian white noise is a stationary solution of \eqref{spde}, for any $(\beta, \gamma) \in R_0$. 
\end{rem}

\subsection{Tightness of the martingale}

Our goal now is to prove the following result.
\begin{prop} \label{tightmartterm}
Let $(\beta, \gamma) \in R_0$. Then the sequence $\{ \mathcal{M}_t^n(G); t \in [0,T] \}_{n \geq 1}$ is tight with respect to the Skorohod topology of $\mcb{D}( [0,T], \mathbb{R})$, for every $G \in \mathcal{S}_{\beta,\gamma}$ and every $t >0$.
\end{prop}
We will use Aldous' criterion to prove last result.

\begin{prop} (Aldous' criterion) \label{tightmarterm0}\\
The sequence $\{ x_t^n; t \in [0,T] \}_{n \geq 1}$ is tight with respect to the Skorohod topology of $\mcb{D}( [0,T], \mathbb{R})$ if
\begin{enumerate}
\item
$\lim_{A \rightarrow \infty} \limsup_{n \rightarrow \infty}  \mathbb{P}_{\nu_b} \big( \sup_{t \in [0,T]} |x_t^n | > A \big)=0;$
\item
For every $\varepsilon >0$, we have
\begin{align*}
\lim _{\omega \rightarrow 0^+} \limsup_{n \rightarrow\infty} \sup_{\tau  \in \mathcal{T}_{T},\bar\tau \leq \omega} \mathbb{P}_{\nu_b} \big(  |x_{\tau+ \bar\tau}^n- x_{\tau}^n  |> \ve \big)  =0,
\end{align*}
where $\mathcal{T}_T$ denotes the set of stopping times bounded by $T$.
\end{enumerate}
\end{prop}
We apply now last proposition to the sequence $(\mathcal{M}_t^n(G))_{n \geq 1}$ of martingales. We begin with the next lemma.
\begin{lem} \label{lemconvmartterm1}
Let $(\beta, \gamma) \in R_0$, $t \in [0,T]$ and $G \in \mathcal{S}_{\beta,\gamma}$. Then
\begin{align*}
\lim_{n \rightarrow \infty} \mathbb{E}_{\nu_b} [ \langle \mathcal{M}^n(G)  \rangle_t ]  =  2 \chi(b) t    \| \nabla_{\beta,\gamma} G \|^2_{2,\beta,\gamma}.
\end{align*}
\end{lem}
\begin{proof}
For every $s \geq 0$, for every $x \neq y \in \mathbb{Z}$, we have $\mathbb{E}_{\nu_b} \big[ [ \eta_s^n(y) - \eta_s^n(x)]^2 \big] = 2 \chi(b)$.
Combining Fubini's Theorem and \eqref{quadvarMtnG} we get
\begin{align*} 
\mathbb{E}_{\nu_b} [\langle \mathcal M^n (G)  \rangle_t ] = &2 \chi(b) t   \frac{\Theta(n)}{n}    \sum_{   \{x,y\} \in \mcb F } p(y-x) [G( \tfrac{y}{n}) - G( \tfrac{x}{n})]^2 \\
+& 2 \chi(b) t  \alpha \frac{\Theta(n)}{n^{1+\beta}} \sum_{   \{x,y\} \in \mcb S } p(y-x) [G( \tfrac{y}{n}) - G( \tfrac{x}{n})]^2 \\
=& 2  \chi(b) t [ \mathcal{A}_{n,\beta} (G) + \mathcal{B}_{n,\beta} (G) ],
\end{align*}
where  
\begin{equation}\label{op_Anb}
\mathcal{A}_{n,\beta} (G): =
\begin{dcases}
\frac{\Theta(n)}{n}   \sum_{x,y } p(y-x) [G( \tfrac{y}{n}) - G( \tfrac{x}{n})]^2, & \beta \in [0,1); \\
 \frac{\Theta(n)}{n} \sum_{   \{x,y\} \in \mcb F } p(y-x) [G( \tfrac{y}{n}) - G( \tfrac{x}{n})]^2, & \beta \geq 1; \\
\end{dcases}
\end{equation}
and
\begin{equation}\label{op_Bnb}
\mathcal{B}_{n,\beta} (G): =
\begin{dcases}
 ( \alpha n^{-\beta} -1) \frac{\Theta(n)}{n} \sum_{   \{x,y\} \in \mcb S } p(y-x) [G( \tfrac{y}{n}) - G( \tfrac{x}{n})]^2 , &  \beta \in [0,1); \\
\alpha \Theta(n) n^{-1-\beta} \sum_{   \{x,y\} \in \mcb S }    p(y-x) [G( \tfrac{y}{n}) - G( \tfrac{x}{n})]^2,  & \beta \geq 1.
\end{dcases}
\end{equation}
Therefore the proof ends as a consequence of Propositions \ref{prop2lem1convmart} and \ref{prop1lem1convmart}.
\end{proof}
\begin{lem} \label{lemconvmartterm2}
Let $(\beta, \gamma) \in R_0$ and $G \in \mathcal{S}_{\beta,\gamma}$. Then
\begin{align} \label{expconvmartterm2}
 \sup_{t \in [0,T]} \mathbb{E}_{\nu_b} \Big[ \Big(  \langle \mathcal M^n (G)  \rangle_t - \mathbb{E}_{\nu_b} [   \langle \mathcal M^n (G)  \rangle_t ] \Big)^2 \Big] \lesssim 1. 
\end{align}
\end{lem} 
\begin{proof}
Combining Fubini's Theorem and \eqref{quadvarMtnG} we obtain
\begin{align*}
  \langle \mathcal M^n (G)  \rangle_t - \mathbb{E}_{\nu_b} [  \langle \mathcal M^n (G)  \rangle_t ]  =& \int_0^t \frac{\Theta(n)}{n} \sum_{ \{x, y\} \in \mcb F } [ G( \tfrac{y}{n} ) - G( \tfrac{x}{n} )]^2 p(y-x)Z_{x,y}(\eta_s^n) ds \\
  +& \alpha \int_0^t \frac{\Theta(n)}{n^{1+\beta}} \sum_{ \{x, y\} \in \mcb S } [ G( \tfrac{y}{n} ) - G( \tfrac{x}{n} )]^2 p(y-x) Z_{x,y}(\eta_s^n) ds, 
\end{align*}
where $Z_{x,y}(\eta_s^n):=[ \eta_s^n(y) - \eta_s^n(x)]^2 - 2 \chi(b)$. Since $(u+v)^2 \leq 2 (u^2+v^2)$ for every $u,v \in \mathbb{R}$,  the expectation in \eqref{expconvmartterm2} is bounded from above by a constant times the sum of
\begin{equation} \label{expconvmartterm2a}
[ \Theta(n) ]^2 n^{-2} \mathbb{E}_{\nu_b} \Big[  \Big(  \int_0^t  \sum_{ \{x, y\} \in \mcb F } [ G( \tfrac{y}{n} ) - G( \tfrac{x}{n} )]^2 p(y-x) Z_{x,y}(\eta_s^n) ds \Big)^2 \Big] 
\end{equation}
and
\begin{equation} \label{expconvmartterm2b}
[ \Theta(n) ]^2 n^{-2-2 \beta} \mathbb{E}_{\nu_b} \Big[  \Big(  \int_0^t  \sum_{ \{x, y\} \in \mcb S } [ G( \tfrac{y}{n} ) - G( \tfrac{x}{n} )]^2 p(y-x) Z_{x,y}(\eta_s^n) ds \Big)^2 \Big]. 
\end{equation}
From the fact that $\mathbb{E}_{\nu_b} [  Z_{x,y}(\eta_s^n)]=0=\mathbb{E}_{\nu_b} [  Z_{x,y}(\eta_s^n) Z_{w,r}(\eta_s^n)]=0$ for disjoint sets $\{x,y\}$, $\{w,r\}$, together with Cauchy-Schwarz inequality and Fubini's Theorem, the expectation in \eqref{expconvmartterm2a} is bounded from above by
\begin{align*}
& t \int_0^t \mathbb{E}_{\nu_b} \Big[ \Big( \sum_{ \{x, y\} \in \mcb F } [ G( \tfrac{y}{n} ) - G( \tfrac{x}{n} )]^2 p(y-x) Z_{x,y}(\eta_s^n) \Big)^2 \Big] ds \\
\lesssim&  \int_0^t  \sum_{ \{x, y\} \in \mcb F } [ G( \tfrac{y}{n} ) - G( \tfrac{x}{n} )]^4 [p(y-x)]^2 \mathbb{E}_{\nu_b} \big[ \big( Z_{x,y}(\eta_s^n) \big)^2 \big] ds \\
+&   \int_0^t  \sum_{ \{x, y\} \in \mcb F } \sum_{w: \{x,w\} \in \mcb F} [ G( \tfrac{y}{n} ) - G( \tfrac{x}{n} )]^2 [ G( \tfrac{w}{n} ) - G( \tfrac{x}{n} )]^2 p(y-x) p(x-w)  \mathbb{E}_{\nu_b} \big[  Z_{x,y}(\eta_s^n) Z_{x,w}(\eta_s^n)  \big] ds.
\end{align*}
Since $\mathbb{E}_{\nu_b} \big[ \big( Z_{x,y}(\eta_s^n) \big)^2 \big] $ and $\mathbb{E}_{\nu_b} \big[  Z_{x,y}(\eta_s^n) Z_{x,w}(\eta_s^n)  \big]$ are constant and finite, last display is bounded from above by a constant times
\begin{align} \label{expconvmartterm2c}
& [ \Theta(n) ]^2 n^{-2} \Big\{ \sum_{ \{x, y\} \in \mcb F } [ G( \tfrac{y}{n} ) - G( \tfrac{x}{n} )]^4 p^2(y-x) +  \sum_{ x } \Big[ \sum_{y: \{x,y\} \in \mcb F } p(y-x) [ G( \tfrac{y}{n} ) - G( \tfrac{x}{n} )]^2 \Big]^2 \Big\}.
\end{align}
Performing an analogous procedure, \eqref{expconvmartterm2b} is bounded from above by a constant times
\begin{align} \label{expconvmartterm2d}
& [ \Theta(n) ]^2 n^{-2-2\beta} \Big\{ \sum_{ \{x, y\} \in \mcb S } [ G( \tfrac{y}{n} ) - G( \tfrac{x}{n} )]^4 p^2(y-x) +  \sum_{ x } \Big[ \sum_{y: \{x,y\} \in \mcb S } p(y-x) [ G( \tfrac{y}{n} ) - G( \tfrac{x}{n} )]^2 \Big]^2 \Big\}.
\end{align}
If $\beta \in [0,1)$, then $G \in \mathcal{S}_{\beta, \gamma} = \mathcal{S}(\mathbb{R})$, and we bound the sum of \eqref{expconvmartterm2c} and \eqref{expconvmartterm2d} from above by
\begin{align*}
[ \Theta(n) ]^2 n^{-2} \Big\{ \sum_{ x,y } [ G( \tfrac{y}{n} ) - G( \tfrac{x}{n} )]^4 p^2(y-x) +  \sum_{ x } \Big[ \sum_{y  } p(y-x) [ G( \tfrac{y}{n} ) - G( \tfrac{x}{n} )]^2 \Big]^2 \Big\},
\end{align*}
hence the proof ends in this regime as a consequence of Proposition \ref{proplemconvmart2}. 

When $\beta \geq 1$ and $\gamma \geq 2$ we get $[ \Theta(n) ]^2 n^{-2-2\beta} \lesssim 1$, hence \eqref{expconvmartterm2d} is bounded from above by
\begin{align*}
&[ \Theta(n) ]^2 n^{-2-2\beta} (2 \| G \|_{\infty})^{4} \Big\{ \sum_{ \{x, y\} \in \mcb S }  p^2(y-x) +  \sum_{ x } \Big[ \sum_{y: \{x,y\} \in \mcb S } p(y-x)  \Big]^2 \Big\} \\
\lesssim & [ \Theta(n) ]^2 n^{-2-2\beta} \Big\{ \sum_{ \{x, y\} \in \mcb S }  p(y-x) +  \sum_{ x }  \sum_{y: \{x,y\} \in \mcb S } p(y-x)  \Big\} = 2m [ \Theta(n) ]^2 n^{-2-2\beta} \lesssim 1.
\end{align*}
Finally, we bound \eqref{expconvmartterm2c} from above by
\begin{align*}
&[ \Theta(n) ]^2 n^{-2} \Big\{ \sum_{ x,y } [ G_{-}( \tfrac{y}{n} ) - G_{-}( \tfrac{x}{n} )]^4 p^2(y-x) +  \sum_{ x } \Big[ \sum_{y  } p(y-x) [ G_{-}( \tfrac{y}{n} ) - G_{-}( \tfrac{x}{n} )]^2 \Big]^2 \Big\} \\
+ & [ \Theta(n) ]^2 n^{-2} \Big\{ \sum_{ x,y } [ G_{+}( \tfrac{y}{n} ) - G_{+}( \tfrac{x}{n} )]^4 p^2(y-x) +  \sum_{ x } \Big[ \sum_{y  } p(y-x) [ G_{+}( \tfrac{y}{n} ) - G_{+}( \tfrac{x}{n} )]^2 \Big]^2 \Big\}.
\end{align*}
In last display, $G_{-}, G_{+} \in \mathcal{S}(\mathbb{R})$ are such that $G \equiv G_{-}$ on $(-\infty,0)$ and $G \equiv G_{+}$ on $[0, \infty)$. Then the desired result comes in this regime from the application of Proposition \ref{proplemconvmart2} to $G_{-}$ and $G_{+}$.
\end{proof}
A direct consequence of Lemma \ref{lemconvmartterm1} is that
\begin{equation} \label{M0thnL2}
\lim_{n \rightarrow \infty} \mathbb{E}_{\nu_b} \big[ \big( \mathcal{M}_t^n(G) \big)^2 \big] =  \lim_{n \rightarrow \infty} \mathbb{E}_{\nu_b} [ \langle  \mathcal{M}^n(G)\rangle_t ] =  2 \chi(b) t  \kappa_{\gamma}  \| \nabla_{\beta,\gamma} G \|^2_{2,\beta,\gamma} < \infty.
\end{equation} 
In particular, for every $t \in [0,T]$ and $G \in \mathcal{S}_{\beta,\gamma}$, we have
\begin{equation} \label{MthnL2}
\sup_{t \in [0,T]} \mathbb{E}_{\nu_b} \big[ \big( \mathcal{M}_t^n(G) \big)^2 \big] = \sup_{t \in [0,T]} \mathbb{E}_{\nu_b} [  \langle  \mathcal{M}^n(G)\rangle_t  ]  \lesssim 1.
\end{equation}
Finally, we will prove Proposition \ref{tightmartterm} as a consequence of Proposition \ref{tightmarterm0}. 
\begin{proof} [Proof of Proposition \ref{tightmartterm}]
From Doob's inequality and \eqref{MthnL2}, the sequence of martingales $\{ \mathcal{M}_t^n(G); t \in [0,T] \}_{n \geq 1}$ satisfies the first condition of Proposition \ref{tightmarterm0}. In order to verify the second one, we use Chebyshev's inequality, which leads to
\begin{align*}
& \mathbb{P}_{\nu_b} \big(  |\mathcal{M}_{\tau+ \bar\tau}^n(G)- \mathcal{M}_{\tau}^n(G)  |> \ve \big) \leq \frac{1}{\varepsilon^2} \mathbb{E}_{\nu_b} \big[  \big(\mathcal{M}_{\tau+ \bar\tau}^n(G)- \mathcal{M}_{\tau}^n(G)  \big)^2 \big] \\
=& \frac{1}{\varepsilon^2} \mathbb{E}_{\nu_b} \Big[ \int_{\tau}^{\tau+ \bar\tau}  \langle  \mathcal{M}^n(G)\rangle_s ds  \Big] \leq \frac{1}{\varepsilon^2}  \int_{\tau}^{\tau+ \bar\tau}  \sup_{t \in [0,T]} \mathbb{E}_{\nu_b} [   \langle  \mathcal{M}^n(G) \rangle_s   ] ds \lesssim \frac{\bar\tau}{\varepsilon^2} \leq \frac{\omega}{\varepsilon^2},
\end{align*}
which goes to zero as $\omega \rightarrow 0$.  Above we used Fubini's Theorem and \eqref{MthnL2}.
\end{proof}

\subsection{Tightness of the integral term}

Our goal is to prove the following result.
\begin{prop} \label{tightintterm}
Let $(\beta, \gamma) \in R_0$. Then the sequence $\{ \mathcal{I}_t^n(G); t \in [0,T] \}_{n \geq 1}$  is tight with respect to the Skorohod topology of $\mcb{D}( [0,T], \mathbb{R})$, for every $G \in \mathcal{S}_{\beta,\gamma}$.
\end{prop}
By performing some algebraic manipulations, it holds $\mathcal{I}_t^n(G)=\mathcal{C}_t^n(G)+\mathcal{E}_t^n(G)$, where 
\begin{align}
\mathcal{C}_t^n(G): =&  \int_{0}^{t}  \frac{1}{\sqrt{n}} \sum _{x }  \kappa_{\gamma} \Delta_{\beta,\gamma} G\left( \tfrac{x}{n} \right)  \bar{\eta}_{s}^{n}(x) ds = \kappa_{\gamma} \int_0^t  \mathcal{Y}_s^n ( \Delta_{\beta,\gamma}   G )ds; \label{princnbfluc} \\
 \mathcal{E}_t^n(G):=&  \int_{0}^{t} \frac{1}{\sqrt{n}} \sum _{x } \big[ \Theta(n)  \mcb {R}_{n,\beta} G \left( \tfrac{x}{n} \right) + \Theta(n)  \mcb {K}_{n,\beta} G \left( \tfrac{x}{n} \right) - \kappa_{\gamma} \Delta_{\beta,\gamma} G\left( \tfrac{x}{n} \right) \big]   \bar{\eta}_{s}^{n}(x) ds. \label{extranbfluc}  
\end{align}
Above, $\mcb {K}_{n,\beta}$ and $\mcb {R}_{n,\beta}$ are defined by
\begin{equation}\label{op_Knb}
\mcb{K}_{n,\beta} G \left( \tfrac{x}{n} \right): =
\begin{dcases}
 \sum_{y} \left[ G( \tfrac{y}{n}) -G( \tfrac{x}{n}) \right] p(y-x) = \sum_{r } [ G( \tfrac{x+r}{n}) -G( \tfrac{x}{n}) ] p(r),  & \beta \in [0,1); \\
 \sum_{y=0}^{\infty} \big[ G( \tfrac{y}{n} )  -  G( \tfrac{x}{n} ) - n^{-1} G_{+}^{'}(0) (y-x) \big]p(y-x), & \beta \geq 1, x \geq 0; \\
  \sum_{y=-\infty}^{-1} \big[ G( \tfrac{y}{n} )  -  G( \tfrac{x}{n} ) - n^{-1} G_{-}^{'}(0) (y-x) \big]p(y-x), & \beta \geq 1, x \leq -1;
\end{dcases}
\end{equation}
and
\begin{equation}\label{op_Rnb}
\mcb{R}_{n,\beta} G \left( \tfrac{x}{n} \right): =
\begin{dcases}
 (1 -  \alpha n^{-\beta} ) \sum_{y: \{x, y\} \in \mcb S} \left[ G( \tfrac{y}{n}) -G( \tfrac{x}{n}) \right] p(y-x), & \beta \in [0,1); \\
\frac{\alpha}{n^{\beta}}   \sum_{y=-\infty}^{-1} \left[ G( \tfrac{y}{n}) -G( \tfrac{x}{n}) \right] p(y-x) + \frac{G_{+}^{'}(0)}{n}  \sum_{y=0}^{\infty}(y - x)  p(y-x),  & \beta \geq 1, x \geq 0; \\
\frac{\alpha}{n^{\beta}} \sum_{y=0}^{\infty} \left[ G( \tfrac{y}{n}) -G( \tfrac{x}{n}) \right] p(y-x) + \frac{G_{-}^{'}(0)}{n}  \sum_{y=-\infty}^{-1}(y - x)  p(y-x), & \beta \geq 1, x \leq -1.
\end{dcases}
\end{equation}
Therefore we are done if we can prove the next result. 
\begin{prop}  \label{tightfluc2}
The sequences $ \big\{ \mathcal{C}_t^n(G) ; t \in [0,T] \big\}_{n \geq 1}$ and $ \big\{ \mathcal{E}_t^n(G) ; t \in [0,T] \big\}_{n \geq 1}$ are tight, for every $(\beta, \gamma) \in R_0$ and every $G \in \mathcal{S}_{\beta,\gamma}$.
\end{prop}
In Section \ref{bgibsgeral} we will prove the next result, which is useful to treat $ \big\{ \mathcal{E}_t^n(G) ; t \in [0,T] \big\}_{n \geq 1}$.
\begin{prop} \label{convrem}
Let $(\beta,\gamma) \in R_0$. For every $G \in \mathcal{S}_{\beta,\gamma}$, it holds
\begin{align*} 
\lim_{n \rightarrow \infty} \mathbb{E}_{\nu_b} \Big[ \sup_{t \in [0,T]} \big(\mathcal{E}_t^n(G) \big)^2 \Big] =0. 
\end{align*}
\end{prop}
Assuming last result, we now show Proposition \eqref{tightfluc2}.
\begin{proof}[Proof of Proposition \ref{tightfluc2}]
Let $G \in \mathcal{S}_{\beta,\gamma}$. In order to treat $ \big\{ \mathcal{C}_t^n(G) ; t \in [0,T] \big\}_{n \geq 1}$, we follow the strategy of \cite{jarafluc}. Since $\mathcal{C}_0^n(G)=0$, from Kolmogorov-Centsov's criterion (see Proposition 4.3 of \cite{jarafluc}), it is enough to verify that there exists $K >0$ such that
\begin{align*}
\forall r,t \in [0,T], \quad \mathbb{E}_{\nu_b} [  |  \mathcal{C}_t^n(G) -  \mathcal{C}_r^n(G) |^2   ] \leq K |t-r|^2.
\end{align*} 
Without loss of generality, we can assume that $r \leq t$. We observe that
\begin{align*} 
K_1:=& \sup_{t \in [0,T], n \geq 1} \mathbb{E}_{\nu_b} \big[  \big( \mathcal{Y}_t^n ( \Delta_{\beta,\gamma}  G ) \big)^2 \big]  = \sup_{ n \geq 1} \Big\{ \frac{\chi(b)}{n} \sum_{x} [\Delta_{\beta,\gamma}  G (\tfrac{x}{n})]^2  \Big\}< \infty.
\end{align*}
Therefore, Cauchy-Schwarz's inequality and Fubini's Theorem lead to
\begin{align*}
&\mathbb{E}_{\nu_b} [  |  \mathcal{C}_t^n(G) -  \mathcal{C}_r^n(G) |^2   ] = \mathbb{E}_{\nu_b}\Big[  \Big| \kappa_{\gamma} \int_r^t  \mathcal{Y}_s^n ( \Delta_{\beta,\gamma}   G )ds \Big|^2  \Big] \\ 
\leq & (\kappa_{\gamma})^2 (t-r)     \int_r^t \sup_{t \in [0,T], n \geq 1} \mathbb{E}_{\nu_b} \big[  \big( \mathcal{Y}_t^n ( \Delta_{\beta,\gamma}  G ) \big)^2 \big] ds  \leq K_1  (\kappa_{\gamma})^2 (t-r)^2,
\end{align*}
leading to the tightness of $ \big\{  \mathcal{C}_t^n(G) ; t \in [0,T] \big\}_{n \geq 1}$. It remains to analyse $ \big\{  \mathcal{E}_t^n(G) ; t \in [0,T] \big\}_{n \geq 1}$ for $G \in \mathcal{S}_{\beta,\gamma}$. We do so by applying Aldous' criterion. From Chebyshev's inequality and Proposition \ref{convrem}, we get
\begin{align*}
\limsup_{n \rightarrow \infty} \mathbb{P}_{\nu_b} \big( \sup_{t \in [0,T]} | \mathcal{E}_t^n(G) | > A \big) \leq & \limsup_{n \rightarrow \infty} \frac{1}{A^2} \mathbb{E}_{\nu_b} \big[ \big( \sup_{t \in [0,T]} | \mathcal{E}_t^n(G)| \big)^2 \big] \\
=& \limsup_{n \rightarrow \infty} \frac{1}{A^2} \mathbb{E}_{\nu_b} \big[  \sup_{t \in [0,T]} \big( \mathcal{E}_t^n(G) \big)^2 \big] =0 
\end{align*}
for every $A >0$, hence the first condition of Proposition \ref{tightmarterm0} is satisfied. Another application of Chebyshev's inequality and Proposition \ref{convrem} leads to
\begin{align*}
\limsup_{n \rightarrow\infty} \sup_{\tau  \in \mathcal{T}_{T},\bar\tau \leq \omega} \mathbb{P}_{\nu_b} \big(  |\mathcal{E}_{\tau+ \bar\tau}^n(G) - \mathcal{E}_{\tau}^n(G)  |> \ve \big) \leq& \frac{1}{\ve^2} \limsup_{n \rightarrow\infty} \sup_{\tau  \in \mathcal{T}_{T},\bar\tau \leq \omega} \mathbb{E}_{\nu_b} \big[  |\mathcal{E}_{\tau+ \bar\tau}^n(G) - \mathcal{E}_{\tau}^n(G)  |^2  \big] \\
\leq& \frac{2}{\ve^2} \limsup_{n \rightarrow\infty} \sup_{\tau  \in \mathcal{T}_{T},\bar\tau \leq \omega} \mathbb{E}_{\nu_b} \big[  \big(\mathcal{E}_{\tau+ \bar\tau}^n(G) \big)^2+ \big( \mathcal{E}_{\tau}^n(G) \big)^2  \big] \\
 \leq & \frac{4}{\ve^2} \limsup_{n \rightarrow\infty} \sup_{\tau  \in \mathcal{T}_{T},\bar\tau \leq \omega} \mathbb{E}_{\nu_b} \big[  \sup_{t \in [0,T]} \big( \mathcal{E}_t^n(G) \big)^2 \big] =0
 \end{align*}
for every $\omega, \varepsilon >0$. Therefore the second condition of Proposition \ref{tightmarterm0} is also satisfied and we conclude that $ \big\{ \mathcal{E}_t^n(G) ; t \in [0,T] \big\}_{n \geq 1}$ is tight.
\end{proof}

\section{Characterization of limit points} \label{characfluc}

From Propositions \ref{tightfluc1}, \ref{tightmartterm} and \ref{tightfluc2}, we know that there exists at least a subsequence $(n_j)_{j \geq 1}$ such that $(\mathcal{Y}_0^{n_j})_{j \geq 1}$, $(\mathcal{M}_t^{n_j})_{j \geq 1}$ and $(\mathcal{I}_t^{n_j})_{j \geq 1}$ converge in distribution to $\mathcal{Y}_0$, $\mathcal{M}_t$ and $\mathcal{I}_t$, respectively; in particular,  $(\mathcal{Y}_t^{n_j})_{j \geq 1}$ converges in distribution to some limit point $\mathcal{Y}_t$. In this section we will characterize such limit point.

Fix $(\beta, \gamma) \in R_0$ and $G \in \mathcal{S}_{\beta,\gamma}$. From \eqref{MthnL2}, $\mathcal{M}_t(G)$ is the limit in distribution of the uniformly integrable sequence $\big(\mathcal{M}_t^{n_j}(G) \big)_{j \geq 1}$ of martingales; in particular, $\mathcal{M}_t(G)$ is also a martingale.

Since $\mathcal{N}_t^{n_j}(G)=[\mathcal{M}_t^{n_j}(G)]^2 - \langle \mathcal M^{n_j} (G)  \rangle_t$ for every $t \in [0,T]$ and $j \geq 1$, from Lemma \ref{lemconvmartterm1} we get that $\mathcal{N}_t^n(G)$ converges in distribution to
\begin{align} \label{defNtG}
\mathcal{N}_t(G):= [ \mathcal{M}_t(G)]^2 - 2 \chi(b) t \kappa_{\gamma}  \| \nabla_{\beta,\gamma} G  \|_{2, \beta, \gamma}.
\end{align}
We claim that $\mathcal{N}_t(G)$ is  \textcolor{magenta}{ an} $\mathcal{F}_t$-martingale. Since it is the limit in distribution of the sequence of martingales $\big( \mathcal{N}_t^n(G) \big)_{n \geq 1}$, it is enough to prove the uniform integrability of the sequence. A sufficient condition is
\begin{align*}
\sup_{t \in [0,T], n \geq 1} \mathbb{E}_{\nu_b} \Big[ \big( \mathcal{N}_t^n(G) \big)^2 \Big]  = \sup_{t \in [0,T], n \geq 1} \mathbb{E}_{\nu_b} \Big[ \big( [ \mathcal{M}_t^n(G)]^2 + \langle  \mathcal{M}^n(G) \rangle_t  \big)^2 \Big]  < \infty. 
\end{align*}
Since $(u+v)^2 \leq 2 (u^2 + v^2)$, last display is bounded from above by a constant times
\begin{equation} \label{sumsupr}
\sup_{t \in [0,T], n \geq 1} \mathbb{E}_{\nu_b} \big[ \big(   \langle  \mathcal{M}^n(G) \rangle_t  \big)^2 \big]  + \sup_{t \in [0,T], n \geq 1} \mathbb{E}_n \big[  [ \mathcal{M}_t^n(G)]^4 \big] 
\end{equation}
From Lemmas \ref{lemconvmartterm1} and \ref{lemconvmartterm2}, we get
\begin{align*}
\limsup_{n \rightarrow \infty} \mathbb{E}_{\nu_b} \Big[ \big(   \langle  \mathcal{M}^n(G) \rangle_t  \big)^2 \big] = & \lim_{n \rightarrow \infty} \big( \mathbb{E}_{\nu_b} [ \langle \mathcal{M}^n(G)  \rangle_t ] \big)^2 + \limsup_{n \rightarrow \infty} \mathbb{E}_{\nu_b} \Big[  \Big(  \langle \mathcal{M}^n (G)  \rangle_t - \mathbb{E}_{\nu_b} [  \langle \mathcal{M}^n (G)  \rangle_t ] \Big)^2 \Big] \\
\lesssim& \big[ 2 \chi(b) t \kappa_{\gamma} \| \nabla_{\beta,\gamma} G \|_{2, \beta, \gamma} \big]^2 + 1 \lesssim 1,
\end{align*}
and then the first supremum in \eqref{sumsupr} is finite. From the definition of \eqref{defMtngfluc} we observe that
\begin{equation*} 
|\mathcal M_{t-}^n(G) -\mathcal M_t^n(G)| = | \mathcal{Y}_{t-}^n (G)  -  \mathcal{Y}_t^n (G)  | = \Big| \frac{1}{\sqrt{n}} \sum_{x } G( \tfrac{x}{n} )[ \eta_{t-}^n (x)   - \eta_{t}^n (x) ] \Big| \leq  \frac{2 \|G \|_{\infty} }{\sqrt{n}}.
\end{equation*}
The inequality holds since in an infinitesimal interval of time at most one jump occurs. Now we treat the second supremum by using Lemma 3 of \cite{Dittrich91}, from where we know that there exists $C>0$ such that
\begin{align*}
\mathbb{E}_{\nu_b} \big[  [ \mathcal{M}_t^n(G)]^4 \big]  \leq C \Big( \mathbb{E}_{\nu_b} \big[  [ \mathcal{M}_t^n(G)]^2 \big] +   \mathbb{E}_{\nu_b} \big[ \sup_{t \in [0,T]} |\mathcal M_{t-}^n(G) -\mathcal M_t^n(G)|^{4} \big] \Big) \\
\leq C  \mathbb{E}_{\nu_b} [  \langle \mathcal{M}^n(G) \rangle_t] + C  \mathbb{E}_{\nu_b} \Big[ \Big(  \frac{2 \|G \|_{\infty} }{\sqrt{n}} \Big)^{4} \Big] = C 2 b(1-b) t \kappa_{\gamma} \| \nabla_{\beta,\gamma} G \|_{2, \mathbb{R}}  + \frac{16 C \|G \|_{\infty}^4}{n^2},
\end{align*}
and then the second supremum in \eqref{sumsupr} is also finite. Therefore, $\mathcal{N}_t(G)$ is indeed  \textcolor{magenta}{ an} $\mathcal{F}_t$-martingale. 

Furthermore, combining \eqref{defMtngfluc}, \eqref{princnbfluc}, \eqref{extranbfluc} and Proposition \ref{convrem}, for every $j \geq 1$ it holds
\begin{align*}
\mathcal{M}_t^n(G) := \mathcal{Y}_t^n (G) - \mathcal{Y}_0^n (G) - \kappa_{\gamma} \int_0^t  \mathcal{Y}_s^n ( \Delta_{\beta,\gamma}   G )ds, 
\end{align*}
plus a term that goes to zero in $L^2(\mathbb{P}_{\nu_b})$, as $j \rightarrow \infty$. Therefore, we conclude that $\mathcal{Y}$ satisfies the conditions stated in Proposition \ref{uniqou}. Therefore, in order to finish the proof of Theorem \ref{clt}, it only remains to show Proposition \ref{convrem}, which is done in next section. 

\section{Useful $L^2(\mathbb{P}_{\nu_b})$ estimates} \label{bgibsgeral}

From the definition of $\mathcal{E}_t^n(G)$ given in \eqref{extranbfluc}, Proposition \ref{convrem} is a direct consequence of next two results.
\begin{prop} \label{convrem1}
Let $(\beta,\gamma) \in R_0$. For every $G \in \mathcal{S}_{\beta,\gamma}$, it holds
\begin{align} \label{eqconvrem1}
\lim_{n \rightarrow \infty} \mathbb{E}_{\nu_b} \Big[ \sup_{t \in [0,T]} \Big( \int_{0}^{t} \frac{1}{\sqrt{n}} \sum _{x } \big[  \Theta(n)  \mcb {K}_{n,\beta} G \left( \tfrac{x}{n} \right) - \kappa_{\gamma} \Delta_{\beta,\gamma} G\left( \tfrac{x}{n} \right) \big]   \bar{\eta}_{s}^{n}(x) ds \Big)^2 \Big] =0. 
\end{align}
\end{prop}
\begin{prop} \label{convrem2}
Let $(\beta,\gamma) \in R_0$. For every $G \in \mathcal{S}_{\beta,\gamma}$, it holds
\begin{align*} 
\lim_{n \rightarrow \infty} \mathbb{E}_{\nu_b} \Big[ \sup_{t \in [0,T]} \Big(\int_{0}^{t} \frac{1}{\sqrt{n}} \sum _{x } [ \Theta(n)  \mcb {R}_{n,\beta} G \left( \tfrac{x}{n} \right)   ]   \bar{\eta}_{s}^{n}(x) ds \Big)^2 \Big] =0. 
\end{align*}
\end{prop}
We begin by proving the former result.
\begin{proof}[Proof of Proposition \ref{convrem1}]
From Cauchy-Schwarz's inequality and Fubini's Theorem, the expectation in \eqref{eqconvrem1} is bounded from above by
\begin{align*}
 &  T \int_{0}^{T} \mathbb{E}_{\nu_b} \Big[ \Big( \frac{1}{\sqrt{n}} \sum _{x } \big[  \Theta(n)  \mcb {K}_{n,\beta} G \left( \tfrac{x}{n} \right) - \kappa_{\gamma} \Delta_{\beta,\gamma} G\left( \tfrac{x}{n} \right) \big]   \bar{\eta}_{s}^{n}(x)  \Big)^2 \Big] ds   \\
&T  \int_{0}^{T} \frac{1}{n} \sum _{x } \big[  \Theta(n)  \mcb {K}_{n,\beta} G \left( \tfrac{x}{n} \right) - \kappa_{\gamma} \Delta_{\beta,\gamma} G\left( \tfrac{x}{n} \right) \big]^2 \mathbb{E}_{\nu_b} \big[ \big(   \bar{\eta}_{s}^{n}(x) \big)^2 \big] ds \\
\leq& 2 \chi(b) T^2 \sup_{n \geq 1, x \in \mathbb{Z}} |\Theta(n)  \mcb {K}_{n,\beta} G \left( \tfrac{x}{n} \right) - \kappa_{\gamma} \Delta_{\beta,\gamma} G\left( \tfrac{x}{n} \right)  | \frac{1}{n} \sum _{x } |  \Theta(n)  \mcb {K}_{n,\beta} G \left( \tfrac{x}{n} \right) - \kappa_{\gamma} \Delta_{\beta,\gamma} G\left( \tfrac{x}{n} \right) |,
\end{align*}
which goes to zero as $n \rightarrow \infty$, from Proposition \ref{convknbeta}. 
\end{proof}
To prove Proposition \ref{convrem2}, we treat two cases separately: $(\beta, \gamma) \in \{1\} \times (2, \infty)$, presented in Subsection \ref{bgibbscrit} ; and $(\beta, \gamma) \in R_1: = R_0 - \{1\} \times (2, \infty)$, presented in Subsection \ref{bgibbs}.

\subsection{Case $(\beta, \gamma) \in \{1\} \times (2, \infty)$} \label{bgibbscrit}

In this case, Proposition \eqref{convrem2} is a direct consequence of Propositions \ref{prop1bgcrit}, \ref{prop2bgcrit} and \ref{prop3bgcrit}.  
\begin{prop} \label{prop1bgcrit}
Let $(\beta,\gamma) \in \{1\} \times (2,\infty)$ and $G \in \mathcal{S}_{\beta,\gamma}$. Then
\begin{equation}   \label{bgcritpos1}
  \lim_{n \rightarrow \infty} \mathbb{E}_{\nu_b} \Big[ \sup_{t \in [0,T]} \Big(    \int_0^t \frac{\Theta(n)}{\sqrt{n}} \sum_{x=0}^{\infty} \mcb{R}_{n,\beta} G \left( \tfrac{x}{n} \right) \bar{\eta}_s(0) ds  \Big)^2 \Big] =0,
\end{equation}
\begin{equation}   \label{bgcritneg1}
  \lim_{n \rightarrow \infty} \mathbb{E}_{\nu_b} \Big[ \sup_{t \in [0,T]} \Big(    \int_0^t \frac{\Theta(n)}{\sqrt{n}} \sum_{x=-\infty}^{-1} \mcb{R}_{n,\beta} G \left( \tfrac{x}{n} \right) \bar{\eta}_s(-1) ds  \Big)^2 \Big] =0.
\end{equation}
\end{prop}
\begin{proof}
We will prove only \eqref{bgcritpos1} but the proof of \eqref{bgcritneg1} is analogous. It is enough to prove that
\begin{align} \label{suprob0}
\lim_{n \rightarrow \infty} \sup_{s \in [0,T]} \Big| \frac{\Theta(n)}{\sqrt{n}} \sum_{x=0}^{\infty} \mcb{R}_{n,\beta} G \left( \tfrac{x}{n} \right) \bar{\eta}_s(0)  \Big| =0.
\end{align}
 Using the symmetry of $p(\cdot)$ and the hypothesis that $(\beta,\gamma) \in \{1\} \times (2,\infty)$, we can rewrite the expression inside the absolute value in \eqref{suprob0} as
\begin{align}
n^{3/2}   \sum_{x=0}^{\infty} \mcb{R}_{n,\beta} G \left( \tfrac{x}{n} \right) \bar{\eta}_s(0)  =& \sqrt{n} \bar{\eta}_s(0) \Big[ \alpha \sum_{x=0}^{\infty}  \sum_{y=-\infty}^{-1} \left[ G( \tfrac{y}{n}) -G( \tfrac{x}{n}) \right] p(y-x) + G_{+}^{'}(0) \sum_{x=0}^{\infty} \sum_{r=x+1}^{\infty} r  p(r) \Big] \nonumber \\
=& \sqrt{n} \bar{\eta}_s(0) \Big[ \alpha \left[ G_{-}( 0) -G_{+}( 0) \right]  \sum_{x=0}^{\infty}  \sum_{y=-\infty}^{-1}  p(y-x) + G_{+}^{'}(0) \sum_{x=0}^{\infty} \sum_{r=x+1}^{\infty} r  p(r) \Big] \label{eqprop1bgcrita} \\
+& \sqrt{n} \bar{\eta}_s(0) \Big[ \alpha \sum_{x=0}^{\infty}  \sum_{y=-\infty}^{-1} \big( \left[ G_{-}( \tfrac{y}{n}) -G_{-}( 0) \right]  + [G_{+}( 0)  - G_{+}( \tfrac{x}{n})   ] \big) p(y-x) \Big]. \label{eqprop1bgcritb}
\end{align}
Since $(\beta,\gamma) \in \{1\} \times (2,\infty)$, $G \in \mathcal{S}_{Rob}(\mathbb{R}^{*})$ and we can rewrite \eqref{eqprop1bgcrita} as
\begin{align*}
\sqrt{n} \bar{\eta}_s(0) G_{+}^{'}(0) \Big[ -\frac{\alpha}{\hat{\alpha}}    \sum_{x=0}^{\infty}  \sum_{y=-\infty}^{-1}  p(y-x) +  \sum_{x=0}^{\infty} \sum_{r=x+1}^{\infty} r  p(r) \Big] = \sqrt{n} \bar{\eta}_s(0) G_{+}^{'}(0) \Big[ -\frac{\alpha}{\hat{\alpha}} m + \frac{\sigma^2}{2} \Big] =0,
\end{align*}
hence we need only to treat \eqref{eqprop1bgcritb}. In last equality we used the definition of $\hat{\alpha}$, given in \eqref{alphahat}. Using the fact that $|\bar{\eta}^n_s(\cdot)| \leq 1$ and Taylor expansions of first order on $G_{-}$ and $G_{+}$, the expression inside the limit in  \eqref{suprob0} is bounded from above by
\begin{align*}
\frac{1}{\sqrt{n}}  \Big[ \alpha \sum_{x=0}^{\infty}  \sum_{y=-\infty}^{-1} ( \| G_{-} \|_{\infty} |y| +\| G_{+} \|_{\infty} x  ) p(y-x) \Big] \lesssim \frac{1}{\sqrt{n}} \sum_{x=0}^{\infty}  \sum_{y=-\infty}^{-1} (x-y) p (x-y) = \frac{1}{\sqrt{n}} \frac{\sigma^2}{2},
\end{align*}
which ends the proof.
\end{proof}
\begin{prop} \label{prop2bgcrit}
Let $(\beta,\gamma) \in (1/2, \infty) \times [2,\infty)$ and $G \in \mathcal{S}(\mathbb{R}^{*})$. Then 
\begin{equation}   \label{bgcritpos2}
  \lim_{n \rightarrow \infty} \mathbb{E}_{\nu_b} \Big[\sup_{t \in [0,T]} \Big(    \int_0^t \frac{\Theta(n)}{\sqrt{n}}  \sum_{x=0}^{\infty} \alpha n^{-\beta} \sum_{y=-\infty}^{-1} [ G \left( \tfrac{y}{n} \right) - G \left( \tfrac{x}{n} \right) ] p(y-x)  [ \bar{\eta}_s(x) - \bar{\eta}_s(0)] ds  \Big)^2 \Big] =0,
\end{equation}
\begin{equation}   \label{bgcritneg2}
  \lim_{n \rightarrow \infty} \mathbb{E}_{\nu_b} \Big[\sup_{t \in [0,T]} \Big(    \int_0^t \frac{\Theta(n)}{\sqrt{n}} \sum_{x=-\infty}^{-1} \alpha n^{-\beta} \sum_{y=-\infty}^{-1} [ G \left( \tfrac{y}{n} \right) - G \left( \tfrac{x}{n} \right) ] p(y-x)  [ \bar{\eta}_s(x) - \bar{\eta}_s(-1)] ds  \Big)^2 \Big] =0.
\end{equation}
\end{prop}
\begin{proof}
We will prove only \eqref{bgcritpos2} but we the proof of \eqref{bgcritneg2} is analogous. From Lemma 4.3 of \cite{CLO}, the expectation in \eqref{bgcritpos2} is bounded from above by a constant times
\begin{align} \label{sup01crit}
 \sup_{g \in L^2 (  \nu_{b})} \Big\{  \frac{\Theta(n)}{n^{\frac{1}{2}+\beta}}  \sum_{x=0}^{\infty}  \sum_{y=-\infty}^{-1} [ G \left( \tfrac{y}{n} \right) - G \left( \tfrac{x}{n} \right) ] p(y-x) \langle \bar{\eta}(x) - \bar{\eta}(0) , g \rangle_{\nu_{b}}   +\Theta(n) \langle \mcb {L}_n g, g \rangle_{\nu_{b}}   \Big\},
\end{align}
Simple computations show that for $g \in L^2 (  \nu_{b})$,
 \begin{align}\label{bound2}
\Theta(n) \langle\mcb {L}_{n} g , g \rangle_{\nu_b} = - \frac{\Theta(n)}{2}  D_n (g, \nu_b) ,
\end{align}
where $D_n (g, \nu_b ) := D^{\mcb F} (g, \nu_b ) + \alpha n^{-\beta}  D^{\mcb S} (g, \nu_b )$, and for $\mcb K \in \{\mcb F, \mcb S\}$, we have
\begin{align*}
D^{\mcb K}  (g, \nu_b) := \frac{1}{2} \sum_{\{ x, y \} \in \mcb K} p(y-x) I_{x,y}  (g, \nu_b ).
\end{align*}
Above, $I_{x,y}  (g, \nu_b ) := \int_{\Omega} [ g \left( \eta^{x,y} \right)  - g \left( \eta \right)  ]^2 d \nu_b$. Since $\nu_b(\eta)=\nu_b(\eta^{x,y})$ for every $x,y \in \mathbb{Z}$ and every $\eta \in \Omega$, with the change of variables $\eta \rightarrow \eta^{x,y}$ and Young's inequality, we obtain that
\begin{align} \label{young2}
\forall  A_{x,y}>0, \quad \Big| \int_{\Omega} [\eta(x) - \eta(y)] g(\eta) d \nu_b \Big| \leq   \frac{ I_{x,y}  (g, \nu_b )}{4A_{x,y}} +  A_{x,y},
\end{align}
for every $g \in L^2(\Omega, \nu_b)$. From \eqref{bound2}, the second term inside the supremum in \eqref{sup01crit} is bounded from above by
\begin{align}  \label{bgcrit1a}
- \frac{\Theta(n)}{4}   \sum_{w= 1}^{\infty} p (1)  I_{w,w-1}  (g, \nu_{b} ) 
\end{align}
From a telescopic sum, it holds
\begin{equation} \label{telsum}
\forall x \geq 1, \forall g \in L^2 (  \nu_{b}), \quad  \langle \bar{\eta}(x) - \bar{\eta}(0) , g \rangle_{\nu_{b}} = \sum_{w=0}^{x-1} \int_{\Omega} [\eta(w+1) - \eta(w)] g(\eta) d \nu_b.
\end{equation}
From \eqref{telsum}, the first term inside the supremum in \eqref{sup01crit} from above by
\begin{align*}
& \frac{\Theta(n)}{n^{\frac{1}{2}+\beta}}  \sum_{x=1}^{\infty}  \sum_{y=-\infty}^{-1} 2 \| G \|_{\infty}  p(y-x) \sum_{w=0}^{x-1} \Big| \int_{\Omega} [ \eta(w+1) - \eta(w) ] g (\eta) d \nu_{b} \Big| \\
=& \frac{\Theta(n)}{n^{\frac{1}{2}+\beta}} \sum_{w=0}^{\infty} \Big| \int_{\Omega} [ \eta(w+1) - \eta(w) ] g (\eta) d \nu_{b} \Big| \Big(2  \| G  \|_{\infty}  \sum_{x=w+1}^{\infty}  \sum_{y=-\infty}^{-1}    p(x-y) \Big) \\
\lesssim&  C \frac{\Theta(n)}{n^{\frac{1}{2}+\beta}} \sum_{w=0}^{\infty} \Big| \int_{\Omega} [ \eta(w+1) - \eta(w) ] g (\eta) d \nu_{b} \Big|   \sum_{x=w+1}^{\infty} x^{-\gamma} \\
\leq& C \frac{\Theta(n)}{n^{\frac{1}{2}+\beta}} \sum_{w=1}^{\infty} w^{1-\gamma} \Big| \int_{\Omega} [ \eta(w) - \eta(w-1) ] g (\eta) d \nu_{b} \Big|,
\end{align*}
for some $C>0$. In the equality above we used Fubini's Theorem. By choosing $A_{w,w-1}=\frac{C w^{1-\gamma}}{n^{\frac{1}{2}+\beta} p(1)}$ in \eqref{young2} and using \eqref{bgcrit1a}, the expression inside the supremum in \eqref{sup01crit} is bounded from above by a constant times
\begin{align*}
\frac{\Theta(n)}{n^{1+2 \beta}} \sum_{w=1}^{\infty} w^{2-2 \gamma},
\end{align*}
and this vanishes as $n \rightarrow \infty$, since $\beta >1/2$, $\gamma \geq 2$ and last sum is convergent. This ends the proof.
\end{proof}
\begin{prop} \label{prop3bgcrit}
Let $(\beta,\gamma) \in R_0$ and $G \in \mathcal{S}(\mathbb{R}^{*})$. Then
\begin{equation}  \label{bgcritpos3}
  \lim_{n \rightarrow \infty} \mathbb{E}_{\nu_b} \Big[ \sup_{t \in [0,T]} \Big(    \int_0^t \frac{\Theta(n)}{\sqrt{n}}  \sum_{x=0}^{\infty} \frac{G_{+}'(0)}{n} \sum_{y=0}^{\infty} (y-x) p(y-x)   [ \bar{\eta}_s(x) - \bar{\eta}_s(0)] ds  \Big)^2 \Big] =0,
\end{equation}
\begin{equation}  \label{bgcritneg3}
  \lim_{n \rightarrow \infty} \mathbb{E}_{\nu_b} \Big[ \sup_{t \in [0,T]}\Big(    \int_0^t \frac{\Theta(n)}{\sqrt{n}} \sum_{x=-\infty}^{-1} \frac{G_{-}'(0)}{n} \sum_{y=-\infty}^{-1} (y-x) p(y-x)   [ \bar{\eta}_s(x) - \bar{\eta}_s(-1)] ds  \Big)^2 \Big] =0.
\end{equation}
\end{prop}
\begin{proof}
We will prove only \eqref{bgcritpos3} but the proof of \eqref{bgcritneg3} is analogous. From the inequality $(u+v)^2 \leq 2(u^2+v^2)$, it is enough to prove that
\begin{equation}   \label{bgcritpos4}
   \lim_{n \rightarrow \infty} \mathbb{E}_{\nu_b} \Big[\sup_{t \in [0,T]} \Big(    \int_0^t \frac{\Theta(n)}{\sqrt{n}}  \sum_{x=n}^{\infty} \frac{G_{+}'(0)}{n} \sum_{y=0}^{\infty} (y-x) p(y-x)   [ \bar{\eta}_s(x) - \bar{\eta}_s(0)] ds  \Big)^2 \Big] =0,
\end{equation}
\begin{equation}  \label{bgcritpos5}
  \lim_{n \rightarrow \infty} \mathbb{E}_{\nu_b} \Big[ \sup_{t \in [0,T]} \Big(    \int_0^t \frac{\Theta(n)}{\sqrt{n}}  \sum_{x=0}^{n-1} \frac{G_{+}'(0)}{n} \sum_{y=0}^{\infty} (y-x) p(y-x)   [ \bar{\eta}_s(x) - \bar{\eta}_s(0)] ds  \Big)^2 \Big] =0.
\end{equation}
We begin by proving the former limit. Combining Cauchy-Schwarz inequality with Fubini's Theorem and the fact that the random variables $ \bar{\eta}_s(x) - \bar{\eta}_s(0)$ have mean zero for every $s \in [0,T]$ and are independent, the expression inside the limit in \eqref{bgcritpos4} is bounded from above by
\begin{align*}
& T \mathbb{E}_{\nu_b} \Big[      \int_0^T \Big( \frac{\Theta(n)}{\sqrt{n}}  \sum_{x=n}^{\infty} \frac{G_{+}'(0)}{n} \sum_{y=0}^{\infty} (y-x) p(y-x)   [ \bar{\eta}_s(x) - \bar{\eta}_s(0) ] \Big)^2 ds   \Big] \\
 = &  T [G_{+}'(0)]^2 \frac{\Theta^2(n)}{n^3} \sum_{x=n}^{\infty} [r_{n,+}(x)]^2 \int_0^T \mathbb{E}_{\nu_b} \big[ \big( \bar{\eta}_s(x) - \bar{\eta}_s(0) \big)^2 \big] ds \\
 = & 2 \chi(b) [T G_{+}'(0)]^2   \frac{\Theta^2(n)}{n^3} \sum_{x=n}^{\infty} [r_{n,+}(x)]^2 \lesssim  \frac{\Theta^2(n)}{n^2} \frac{1}{n} \sum_{x=an}^{\infty} x^{2-2\gamma},
\end{align*}
where for every $x >0$, we have
\begin{align*}
r_{n,+}(x):= \sum_{y=0}^{\infty} (y-x) p(y-x) = \sum_{r=x+1}^{\infty} r p(r) \lesssim x^{1-\gamma}.
\end{align*}
Then the expression inside the limit in \eqref{bgcritpos4} is bounded from above by a constant times
\begin{align*}
\frac{\Theta^2(n)}{n^2} n^{2-2 \gamma} \frac{1}{n} \sum_{x=n}^{\infty} \Big( \frac{x}{n} \Big) ^{2-2\gamma} \lesssim n^{4-2 \gamma} \frac{\Theta^2(n)}{n^{2 \gamma}} \int_{1}^{\infty} u^{2-2 \gamma} du \lesssim \frac{\Theta^2(n)}{n^{2 \gamma}},
\end{align*}
and this vanishes as $n \rightarrow \infty$, for every $\gamma \geq 2$, leading to \eqref{bgcritpos4}. It remains to treat \eqref{bgcritpos5}.

 From the symmetry of $p(\cdot)$ and Lemma 4.3 of \cite{CLO}, the expectation in \eqref{bgcritpos3} is bounded from above by a constant times
\begin{align} \label{sup02crit}
 \sup_{g \in L^2 (  \nu_{b})} \Big\{  \frac{\Theta(n)}{n^{3/2}}  \sum_{x=0}^{n-1} G_{+}'(0) \sum_{r=x+1}^{\infty} r p(r)  \langle \bar{\eta}(x) - \bar{\eta}(0) , g \rangle_{\nu_{b}}   +\Theta(n) \langle \mcb{L}_n g, g \rangle_{\nu_{b}}   \Big\},
\end{align}
From \eqref{bound2}, the second term inside this supremum is bounded from above by
\begin{align}  \label{bgcrit2a}
- \frac{\Theta(n)}{4}  \sum_{w=1}^{n-1} p (1)  I_{w,w-1}  (g, \nu_{b} ) 
\end{align}
From \eqref{telsum}, the first term inside the supremum in \eqref{sup02crit} is bounded from above by
\begin{align*}
& \frac{\Theta(n)}{n^{3/2}}  \sum_{x=0}^{n-1}  \sum_{r=x+1}^{\infty}  \| G'_{+} \|_{\infty} r   p(r) \sum_{w=0}^{x-1} \Big| \int_{\Omega} [ \eta(w+1) - \eta(w) ] g (\eta) d \nu_{b} \Big| \\
=&  \frac{\Theta(n)}{n^{3/2}} \sum_{w=0}^{n-2} \Big| \int_{\Omega} [ \eta(w+1) - \eta(w) ] g (\eta) d \nu_{b} \Big| \Big(  \| G'_{+}  \|_{\infty}   \sum_{x=0}^{n-1}  \sum_{r=x+1}^{\infty} r p(r) \Big) \\
\lesssim&   \frac{\Theta(n)}{n^{3/2}} \sum_{w=0}^{n-2} \Big| \int_{\Omega} [ \eta(w+1) - \eta(w) ] g (\eta) d \nu_{b} \Big|   \sum_{x=w+1}^{n-1} x^{1-\gamma} \\
\leq& C \frac{\Theta(n)}{n^{3/2}} \sum_{w=1}^{n-1} w^{2-\gamma} \Big| \int_{\Omega} [ \eta(w) - \eta(w-1) ] g (\eta) d \nu_{b} \Big|,
\end{align*}
for some $C>0$. In the equality above we used Fubini's Theorem. By choosing $A_{w,w-1}=\frac{C w^{2-\gamma}}{n^{3/2} p(1)}$ in \eqref{young2} and using \eqref{bgcrit2a}, the expression inside the supremum in \eqref{sup02crit} is bounded from above by a constant times
\begin{align} \label{bgcrit3}
\frac{\Theta(n)}{n^3} \sum_{w=1}^{n-1} w^{4-2 \gamma}.
\end{align}
If $\gamma > 5/2$, the sum is convergent and \eqref{bgcrit3} is of order $\Theta(n) n^{-3}$. If $\gamma=5/2$, the sum is of order $\log(n)$ and \eqref{bgcrit3} is of order $\Theta(n) \log(a n) n^{-3}$. Finally, if $\gamma \in [2, 5/2)$, \eqref{bgcrit3} is of order $\Theta(n) n^{2-2 \gamma} a^{5-2 \gamma}$. Therefore \eqref{bgcrit3} goes to zero when $n \rightarrow \infty$, ending the proof.
\end{proof}
Next, we treat the case $(\beta, \gamma) \in R_1$.

\subsection{Case $(\beta, \gamma) \in R_1$} \label{bgibbs}

In this case, it will be useful to introduce extra notation. Observe that when $\beta \geq 1$ and $(\beta, \gamma) \in R_1$, $\mathcal{S}_{\beta,\gamma}:=\mathcal{S}_{Neu}( \mathbb{R}^{*})$, hence for $G \in \mathcal{S}_{\beta,\gamma}$ we have $G_{+}^{'}(0)=G_{+}^{'}(0)=0$. This motivates us to rewrite $\mcb{R}_{n,\beta} G$ as
\begin{align*}
\mcb{R}_{n,\beta} G \left( \tfrac{x}{n} \right) = C_{\alpha,\beta,n} \sum_{y=-\infty}^{-1} \left[ G( \tfrac{y}{n}) -G( \tfrac{x}{n}) \right] p(y-x), \quad  x \geq 0, \; \beta \geq 0,
\end{align*}  
where $C_{\alpha,\beta,n}$ is given by
\begin{equation} \label{defcabn}
C_{\alpha,\beta,n}:=
\begin{cases}
1-\alpha n^{-\beta} \lesssim n^{-o_{\beta}}, & \beta \in [0,1); \\
\alpha n^{-\beta} \lesssim n^{-o_{\beta}}, & \beta \geq 1.
\end{cases}
\end{equation}
Above, $o_{\beta}:=0$ for $\beta \in [0,1)$ and $o_{\beta}:=\beta$ for $\beta \geq 1$. This extra notation is useful to treat the cases $\beta \in [0,1)$ and $\beta \geq 1$ simultaneously.
The following result will be useful  to obtain Proposition \ref{convrem2}.
\begin{prop} \label{convrem3}
Let $(\beta,\gamma) \in R_1$ and $G \in \mathcal{S}_{\beta,\gamma}$. Then
\begin{equation} \label{limbeta00a}
 \limsup_{\varepsilon \rightarrow 0^+} \limsup_{n \rightarrow \infty} \mathbb{E}_{\nu_b} \Big[\sup_{t \in [0,T]} \Big(  \int_0^t \frac{ \Theta(n) C_{\alpha,\beta,n} }{\sqrt{n}}    \sum_{|x| \geq \varepsilon n}  \sum_{y=-\varepsilon n +1}^{-1}  [ G( \tfrac{y}{n} ) - G( \tfrac{x}{n} ) ] p(y-x)   [\bar{\eta}_s ^{\rightarrow \ell} (0) - \bar{\eta}_s^n(0) ] ds \Big)^2 \Big] =0,
\end{equation}
\begin{equation} \label{limbeta00b}
\limsup_{\varepsilon \rightarrow 0^{+}} \limsup_{n \rightarrow \infty} \mathbb{E}_{\nu_b} \Big[\sup_{t \in [0,T]} \Big( \int_{0}^{t} \frac{ C_{\alpha,\beta,n} }{\sqrt{n}}  \sum _{|x| < \varepsilon n} \Theta(n) \sum_{|y| \geq \varepsilon n: \{x,y\} \in \mcb S}  [G( \tfrac{y}{n} ) - G( \tfrac{x}{n} ) ] p(y-x)    \bar{\eta}_{s}^{n}(x) ds \Big)^2 \Big] =0.
\end{equation}
\end{prop}
\begin{proof}
Since $(u+v)^2 \leq 2(u^2+v^2)$, \eqref{limbeta00a} is a consequence of
\begin{align}  \label{limconvrem2a}
\limsup_{\varepsilon \rightarrow 0^{+}} \limsup_{n \rightarrow \infty} \mathbb{E}_{\nu_b} \Big[\sup_{t \in [0,T]} \Big( \int_{0}^{t} \frac{1}{\sqrt{n}} \sum _{x=\varepsilon n}^{\infty} \Theta(n)  \mcb {R}_{n,\beta} G \left( \tfrac{x}{n} \right)   \bar{\eta}_{s}^{n}(x) ds \Big)^2 \Big] =0,
\end{align}
\begin{align}  \label{limconvrem2b}
\limsup_{\varepsilon \rightarrow 0^{+}} \limsup_{n \rightarrow \infty} \mathbb{E}_{\nu_b} \Big[\sup_{t \in [0,T]} \Big( \int_{0}^{t} \frac{1}{\sqrt{n}} \sum _{x= - \infty }^{-\varepsilon n} \Theta(n)  \mcb {R}_{n,\beta} G \left( \tfrac{x}{n} \right)   \bar{\eta}_{s}^{n}(x) ds \Big)^2 \Big] =0.
\end{align}
We prove only \eqref{limconvrem2a}, but the proof of \eqref{limconvrem2b} is analogous. From Cauchy-Schwarz's inequality and Fubini's Theorem, the expectation in \eqref{limconvrem2a} is bounded from above by
\begin{align} \label{auxconvrem2a}
&T  \int_{0}^{T} \frac{1}{n} \sum _{x=\varepsilon n}^{\infty} \big[  \Theta(n)  \mcb {R}_{n,\beta} G \left( \tfrac{x}{n} \right)  \big]^2 \mathbb{E}_{\nu_b} \big[ \big(   \bar{\eta}_{s}^{n}(x) \big)^2 \big] ds = 2b(1-b)  \frac{[T \Theta(n)]^2}{n} \sum _{x=\varepsilon n}^{\infty} \big[   \mcb {R}_{n,\beta} G \left( \tfrac{x}{n} \right)  \big]^2.
\end{align}
From \eqref{defcabn}, for every $\beta \geq 0$ we get
\begin{align*}
 |\mcb {R}_{n,\beta} G \left( \tfrac{x}{n} \right)| \lesssim n^{-o_{\beta}} \sum_{y=-\infty}^{-1} | G( \tfrac{y}{n} ) - G( \tfrac{x}{n} ) |p(y-x) \leq&  2 \|G \|_{\infty} c_{\gamma} n^{-\gamma-o_{\beta}} \frac{1}{n} \sum_{y=-\infty}^{-1} \Big( \frac{x-y}{n} \Big)^{-\gamma-1} \\
\lesssim & n^{-\gamma-o_{\beta}} \int_{-\infty}^{0} \Big( \frac{x}{n} - u \Big)^{-\gamma-1}du \lesssim  n^{-o_{\beta}}  x^{-\gamma},
\end{align*}
and \eqref{auxconvrem2a} is bounded from above by a constant times
\begin{align*}
\frac{[\Theta(n)]^2}{n}  \sum _{x=\varepsilon n}^{\infty} n^{-2o_{\beta}} x^{-2 \gamma} = \frac{[\Theta(n)]^2}{n^{2 \gamma+2 o_{\beta}}} \frac{1}{n} \sum _{x=\varepsilon n}^{\infty} \Big( \frac{x}{n} \Big)^{-2 \gamma} \lesssim \Big( \frac{\Theta(n)}{n^{\gamma+o_{\beta}}} \Big)^2 \varepsilon^{1-2 \gamma} \lesssim \Big( \frac{\Theta(n)}{n^{\gamma}} \Big)^2 \varepsilon^{1-2 \gamma},
\end{align*}
and this vanishes since we take first $n \rightarrow \infty$ and then $\varepsilon \rightarrow 0^{+}$. It remains to show \eqref{limbeta00b}. In order to do so, it is enough to prove that
\begin{align}  \label{limconvrem5c}
\limsup_{\varepsilon \rightarrow 0^{+}} \limsup_{n \rightarrow \infty} \mathbb{E}_{\nu_b} \Big[\sup_{t \in [0,T]} \Big( \int_{0}^{t} \frac{ C_{\alpha,\beta,n}  }{\sqrt{n}}  \sum _{x =0}^{\varepsilon n -1} \Theta(n) \sum_{y=-\infty}^{-\varepsilon n}  [G( \tfrac{y}{n} ) - G( \tfrac{x}{n} ) ] p(y-x)    \bar{\eta}_{s}^{n}(x) ds \Big)^2 \Big] =0,
\end{align}
\begin{align}  \label{limconvrem6c}
\limsup_{\varepsilon \rightarrow 0^{+}} \limsup_{n \rightarrow \infty} \mathbb{E}_{\nu_b} \Big[\sup_{t \in [0,T]} \Big( \int_{0}^{t} \frac{ C_{\alpha,\beta,n}  }{\sqrt{n}}  \sum _{x =- \varepsilon n +1}^{ -1} \Theta(n) \sum_{y=\varepsilon n}^{\infty}  [G( \tfrac{y}{n} ) - G( \tfrac{x}{n} ) ] p(y-x)    \bar{\eta}_{s}^{n}(x) ds \Big)^2 \Big] =0.
\end{align}
We prove only \eqref{limconvrem5c}, but the proof of \eqref{limconvrem6c} is analogous. Observe that $|C(\alpha, \beta, n)|$ is always bounded from above by $1+\alpha$. From Cauchy-Schwarz's inequality and Fubini's Theorem, the expectation in \eqref{limconvrem5c} is bounded from above by
\begin{align*} 
&t  \int_{0}^{t} \frac{1}{n} \sum _{x =0}^{\varepsilon n -1} \Big( (\alpha + 1) \Theta(n) \sum_{y=-\infty}^{- \varepsilon n}  |G( \tfrac{y}{n} ) - G( \tfrac{x}{n} ) | p(y-x) \Big)^2  \mathbb{E}_{\nu_b} \big[ \big(   \bar{\eta}_{s}^{n}(x) \big)^2 \big] ds \\
\lesssim& \chi(b) [  \| G \|_{\infty} t  ]^2   \frac{[\Theta(n) n^{-\gamma}]^2}{n} \sum _{x =0}^{\varepsilon n -1} \Big( \frac{1}{n}  \sum_{y=-\infty}^{-\varepsilon n}  (\tfrac{x-y}{n} )^{-\gamma-1}  \Big)^2 \\
\lesssim &  \frac{[\Theta(n) n^{-\gamma}]^2}{n} \sum _{x =0}^{\varepsilon n -1} \Big( \int_{-\infty}^{-\varepsilon} ( \tfrac{x}{n} - u )^{-\gamma-1} du  \Big)^2 \lesssim  \frac{[\Theta(n) n^{-\gamma}]^2}{n} \sum _{x =0}^{\varepsilon n -1} ( \tfrac{x}{n} + \varepsilon )^{-\gamma} \leq \Big( \frac{\Theta(n)}{n^{\gamma}} \Big)^{2} \varepsilon^{1-\gamma}, 
\end{align*}
that vanishes since we take first $n \rightarrow \infty$ and then $\varepsilon \rightarrow 0^{+}$. This ends the proof.
\end{proof} 
From a change of variables and the symmetry of $p(\cdot)$, we have the following identity:
\begin{align*}
& \sup_{t \in [0,T]} \Big( \int_{0}^{t} \frac{ \Theta(n) C_{\alpha,\beta,n}  }{\sqrt{n}}  \sum _{|x| < \varepsilon n}  \sum_{|y| < \varepsilon n: \{x,y\} \in \mcb S}  [G( \tfrac{y}{n} ) - G( \tfrac{x}{n} ) ] p(y-x)    \bar{\eta}_{s}^{n}(x) ds \Big)^2 \\
 = & \sup_{t \in [0,T]} \Big( \int_{0}^{t} \frac{ \Theta(n) C_{\alpha,\beta,n} }{\sqrt{n}}  \sum _{x =0}^{\varepsilon n -1} \sum_{y=-\varepsilon n +1}^{-1}    [G( \tfrac{y}{n} ) - G( \tfrac{x}{n} ) ] p(y-x)  [  \bar{\eta}_{s}^{n}(x) - \bar{\eta}_{s}^{n}(y)]  ds \Big)^2.
\end{align*}
Therefore, the proof of Proposition \eqref{convrem2} ends if we can show the  next result.
\begin{prop} \label{corbg}
Let $(\beta,\gamma) \in R_1$ and $G \in \mathcal{S}_{\beta,\gamma}$.  Then
\begin{equation*} 
 \limsup_{\varepsilon \rightarrow 0^+} \limsup_{n \rightarrow \infty} \mathbb{E}_{\nu_b} \Big[\sup_{t \in [0,T]} \Big(  \int_0^t \frac{ \Theta(n) C_{\alpha,\beta,n}  }{\sqrt{n}}  \sum_{x=0}^{\varepsilon n -1}  \sum_{y=-\varepsilon n +1}^{-1}  [ G( \tfrac{y}{n} ) - G( \tfrac{x}{n} ) ] p(y-x)   [ \bar{\eta}_s(x) -  \bar{\eta}_s(y) ] ds \Big)^2 \Big] =0.
\end{equation*}
\end{prop}
Proposition \ref{corbg} is a direct consequence of Propositions \ref{firstbg}, \ref{secondbg} and \ref{thirdbg}. For $\ell_0 \geq 1$, define 
\begin{equation} \label{medbox2}
 \overrightarrow{\bar{\eta}}^{\ell_0 }(0):=\frac{1}{\ell_0} \sum_{y=1}^{\ell_0} \bar{\eta}(y)   \; \; \text{and} \; \; \overleftarrow{\bar{\eta}}^{\ell_0 }(0):=\frac{1}{\ell_0} \sum_{y=-\ell_0}^{-1} \bar{\eta}(y).
\end{equation}
\begin{prop} \label{firstbg}
Let $(\beta,\gamma) \in R_1$, $G \in \mathcal{S}_{\beta,\gamma}$ and for $\varepsilon >0$ and $n \geq 1$, we define $\ell$ by
\begin{align} \label{defellmic}
\ell(\varepsilon,n):= \frac{\varepsilon \Theta(n)}{n}.
\end{align}
Then
\begin{equation} \label{limbeta01a}
 \limsup_{\varepsilon \rightarrow 0^+} \limsup_{n \rightarrow \infty} \mathbb{E}_{\nu_b} \Big[\sup_{t \in [0,T]} \Big(  \int_0^t \frac{ \Theta(n) C_{\alpha,\beta,n} }{\sqrt{n}}    \sum_{x=0}^{\varepsilon n-1}  \sum_{y=-\varepsilon n +1}^{-1}  [ G( \tfrac{y}{n} ) - G( \tfrac{x}{n} ) ] p(y-x)   [\overrightarrow{\bar{\eta}}^{\ell}_s (0) - \bar{\eta}_s^n(0) ] ds \Big)^2 \Big] =0,
\end{equation}
\begin{equation}  \label{limbeta01b}
 \limsup_{\varepsilon \rightarrow 0^+} \limsup_{n \rightarrow \infty} \mathbb{E}_{\nu_b} \Big[\sup_{t \in [0,T]} \Big(  \int_0^t  \frac{ \Theta(n) C_{\alpha,\beta,n}  }{\sqrt{n}}  \sum_{x=0}^{\varepsilon n-1}  \sum_{y=-\varepsilon n +1}^{-1}  [ G( \tfrac{y}{n} ) - G( \tfrac{x}{n} ) ] p(y-x)  [\bar{\eta}_s^n(0) - \overleftarrow{\bar{\eta}}^{\ell}_s(0)   ] ds \Big)^2 \Big] =0.
\end{equation}
\end{prop}
\begin{proof}
We present here only the proof of \eqref{limbeta01b}, but the proof of \eqref{limbeta01a} is analogous. From Lemma 4.3 of \cite{CLO} and \eqref{defcabn}, the expectation in \eqref{limbeta01b} is bounded from above by a constant times
\begin{align} \label{sup01b}
 \sup_{g \in L^2 (  \nu_{b})} \Big\{  \frac{\Theta(n)}{\sqrt{n} n^{ o_{\beta} } } \sum_{x=0}^{\varepsilon n-1}  \sum_{y=-\varepsilon n +1}^{-1}  [ G( \tfrac{y}{n} ) - G( \tfrac{x}{n} ) ] p(y-x)   \langle \bar{\eta}(0) - \overleftarrow{\bar{\eta}}^{\ell}(0) , g \rangle_{\nu_{b}}   +\Theta(n) \langle \mcb {L}_n g, g \rangle_{\nu_{b}}   \Big\}.
\end{align}
From \eqref{bound2}, the second term inside last supremum is bounded from above by
\begin{align}  \label{bgibbs1a}
- \frac{\Theta(n)}{4}   \sum_{w=- \ell+1}^{-1} p (1)  I_{w,w-1}  (g, \nu_{b} ) - \alpha \frac{\Theta(n)}{4n^{\beta}}    p(1) I_{0,-1}  (g, \nu_{b} ).
\end{align}
Now we will treat the first term inside the supremum in \eqref{sup01b}. From \eqref{medbox2}, we get
 \begin{align*}
& |\langle \bar{\eta}(0) - \overleftarrow{\bar{\eta}}^{\ell}(0) , g \rangle_{\nu_{b}} |=\Big| \frac{1}{\ell}  \sum_{y= - \ell }^{-1} \int    [  \eta(0) -  \eta(y) ] g(\eta) d \nu_{b} \Big| \\
=& \Big|\frac{1}{\ell}  \sum_{y= - \ell }^{-1} \sum_{w=y+1}^{0} \int    [  \eta(w) -  \eta(w-1) ] g(\eta) d \nu_{b} \Big| \leq \sum_{w=-\ell+1}^{0} \Big|\int    [  \eta(w) -  \eta(w-1) ] g(\eta) d \nu_{b}  \Big|.
\end{align*}
For the remainder of the proof, we treat two cases separately: $\beta \in [0,1)$ and $\beta \geq 1$.

\textbf{I.)} In this case $\beta \in [0,1)$, $o_{\beta}=0$ and $G \in \mathcal{S}_{\beta, \gamma} = \mathcal{S}(\mathbb{R})$. Then by performing a Taylor expansion of first order on $G$, the first term inside the supremum in \eqref{sup01b} is bounded from above by
\begin{align*}
& \frac{\Theta(n)}{n \sqrt{n}} \sum_{x=0}^{\varepsilon n-1}  \sum_{y=-\varepsilon n +1}^{-1}   \| G'  \|_{\infty}  (x-y) p(x-y) | \langle \bar{\eta}(0) - \overleftarrow{\bar{\eta}}^{\ell}(0) , g \rangle_{\nu_{b}}| \\
\lesssim & \sqrt{n} \sum_{w=-\ell+1}^{0} \Big|\int    [  \eta(w) -  \eta(w-1) ] g(\eta) d \nu_{b}  \Big| \Big(  \frac{\Theta(n)}{n^2} \sum_{x=0}^{\varepsilon n-1}  \sum_{y=-\varepsilon n +1}^{-1}    (x-y) p(x-y) \Big) \\
\lesssim & \sqrt{n} \sum_{w=-\ell+1}^{0} \Big|\int    [  \eta(w) -  \eta(w-1) ] g(\eta) d \nu_{b}  \Big|  \\ 
= &   \sqrt{n}   \sum_{w=-\ell+1}^{-1}  \Big| \int    [  \eta(w) -  \eta(w-1) ] f(\eta) d \nu_{b} \Big| +  \sqrt{n}  \Big| \int    [  \eta(0) -  \eta(-1) ] f(\eta) d \nu_{b} \Big|.
\end{align*}
Above we used \eqref{thetan}. We treat $ \int   [  \eta(0) -  \eta(-1) ] g(\eta) d \nu_{b}$ separately because only for $w=0$ the set $\{w,w-1\}$ corresponds to a slow bond. From \eqref{young2}, for any positive constants $A_{w,w-1}$ and $A_{0,-1}$ last display is bounded from above by
\begin{align} \label{bgibbs1b}
  \sqrt{n} \Big\{ \sum_{w=-\ell+1}^{-1} \Big[ \frac{I_{w,w-1}  (g, \nu_{b} )}{4A_{w,w-1}} + A_{w,w-1}  \Big] +  \Big[ \frac{I_{0,-1}  (g, \nu_{b} )}{4A_{0,-1}} + A_{0,-1}  \Big] \Big\}.
\end{align}
Next, we choose $A_{w,w-1}$ and $A_{0,-1}$ carefully in order to cancel the negative terms in \eqref{bgibbs1a}. More exactly, we choose 
$A_{w,w-1}=\frac{\sqrt{n}}{\Theta(n) p(1)}$ and $A_{0,-1} = \frac{\sqrt{n} n^{\beta}}{\alpha \Theta(n) p(1)}$. Hence we can bound the sum of \eqref{bgibbs1a} and \eqref{bgibbs1b} from above by a constant times
\begin{align*}
 \frac{\ell n}{\Theta(n)} + \frac{n^{1+\beta}}{\Theta(n)}.
\end{align*}
Last display is therefore an upper bound for the expression inside the supremum inside \eqref{sup01b}. Recalling \eqref{defellmic}, we can bound the expectation in \eqref{limbeta01b} from above by a constant times $\varepsilon + n^{1+\beta}/\Theta(n)$, that vanishes as $n \rightarrow \infty$ and then $\varepsilon \rightarrow 0^{+}$. 

\textbf{II.)} In this case $\beta \geq 1$ and $o_{\beta}=\beta$, hence the first term inside the supremum in \eqref{sup01b} is bounded from above by
\begin{align*}
& \frac{\Theta(n)}{n^{\beta} \sqrt{n}} \sum_{x=0}^{\varepsilon n-1}  \sum_{y=-\varepsilon n +1}^{-1}   \| G  \|_{\infty}   p(x-y) | \langle \bar{\eta}(0) - \overleftarrow{\bar{\eta}}^{\ell}(0) , g \rangle_{\nu_{b}}| \\
\lesssim& \frac{\Theta(n)}{n^{\beta} \sqrt{n}}  \sum_{w=-\ell+1}^{0} \Big|\int    [  \eta(w) -  \eta(w-1) ] g(\eta) d \nu_{b}  \Big| \Big(   \sum_{x=0}^{\infty}  \sum_{y=-\infty}^{-1}    p(x-y) \Big) \\
\lesssim & \frac{\Theta(n)}{n^{\beta} \sqrt{n}}  \sum_{w=-\ell+1}^{0} \Big|\int    [  \eta(w) -  \eta(w-1) ] g(\eta) d \nu_{b}  \Big|  \\ 
= &  \frac{\Theta(n)}{n^{\beta} \sqrt{n}}    \sum_{w=-\ell+1}^{-1}  \Big| \int    [  \eta(w) -  \eta(w-1) ] f(\eta) d \nu_{b} \Big| +\frac{\Theta(n)}{n^{\beta} \sqrt{n}}  \Big| \int    [  \eta(0) -  \eta(-1) ] f(\eta) d \nu_{b} \Big|,
\end{align*}
From \eqref{young2}, for any positive constants $A_{w,w-1}$ and $A_{0,-1}$ last display is bounded from above by
\begin{align} \label{bgibbs1c}
 \frac{\Theta(n)}{n^{\beta} \sqrt{n}} \Big\{ \sum_{w=-\ell+1}^{-1} \Big[ \frac{I_{w,w-1}  (g, \nu_{b} )}{4A_{w,w-1}} + A_{w,w-1}  \Big] +  \Big[ \frac{I_{0,-1}  (g, \nu_{b} )}{4A_{0,-1}} + A_{0,-1}  \Big] \Big\}.
\end{align}
Now we choose $A_{w,w-1}=\frac{1}{n^{  \beta } \sqrt{n} p(1)}$ and $A_{0,-1} = \frac{1}{\alpha \sqrt{n}  p(1)}$. Then applying \eqref{defellmic}, the sum of \eqref{bgibbs1a} and \eqref{bgibbs1c} is bounded from above by a constant times
\begin{align*}
 \frac{\ell \Theta(n)}{ n^{1+2\beta} } + \frac{\Theta(n)}{n^{1+\beta}} = \varepsilon \Big( \frac{\Theta(n)}{n^{1+\beta}} \Big)^2 + \frac{\Theta(n)}{n^{1+\beta}},
\end{align*}
which vanishes as $n \rightarrow \infty$ and then $\varepsilon \rightarrow 0^{+}$, ending the proof. 
\end{proof}
We prove next result with an analogous reasoning.
\begin{prop} \label{secondbg}
Let $\ell$ be given by \eqref{defellmic}, $(\beta,\gamma) \in R_1$ and $G \in \mathcal{S}_{\beta,\gamma}$. Then
\begin{equation}  \label{limbeta01c}
 \limsup_{\varepsilon \rightarrow 0^+} \limsup_{n \rightarrow \infty} \mathbb{E}_{\nu_b} \Big[\sup_{t \in [0,T]} \Big(  \int_0^t \frac{ \Theta(n) C_{\alpha,\beta,n} }{\sqrt{n}}   \sum_{x=0}^{\ell -1}  \sum_{y=-\varepsilon n +1}^{-1}  [ G( \tfrac{y}{n} ) - G( \tfrac{x}{n} ) ] p(y-x)   [ \bar{\eta}_s(x) -  \overrightarrow{\bar{\eta}}^{\ell}_s(0) ] ds \Big)^2 \Big] =0,
\end{equation}
\begin{equation}  \label{limbeta01d}
 \limsup_{\varepsilon \rightarrow 0^+} \limsup_{n \rightarrow \infty} \mathbb{E}_{\nu_b} \Big[\sup_{t \in [0,T]} \Big(  \int_0^t \frac{ \Theta(n) C_{\alpha,\beta,n} }{\sqrt{n}}   \sum_{x=0}^{\varepsilon n -1}  \sum_{y=-\ell +1}^{-1}  [ G( \tfrac{y}{n} ) - G( \tfrac{x}{n} ) ] p(y-x)  [ \overleftarrow{\bar{\eta}}^{\ell}_s(0) - \bar{\eta}_s^n(y)    ] ds \Big)^2 \Big] =0.
\end{equation}
\end{prop}
\begin{proof}
We present here only the proof of \eqref{limbeta01c}, but the proof of \eqref{limbeta01d} is analogous. From Lemma 4.3 of \cite{CLO}, the expectation in \eqref{limbeta01c} is bounded from above by a constant times
\begin{align} \label{sup02b}
 \sup_{g \in L^2 (  \nu_{b})} \Big\{   \frac{\Theta(n)}{\sqrt{n} n^{ o_{\beta} } }  \sum_{x=0}^{\ell -1}  \sum_{y=-\varepsilon n +1}^{-1}  [ G( \tfrac{y}{n} ) - G( \tfrac{x}{n} ) ] p(y-x)   \langle \bar{\eta}(x) - \overrightarrow{\bar{\eta}}^{\ell}(0) , g \rangle_{\nu_{b}}   +\Theta(n) \langle \mcb {L}_n g, g \rangle_{\nu_{b}}   \Big\},
\end{align}
From \eqref{bound2}, the second term inside this supremum is bounded from above by
\begin{align}  \label{bgibbs2a}
- \frac{\Theta(n)}{4}   \sum_{w= 0}^{\ell-1} p (1)  I_{w,w+1}  (g, \nu_{b} ) 
\end{align}
From \eqref{medbox2} $| \langle \bar{\eta}(x) - \overrightarrow{\bar{\eta}}^{\ell}(0) , g \rangle_{\nu_{b}} |$ is bounded from above by
\begin{align*}
  & \frac{1}{\ell}  \sum_{z=1}^{x-1} \sum_{w=1}^{x-1} \Big| \int_{\Omega} [ \eta(w+1) - \eta(w) ] f (\eta) d \nu_{b} \Big| + \frac{1}{\ell} \sum_{z=x+1}^{\ell} \sum_{w=x}^{z-1} \Big| \int_{\Omega} [ \eta(w) - \eta(w+1) ] g (\eta) d \nu_{b} \Big| \\
\leq & 2 \sum_{w=0}^{\ell-1} \Big| \int_{\Omega} [ \eta(w) - \eta(w+1) ] g (\eta) d \nu_{b} \Big|.
\end{align*}
Above we used that $0 \leq x \leq \ell-1$. For the remainder of the proof, we treat two cases separately: $\beta \in [0,1)$ and $\beta \geq 1$.

\textbf{I.)} In this case $\beta \in [0,1)$, $o_{\beta}=0$ and $G \in \mathcal{S}_{\beta, \gamma} = \mathcal{S}(\mathbb{R})$. Then by performing a Taylor expansion of first order on $G$, the first term inside the supremum in \eqref{sup02b} is bounded from above by
\begin{align*}
& \frac{\Theta(n)}{n \sqrt{n}} \sum_{x=0}^{\ell-1}  \sum_{y=-\varepsilon n +1}^{-1}   \| G'  \|_{\infty}  (x-y) p(x-y) | \langle \bar{\eta}(x) - \overrightarrow{\bar{\eta}}^{\ell}(0) , g \rangle_{\nu_{b}}| \\
\lesssim& \sqrt{n} \sum_{w=0}^{\ell-1} \Big| \int_{\Omega} [ \eta(w) - \eta(w+1) ] g (\eta) d \nu_{b} \Big| \Big( \frac{\Theta(n)}{n^2} \sum_{x=0}^{\varepsilon n-1}  \sum_{y=-\varepsilon n +1}^{-1}    (x-y) p(x-y) \Big) \\
\lesssim & \sqrt{n} \sum_{w=0}^{\ell-1} \Big| \int_{\Omega} [ \eta(w) - \eta(w+1) ] g (\eta) d \nu_{b} \Big| ,
\end{align*}
From \eqref{young2}, for any positive $A_{w,w+1}$ last display is bounded from above by
\begin{align} \label{bgibbs2b}
\sqrt{n} \sum_{w=0}^{\ell-1} \Big[ \frac{I_{w,w+1}  (g, \nu_{b} )}{4A_{w,w+1}} + A_{w,w+1}  \Big].
\end{align}
By choosing $A_{w,w+1}=\frac{ \sqrt{n} }{\Theta(n) p(1)}$ and recalling \eqref{defellmic}, the sum of the expressions in \eqref{bgibbs2a} and \eqref{bgibbs2b} is bounded from above by a constant times $\ell n [\Theta(n) ]^{-1}=\varepsilon$, that vanishes as $n \rightarrow \infty$ and then $\varepsilon \rightarrow 0^{+}$.

\textbf{II.)} In this case $\beta \geq 1$ and $o_{\beta}=\beta$, hence the first term inside the supremum in \eqref{sup02b} is bounded from above by
\begin{align*}
& \frac{\Theta(n)}{ n^{\beta} \sqrt{n} } \sum_{x=0}^{\varepsilon n-1}  \sum_{y=-\varepsilon n +1}^{-1}   \| G  \|_{\infty}   p(x-y) | \langle \bar{\eta}(x) - \overrightarrow{\bar{\eta}}^{\ell}(0) , g \rangle_{\nu_{b}}| \\
\lesssim & \frac{\Theta(n)}{ n^{\beta} \sqrt{n}} \sum_{w=0}^{\ell-1} \Big| \int_{\Omega} [ \eta(w) - \eta(w+1) ] g (\eta) d \nu_{b} \Big| \Big( \sum_{x=0}^{\infty}  \sum_{y=-\infty}^{-1}    p(x-y) \Big) \\
\lesssim & \frac{\Theta(n)}{ n^{\beta} \sqrt{n} }  \sum_{w=0}^{\ell-1} \Big| \int_{\Omega} [ \eta(w) - \eta(w+1) ] g (\eta) d \nu_{b} \Big|,
\end{align*}
By choosing $A_{w,w+1}=\frac{1}{ n^{\beta} \sqrt{n} p(1)}$ in \eqref{young2} and recalling \eqref{defellmic}, the expression inside the supremum in \eqref{sup01b} is bounded from above by
\begin{align*}
 \frac{\ell \Theta(n)}{ n^{1+2\beta} }  = \varepsilon \Big( \frac{\Theta(n)}{n^{1+\beta}} \Big)^2, 
\end{align*}
and this vanishes as $n \rightarrow \infty$ and then $\varepsilon \rightarrow 0^{+}$, ending the proof.
\end{proof}
For $\gamma =2$, we have $\ell = \varepsilon \Theta(n) n^{-1} < \varepsilon n$ and we need a complementary result.
\begin{prop} \label{thirdbg}
Let $\ell$ be given by \eqref{defellmic}, $\gamma=2$ (this implies $(\beta,\gamma) \in R_1$) and $G \in \mathcal{S}_{\beta,\gamma}$. Then
\begin{equation}  \label{limbeta01e}
 \limsup_{\varepsilon \rightarrow 0^+} \limsup_{n \rightarrow \infty} \mathbb{E}_{\nu_b} \Big[\sup_{t \in [0,T]} \Big(  \int_0^t \frac{ \Theta(n) C_{\alpha,\beta,n} }{\sqrt{n}}  \sum_{x=\ell}^{\varepsilon n -1}  \sum_{y=-\varepsilon n +1}^{-1}  [ G( \tfrac{y}{n} ) - G( \tfrac{x}{n} ) ] p(y-x)   [ \bar{\eta}_s(x) - \overrightarrow{\bar{\eta}}^{\ell}_s(0) ] ds \Big)^2 \Big] =0,
\end{equation}
\begin{equation}  \label{limbeta01f}
 \limsup_{\varepsilon \rightarrow 0^+} \limsup_{n \rightarrow \infty} \mathbb{E}_{\nu_b} \Big[\sup_{t \in [0,T]} \Big(  \int_0^t \frac{ \Theta(n) C_{\alpha,\beta,n} }{\sqrt{n}}   \sum_{x=0}^{\varepsilon n -1}  \sum_{y=-\varepsilon n +1}^{-\ell}  [ G( \tfrac{y}{n} ) - G( \tfrac{x}{n} ) ] p(y-x)  [ \overleftarrow{\bar{\eta}}^{\ell}_s(0) - \bar{\eta}_s^n(y)    ] ds \Big)^2 \Big] =0.
\end{equation}
\end{prop}
\begin{proof}
We present here only the proof of \eqref{limbeta01e}, but the proof of \eqref{limbeta01f} is analogous. From Lemma 4.3 of \cite{CLO}, the expectation in \eqref{limbeta01e} is bounded from above by a constant times
\begin{align} \label{sup03b}
  \sup_{g \in L^2 (  \nu_{b})} \Big\{  \frac{\Theta(n)}{\sqrt{n} n^{ o_{\beta} } }  \sum_{x=\ell}^{\varepsilon n -1}  \sum_{y=-\varepsilon n +1}^{-1}  [ G( \tfrac{y}{n} ) - G( \tfrac{x}{n} ) ] p(y-x)   \langle \bar{\eta}(x) - \overrightarrow{\bar{\eta}}^{\ell}(0) , g \rangle_{\nu_{b}}   +\Theta(n) \langle \mcb{L}_n g, g \rangle_{\nu_{b}}   \Big\},
\end{align}
Since $\gamma=2$, from \eqref{bound2}, the second term inside this supremum is bounded from above by
\begin{align*} 
- \frac{\Theta(n)}{4}   \sum_{w=1}^{\varepsilon n} p (1)  I_{w,w+1}  (g, \nu_{b} )  = - \frac{n^2}{4 \log(n)}   \sum_{w=1}^{\varepsilon n} p (1)  I_{w,w+1}  (g, \nu_{b} ).  
\end{align*}
From \eqref{medbox2} $| \langle \bar{\eta}(x) - \overrightarrow{\bar{\eta}}^{\ell}(0) , g \rangle_{\nu_{b}} |$ is bounded from above by
\begin{align*}
  &   \frac{1}{\ell}  \sum_{z=1}^{\ell} \sum_{w=z}^{x-1} \Big| \int_{\Omega} [ \eta(w+1) - \eta(w) ] f (\eta) d \nu_{b} \Big|  \leq \sum_{w=1}^{\varepsilon n} \Big| \int_{\Omega} [ \eta(w+1) - \eta(w) ] f (\eta) d \nu_{b} \Big|.
\end{align*}
Since $\gamma =2$, $\ell = \varepsilon n / \log(n)$ and we get
\begin{equation} \label{boundgamma2}
\begin{split}
& \frac{\Theta(n)}{n^2} \sum_{x=\ell}^{\varepsilon n-1}  \sum_{y=-\varepsilon n +1}^{-1}  c_2  (x-y)^{-2}  \lesssim \Theta(n) n^{- 2}  \sum_{x=\ell}^{\varepsilon n-1}  x^{-1} \leq \frac{1}{\log(n)} \Big ( \ell^{-1} + \int_{\ell}^{\varepsilon n} u^{-1} du \Big) \\
=& \frac{1}{\varepsilon n} + \frac{1}{\log(n)} \log \big( \log(n) \big) \lesssim \frac{\log \big( \log(n) \big)}{\log(n)}, 
\end{split}
\end{equation}
since $n$ goes to infinity before $\varepsilon$ goes to zero. For the remainder of the proof, we treat two cases separately: $\beta \in [0,1)$ and $\beta \geq 1$.

\textbf{I.)} In this case $\beta \in [0,1)$, $o_{\beta}=0$ and $G \in \mathcal{S}_{\beta, \gamma} = \mathcal{S}(\mathbb{R})$. Then by performing a Taylor expansion of first order on $G$, the first term inside the supremum in \eqref{sup03b} is bounded from above by
\begin{align*}
& \frac{\Theta(n)}{n \sqrt{n}} \sum_{x=\ell}^{\varepsilon n-1}  \sum_{y=-\varepsilon n +1}^{-1}   \| G'  \|_{\infty}  (x-y) p(x-y) |  \langle \bar{\eta}(x) - \overrightarrow{\bar{\eta}}^{\ell}(0) , g \rangle_{\nu_{b}}| \\
\lesssim & \sqrt{n} \sum_{w=1}^{\varepsilon n} \Big| \int_{\Omega} [ \eta(w+1) - \eta(w) ] g (\eta) d \nu_{b} \Big| \Big(  \frac{\Theta(n)}{n^2} \sum_{x=\ell}^{\varepsilon n-1}  \sum_{y=-\varepsilon n +1}^{-1}   c_2 (x-y)^{-2}  \Big) \\
\lesssim &  \sqrt{n} \frac{\log \big( \log(n) \big)}{\log(n)}   \sum_{w=1}^{\varepsilon n} \Big| \int    [  \eta(w+1) -  \eta(w) ] f(\eta) d \nu_{b}  \Big|.
\end{align*}
Above we used \eqref{boundgamma2}. By choosing $A_{w,w+1}=\frac{  \log ( \log(n) ) }{ p(1) n \sqrt{n}}$ for every $1 \leq w \leq \varepsilon n$ in \eqref{young2}, the expression in \eqref{sup03b} is bounded by a constant times
\begin{align*}
 \varepsilon \frac{\big[\log \big( \log(n) \big) \big]^2}{\log(n)},
\end{align*}
that vanishes as $n \rightarrow \infty$ and then $\varepsilon \rightarrow 0^{+}$. 

\textbf{II.)} In this case $\beta \geq 1$ and $o_{\beta}=\beta$, hence the first term inside the supremum in \eqref{sup03b} is bounded from above by
\begin{align*}
& \frac{\Theta(n)}{ n^{\beta} \sqrt{n} } \sum_{x=\ell}^{\varepsilon n-1}  \sum_{y=-\varepsilon n +1}^{-1}   \| G  \|_{\infty}   p(x-y) | \langle \bar{\eta}(x) -\overrightarrow{\bar{\eta}}^{\ell}(0) , g \rangle_{\nu_{b}}| \\
\lesssim& \frac{\Theta(n)}{ n^{\beta} \sqrt{n} } \sum_{w=1}^{\varepsilon n} \Big| \int_{\Omega} [ \eta(w+1) - \eta(w) ] g (\eta) d \nu_{b} \Big| \Big(    \sum_{x=0}^{\infty}  \sum_{y=-\infty}^{-1}    p(x-y) \Big) \\
\lesssim & \frac{\Theta(n)}{ n^{\beta} \sqrt{n} }  \sum_{w=1}^{\varepsilon n} \Big| \int_{\Omega} [ \eta(w+1) - \eta(w) ] g (\eta) d \nu_{b} \Big|, 
\end{align*}
By choosing $A_{w,w+1}=\frac{1}{n^{\beta} \sqrt{n} p(1)}$ in \eqref{young2} and recalling \eqref{defellmic}, the expression inside the supremum in \eqref{sup01b} is bounded from above by
\begin{align*}
 \frac{\varepsilon n \Theta(n)}{ n^{1+2\beta} }  = \varepsilon \frac{\Theta(n)}{n^{2 \beta}} \leq \varepsilon \frac{\Theta(n)}{n^{2}},
 \end{align*}
that vanishes as $n \rightarrow \infty$ and then $\varepsilon \rightarrow 0^{+}$, ending the proof.
\end{proof}

	\appendix

	\section{Auxiliary results} \label{secdiscconv}
	\label{secuseres}

In this section we prove various results that were used along the text. Recall the definition of $R_0$ and $\mathcal{S}_{\beta, \gamma}$ in Definition \ref{defsbetagamma}. 
\begin{prop} \label{convknbeta}
Let $(\beta, \gamma) \in R_0$ and $G \in \mathcal{S}_{\beta, \gamma}$. Then
\begin{equation} \label{convknbetasup}
\sup_{x \in \mathbb{Z}}  |\Theta(n)  (\mathcal{K}_{n,\beta} G) (\tfrac{x}{n} ) |  \lesssim 1
\end{equation}
and
\begin{equation} \label{convknbetasum}
\lim_{n \rightarrow \infty} \frac{1}{n} \sum_{x } | \Theta(n) (\mathcal{K}_{n,\beta} G) (\tfrac{x}{n} ) - \kappa_{\gamma} \Delta G( \tfrac{x}{n} ) | =0.
\end{equation}
\end{prop}
\begin{proof}
\textbf{I.:} Proof of \eqref{convknbetasup}. 

If $\beta \in [0,1)$, $G \in G \in \mathcal{S}_{\beta, \gamma} = \mathcal{S}( \mathbb{R})$. Therefore, from the symmetry of $p$ and a Taylor expansion of second order on $G$ we get
\begin{align*}
&|\Theta(n) (\mathcal{K}_{n,\beta} G) (\tfrac{x}{n} )| \leq  \Theta(n) \sum_{|r| \geq n } |G(\tfrac{r+x}{n}) - G( \tfrac{x}{n} )|  p(r) + \Theta(n) \Big| \sum_{|r| < n } [G(\tfrac{r+x}{n}) - G( \tfrac{x}{n} )]  p(r) \Big|  \\
\leq& \frac{\Theta(n)}{n^{\gamma}} \frac{2 \| G \|_{\infty}}{n} 2 \sum_{r=n}^{\infty} \Big( \frac{r}{n} \Big)^{-1-\gamma} +   \| G'' \|_{\infty} \frac{\Theta(n)}{n^2} \sum_{r=1}^{n-1}  c_{\gamma} r^{1-\gamma}\lesssim \frac{\Theta(n)}{n^{\gamma}} + 1 \lesssim 1,
\end{align*}
leading to \eqref{convknbetasup}. In last line we used \eqref{thetan}. Now we study the case $\beta \geq 1$ and $x \geq 0$. If $ x < n$, from some Taylor expansions on $G_{+}$ and \eqref{thetan} we get
\begin{align*}
& | \Theta(n)  (\mathcal{K}_{n,\beta} G) (\tfrac{x}{n} ) | =  \Big|   \Theta(n)  \sum_{r=-x}^{\infty} [G_{+}(\tfrac{x+r}{n}) - G_{+}( \tfrac{z}{x} )]  p(r) - \frac{\Theta(n)}{n} G_{+}'(0) \sum_{r=-x}^{\infty} r p (r)  \Big| \\
 \leq&2 \frac{\Theta(n)}{n^2} \| G'' \|_{\infty} \sum_{r=1}^{ n} r^2p(r) +  \| G \|_{\infty} \frac{\Theta(n)}{n^{\gamma}} \frac{1}{n} \sum_{r=n}^{\infty} \Big( \frac{r}{n} \Big)^{-\gamma-1} +   \| G' \|_{\infty} \frac{\Theta(n)}{n^{\gamma}} \frac{1}{n} \sum_{r=n}^{\infty} \Big( \frac{r}{n} \Big)^{-\gamma} \\
 \lesssim &  \frac{\Theta(n)}{n^2} \sum_{r=1}^{ n} r^2p(r) + \frac{\Theta(n)}{n^{\gamma}} \int_{1}^{\infty} u^{-\gamma-1} du + \frac{\Theta(n)}{n^{\gamma}} \int_{1}^{\infty} u^{-\gamma} du \lesssim 1 + 2 \frac{\Theta(n)}{n^{\gamma}}  \lesssim 1.
\end{align*}
On the other hand, if $x \geq n$, from Taylor expansions on $G_{+}$ and \eqref{thetan} we get
\begin{align*}
& | \Theta(n)  (\mathcal{K}_{n,\beta} G) (\tfrac{x}{n} ) | =  \Big|   \Theta(n)  \sum_{r=-x}^{\infty} [G_{+}(\tfrac{x+r}{n}) - G_{+}( \tfrac{x}{n} )]  p(r) - \frac{\Theta(n)}{n} G_{+}'(0) \sum_{r=-x}^{\infty} r p (r)  \Big| \\
\leq & \frac{\Theta(n)}{n} |G_{+}'(0)| \sum_{r=x+1}^{\infty} r p(r) + \frac{\Theta(n)}{n^2} \| G'' \|_{\infty} \sum_{r=1}^{ n} r^2p(r) +   \| G \|_{\infty} \frac{\Theta(n)}{n^{\gamma}} \frac{1}{n} \sum_{r=n}^{\infty} \Big( \frac{r}{n} \Big)^{-\gamma-1} \\
\lesssim & \frac{\Theta(n)}{n^{2-\gamma}} _{1}^{\infty} u^{-\gamma} du + 1 + \frac{\Theta(n)}{n^{2-\gamma}} _{1}^{\infty} u^{-\gamma-1} du \lesssim 1.
\end{align*}
In an analogous way, the result holds for $x <0$ and $\beta \geq 1$. This ends the proof of \eqref{convknbetasup}. It remains to prove \eqref{convknbetasum}. We treat two cases separately: $\beta \in [0,1)$ and $\beta \geq 1$.

\textbf{II.:} Proof of \eqref{convknbetasum} for $\beta \in [0,1)$. 

In this case, $G \in \mathcal{S}_{\beta, \gamma} = \mathcal{S}( \mathbb{R})$ and $\Delta_{\beta, \gamma}$ coincides with the usual Laplacian $\Delta$. Therefore, if $\gamma >2$, \eqref{convknbetasum} is a direct consequence of item $(b)$ of Lemma A.1 of \cite{goncalves2017stochastic}. It remains to prove it when $\gamma=2$.  Since $\beta \in [0,1)$, from \eqref{op_Knb} we have
\begin{align} \label{knbeta01}
\mcb{K}_{n,\beta} G \left( \tfrac{x}{n} \right): = \sum_{r } [ G( \tfrac{x+r}{n}) -G( \tfrac{x}{n}) ] p(r).
\end{align} 
Therefore, it is enough to prove that
 \begin{align}
& \limsup_{B \rightarrow \infty}  \limsup_{n \rightarrow \infty} \frac{1}{n} \sum_{ |x| \leq Bn } | \Theta(n) (\mathcal{K}_{n,\beta} G) (\tfrac{x}{n} ) - \kappa_{\gamma} \Delta G( \tfrac{x}{n} ) | =0,  \label{convprincnear0} \\
& \limsup_{B \rightarrow \infty}  \limsup_{n \rightarrow \infty} \frac{1}{n} \sum_{ |x| > Bn } | \Theta(n) (\mathcal{K}_{n,\beta} G) (\tfrac{x}{n} ) - \kappa_{\gamma} \Delta G( \tfrac{x}{n} ) | =0. \label{convprincfar0}
 \end{align}
We begin with \eqref{convprincnear0}. From \eqref{knbeta01}, for every fixed $B >0$, it holds 
\begin{align}
&\limsup_{n \rightarrow \infty} \frac{1}{n} \sum_{ |x| \leq Bn } | \Theta(n) (\mathcal{K}_{n,\beta} G) (\tfrac{x}{n} ) - \kappa_{\gamma} \Delta G( \tfrac{x}{n} ) | \nonumber \\
 \leq & \limsup_{\varepsilon \rightarrow 0^{+}}  \limsup_{n \rightarrow \infty} \frac{1}{n} \sum_{ |x| \leq Bn } \Big| \Theta(n) \sum_{|r| \geq \varepsilon n }p( r) \left[ G ( \tfrac{x+r}{n} ) - G ( \tfrac{x}{n} ) \right] \Big| \label{convprincnear0a} \\
 +& \limsup_{\varepsilon \rightarrow 0^{+}}  \limsup_{n \rightarrow \infty} \frac{1}{n} \sum_{ |x| \leq Bn } \Big| \Theta(n) \sum_{|r| < \varepsilon n }p( r) \left[ G ( \tfrac{x+r}{n} ) - G ( \tfrac{x}{n} ) \right] - \kappa_{\gamma} \Delta G( \tfrac{x}{n} )  \Big|. \label{convprincnear0b}
\end{align}
To treat \eqref{convprincnear0a}, note that
\begin{align*}
&  \frac{1}{n} \sum_{|x| \leq B n} \Big| \Theta(n) \sum_{|r| \geq \varepsilon n }p(r) \left[ G( \tfrac{x+r}{n} ) - G( \tfrac{x}{n} ) \right]  \Big|  \lesssim  \frac{ \|  G \|_{\infty} \Theta(n)}{n}  \sum_{|x| \leq B n} \sum_{|r| \geq  \varepsilon n} |r|^{-\gamma-1}   \\
\lesssim &  \frac{\Theta(n)}{n^{\gamma}} \frac{1}{n^2} \sum_{|x| \leq B n}  \sum_{r = \varepsilon n}^{\infty} \Big( \frac{r}{n} \Big)^{-\gamma-1} \lesssim  \frac{\Theta(n)}{n^{\gamma}} \int_{- B}^{ B} \int_{\varepsilon}^{\infty} v^{-\gamma-1} dv du \lesssim B \frac{\Theta(n)}{ n^{\gamma}} \varepsilon^{-\gamma} = \frac{B \varepsilon^{-\gamma}}{\log(n)}, 
\end{align*}
which goes to zero for every $B$ fixed, taking first $n \rightarrow \infty$ and then $\varepsilon \rightarrow 0^{+}$. In last equality we used the expression of $\Theta(n)$ for $\gamma=2$. It remains to analyse \eqref{convprincnear0b}. Due to the symmetry of $p$, $ \sum_{|r| < \varepsilon n }rp(r)=0$ for every $\varepsilon >0$. Hence by performing a Taylor expansion of second order on $G$, we can rewrite the expression inside the absolute value in \eqref{convprincnear0b} as
\begin{align*}
 &  \frac{\Theta(n)}{n^2} \sum_{|r| < \varepsilon n} \frac{r^2 p(r)}{2}  \Delta G (\chi^n_{r,x} ) - \kappa_{\gamma}  \Delta G( \tfrac{x}{n} ) \nonumber \\
 =& \kappa_{\gamma}   \Delta G ( \tfrac{x}{n} ) \Big( - 1 + \frac{\Theta(n)}{2 n^2 \kappa_{\gamma}} \sum_{|r| < \varepsilon n} c_{\gamma} r^{1-\gamma} \mathbbm{1}_{\{r \neq 0\}} \Big) + \frac{\Theta(n)}{2 n^2 } \sum_{|r| < \varepsilon n} c_{\gamma} r^{1-\gamma} \mathbbm{1}_{\{r \neq 0\}}   [\Delta G ( \chi^n_{r,x} ) -  \Delta G ( \tfrac{x}{n} ) ],
\end{align*}
where $\chi^n_{r,x}$ lies between $\tfrac{x}{n}$ and $\tfrac{x+r}{n}$. Therefore the display in \eqref{convprincnear0b} is bounded from above by
\begin{align}
& \limsup_{\varepsilon \rightarrow 0^{+}} \lim_{n \rightarrow \infty}  \frac{1}{n} \sum_{|x| \leq Bn}  \Big| \kappa_{\gamma}   \Delta G ( \tfrac{x}{n} ) \Big( - 1 + \frac{\Theta(n)}{2 n^2 \kappa_{\gamma}} \sum_{|r| < \varepsilon n} c_{\gamma} r^{1-\gamma} \mathbbm{1}_{\{r \neq 0\}} \Big)  \Big| \label{convprinc4dif} \\
+ & \limsup_{\varepsilon \rightarrow 0^{+}} \lim_{n \rightarrow \infty}  \frac{1}{n} \sum_{|x| \leq B n}  \Big| \frac{\Theta(n)}{2 n^2 } \sum_{|r| < \varepsilon n} c_{\gamma} r^{1-\gamma} \mathbbm{1}_{\{r \neq 0\}}   [\Delta G ( \chi^n_{r,x} ) -  \Delta G( \tfrac{x}{n} ) ]  \Big|. \label{convprinc5dif}
\end{align}
From \eqref{convkappagamma}, \eqref{convprinc4dif} is bounded from above by
\begin{align*}
& 3 B \| \Delta G \|_{\infty} \limsup_{\varepsilon \rightarrow 0^{+}} \lim_{n \rightarrow \infty}  \Big| - 1 + \frac{\Theta(n)}{2 n^2 \kappa_{\gamma}} \sum_{|r| < \varepsilon n} c_{\gamma} r^{1-\gamma} \Big| =0,
\end{align*}
for every $B >0$ fixed. Next we bound \eqref{convprinc5dif} from above by
\begin{align*}
 &  \limsup_{\varepsilon \rightarrow 0^{+}}\lim_{n \rightarrow \infty} \frac{1}{n} \sum_{|x| \leq B n}  \frac{\Theta(n)}{2 n^2 } \sum_{|r| < \varepsilon n} c_{\gamma} r^{1-\gamma} \mathbbm{1}_{\{r \neq 0\}} \sup_{ u,v  \in \mathbb{R}: |u- v| \leq \varepsilon} |\Delta G (v ) -  \Delta G ( u ) | \\
\lesssim & B \limsup_{\varepsilon \rightarrow 0^{+}} \sup_{ u,v  \in \mathbb{R}: |u- v| \leq \varepsilon} |\Delta G (v ) -  \Delta G ( u ) | =0,
\end{align*}
for every $B >0$ fixed. In last line we used the fact that $\Delta G \in \mathcal{S}(\mathbb{R})$, hence it is uniformly continuous. Then we conclude that
\begin{align*}
\forall B >0, \quad \limsup_{n \rightarrow \infty} \frac{1}{n} \sum_{ |x| \leq Bn } | \Theta(n) (\mathcal{K}_{n,\beta} G) (\tfrac{x}{n} ) - \kappa_{\gamma} \Delta G( \tfrac{x}{n} ) | =0,
\end{align*}
which leads to \eqref{convprincnear0}. It remains to prove \eqref{convprincfar0}. Since $G \in \mathcal{S}( \mathbb{R})$,  $\Delta G  \in \mathcal{S}(\mathbb{R}) \subset L^1(\mathbb{R})$ and we get
\begin{align*}
\limsup_{B \rightarrow \infty}  \limsup_{n \rightarrow \infty} \frac{1}{n} \sum_{ |x| > Bn } \big|  \Delta G( \tfrac{x}{n} ) \big| \lesssim \lim_{B \rightarrow \infty}   \int_{|u| \geq B} |\Delta G(u)| du =0,
 \end{align*}
then it is enough to prove that
\begin{align*}
\limsup_{B \rightarrow \infty}  \limsup_{n \rightarrow \infty} \frac{1}{n} \sum_{ |x| > Bn } |  \Theta(n) (\mathcal{K}_{n,\beta} G) (\tfrac{x}{n} ) | =0.
 \end{align*}
From \eqref{knbeta01}, we are done if we can prove that
 \begin{align} \label{beta01gamma2sum1}
\limsup_{B \rightarrow \infty}  \limsup_{n \rightarrow \infty} \frac{\Theta(n)}{n} \sum_{ |x| > n }  \big|  \sum_{|r| > |x| / 2} [ G( \tfrac{x+r}{n} ) - G( \tfrac{x}{n} )] p(r)  \big| =0,
 \end{align}
  \begin{align} \label{beta01gamma2sum2}
\limsup_{B \rightarrow \infty}  \limsup_{n \rightarrow \infty} \frac{\Theta(n)}{n} \sum_{ |x| > Bn }  \big|  \sum_{|r| \leq |x| / 2} [ G( \tfrac{x+r}{n} ) - G( \tfrac{x}{n} )] p(r)  \big|=0. 
 \end{align}
We begin with the former display. The expression inside the double limit in \eqref{beta01gamma2sum1} is bounded from above by a constant times
\begin{align*}
&  \frac{\Theta(n)}{n^{\gamma}} \frac{\| G \|_{\infty} }{n^2} \sum_{ |x| > Bn }  \sum_{r > |x| / 2}  \Big( \frac{r}{n} \Big)^{-\gamma-1} \lesssim  \frac{\Theta(n)}{n^{\gamma}} \frac{1}{n} \sum_{ |x| > Bn } \int_{ \frac{|x|}{2n} }^{\infty} u^{-\gamma-1} du \\
 \lesssim & \frac{\Theta(n)}{n^{\gamma}} \frac{1}{n} \sum_{ |x| \geq Bn } \Big( \frac{|x|}{n} \Big)^{-\gamma} \lesssim   \frac{\Theta(n)}{n^{\gamma}} \int_{B}^{\infty} u^{-\gamma} du \lesssim   B^{1-\gamma},
\end{align*}
that (since $\gamma  >1$) vanishes as $B \rightarrow \infty$, leading to \eqref{beta01gamma2sum1}. In order to prove \eqref{beta01gamma2sum2}, define $F_1: \mathbb{R} \rightarrow \mathbb{R}$ by
\begin{align*}
F_1(u):= \sup_{|u-v| \leq |u| /2} |  G''(v) |.
\end{align*}
Since $G \in \mathcal{S}(\mathbb{R})$, we have
\begin{equation} \label{limFaux1}
\limsup_{B \rightarrow \infty}  \frac{1}{n} \sum_{ |x| > Bn } F_1 ( \tfrac{x}{n} ) \Big( 1 + \frac{x^2}{n^2} \Big) \lesssim \int_{|u| \geq B} F_1(u)(1 + u^2) du =0.
\end{equation}
From the symmetry of $p$ and a Taylor expansion of second order on $G$, the expression inside the double limit in \eqref{beta01gamma2sum2} is bounded from above by
  \begin{align*}
  &   \frac{\Theta(n)}{n} \sum_{ |x| > B n } \sum_{|r| \leq |x| / 2} \frac{r^2}{2n^2} F_1 ( \tfrac{x}{n} ) p(r) \lesssim  \frac{\Theta(n)}{n^3} \sum_{ |x| > Bn } F ( \tfrac{z}{n} ) \sum_{r=1 }^{|x| / 2} r^{-1} \\
 \leq &  \frac{\Theta(n)}{n^3} \sum_{ |x| > Bn } F_1 ( \tfrac{x}{n} )[ 1 + \log( |x| )] \leq   \frac{1}{\log(n)}  \frac{1}{n} \sum_{ |x| > Bn } F ( \tfrac{x}{n} ) \Big( 1 + \log \Big( \frac{|x|}{n} \Big) + \log(n) \Big)  \\
 \lesssim &    \frac{1}{n} \sum_{ |x| > Bn } F_1 ( \tfrac{x}{n} ) \Big( 1 + \frac{x^2}{n^2} \Big) \lesssim \int_{|u| \geq B} F(u)(1 + u^2) du,
 \end{align*}
 which goes to zero as $B \rightarrow \infty$, due to \eqref{limFaux1}. Then we conclude that \eqref{convknbetasum} holds for $\gamma=2$. 
 
\textbf{III.:} Proof of \eqref{convknbetasum} for $\beta \geq 1$. In this case, from \eqref{op_Knb} we have  
\begin{equation*}
\mcb{K}_{n,\beta} G \left( \tfrac{x}{n} \right): =
\begin{cases}
 \sum_{y=0}^{\infty} \big[ G( \tfrac{y}{n} )  -  G( \tfrac{x}{n} ) - n^{-1} G_{+}^{'}(0) (y-x) \big]p(y-x),  x \geq 0; \\
  \sum_{y=-\infty}^{-1} \big[ G( \tfrac{y}{n} )  -  G( \tfrac{x}{n} ) - n^{-1} G_{-}^{'}(0) (y-x) \big]p(y-x),  \beta \geq 1, x \leq -1.
\end{cases}
\end{equation*}
Therefore we are done if we can prove the following results:
 \begin{align} \label{eqconvneu0a}
 \lim_{n \rightarrow \infty}   \Big| \frac{\Theta(n)}{n} \sum_{y=0 }^{\infty} \big[  [G_{+}(\tfrac{y}{n}) - G_{+}( \tfrac{0}{n})  ]     - \frac{1}{n}   G'_{+}( 0) (y-0) \big] p(y-0) - \frac{\kappa_{\gamma}}{n} \Delta G_{+} (\tfrac{0}{n} )     \Big|  =0,
\end{align}
\begin{align} \label{eqconvneu0b}
 \lim_{n \rightarrow \infty}  \sum_{x=1}^{\infty}  \Big| \frac{\Theta(n)}{n}  \big[  [G_{+}(\tfrac{0}{n}) - G_{+}( \tfrac{x}{n})  ] - \frac{1}{n} G'_{+}( 0) (0-x) \big]   p(0-x)    \Big|  =0,
\end{align}
\begin{align} \label{eqconvneupos}
 \lim_{n \rightarrow \infty}  \sum_{x=1}^{\infty}  \Big| \frac{\Theta(n)}{n} \sum_{y=1 }^{\infty}\big[  [G_{+}(\tfrac{y}{n}) - G_{+}( \tfrac{x}{n})  ] - \frac{1}{n} G'_{+}(0) (y-x) \big]   p(y-x) - \frac{\kappa_{\gamma}}{n} \Delta G_{+} (\tfrac{x}{n} )     \Big|  =0
\end{align}
and
\begin{align} \label{eqconvneuneg}
 \lim_{n \rightarrow \infty}  \sum_{x=-\infty}^{-1}  \Big| \frac{\Theta(n)}{n} \sum_{y=- \infty }^{-1} \big[  [G_{-}(\tfrac{y}{n}) - G_{-}( \tfrac{x}{n})  ] -\frac{1}{n} G'_{-}(0) (y-x) \big]    p(y-x) - \frac{\kappa_{\gamma}}{n} \Delta G_{-} (\tfrac{x}{n} )     \Big|  =0.
\end{align}
We begin by proving \eqref{eqconvneu0a}. We bound it from above by
\begin{align}
& \lim_{n \rightarrow \infty} \frac{\kappa_{\gamma}}{n} \| \Delta G_{+} \|_{\infty}  +  \lim_{n \rightarrow \infty}    \Big| \frac{\Theta(n)}{n} \sum_{y=0 }^{\infty} \big[  [G_{+}(\tfrac{y}{n}) - G_{+}(0 )  ]     - \frac{y}{n} G'_{+}( 0)  \big] p(y)      \Big| \nonumber  \\
=& \lim_{n \rightarrow \infty}    \Big| \frac{\Theta(n)}{n} \sum_{y=0 }^{\infty} \big[  [G_{+}(\tfrac{y}{n}) - G_{+}( 0)  ]     -  \frac{y}{n} G'_{+}( 0)  \big] p(y) \Big|. \label{eqconvneu0a1}    
\end{align}
If $\gamma=2$, then $G \in \mathcal{S}_{Neu}(\mathbb{R}^{*})$, $G'_{+}(0)=0$ and by performing a Taylor expansion of first order on $G_{+}$, \eqref{eqconvneu0a1} is bounded from above by
\begin{align*}
\lim_{n \rightarrow \infty}   \frac{\Theta(n)}{n^2} \sum_{y=0 }^{\infty} y p(y) \| G'_{+} \|_{\infty} \lesssim  \lim_{n \rightarrow \infty}   \frac{1}{\log(n)} \sum_{y=1 }^{\infty} y^{-2}  =0,
\end{align*}
since the sum above converges. If $\gamma > 2$, $\Theta(n)=n^2$ and by performing a Taylor expansion of second order on $G_{+}$, we bound \eqref{eqconvneu0a1} from above by
\begin{align*}
\lim_{n \rightarrow \infty}   \frac{\Theta(n)}{2n^3} \sum_{y=0 }^{\infty} y^2 p(y) \| \Delta G_{+} \|_{\infty} \lesssim  \lim_{n \rightarrow \infty}   \frac{1}{n} \sum_{y=1 }^{\infty} y^{1-\gamma}  =0,
\end{align*}
since the sum in last display converges. Analogously, performing a Taylor expansion of first order (resp. second order) on $G_{+}$, we conclude that the expression inside the limit in \eqref{eqconvneu0b} vanishes as $n \rightarrow \infty$ for $\gamma=2$ (resp. $\gamma > 2$). 

We observe that the proof of \eqref{eqconvneuneg} is analogous to the proof of \eqref{eqconvneupos}, then we will prove only \eqref{eqconvneupos}. For every $\varepsilon >0$, the expression inside the limit in \eqref{eqconvneupos} is bounded from above by the sum of
\begin{align} \label{eqconvneuposa}
 \limsup_{\varepsilon \rightarrow 0^{+}} \limsup_{n \rightarrow \infty} \sum_{x=1}^{\varepsilon n}  \Big| \frac{\kappa_{\gamma}}{n} \Delta G_{+} (\tfrac{x}{n} )     \Big|,
\end{align}
\begin{align} \label{eqconvneuposb}
\limsup_{\varepsilon \rightarrow 0^{+}} \limsup_{n \rightarrow \infty} \sum_{x=1}^{\varepsilon n}  \Big| \frac{\Theta(n)}{n} \sum_{y=1 }^{\infty}\big[  [G_{+}(\tfrac{y}{n}) - G_{+}( \tfrac{x}{n})  ] - \frac{y-x}{n} G'_{+}( 0)  \big]   p(y-x)      \Big|,
\end{align}
\begin{align} \label{eqconvneuposc}
\limsup_{\varepsilon \rightarrow 0^{+}} \limsup_{n \rightarrow \infty} \sum_{x=\varepsilon n +1}^{\infty} \frac{\Theta(n)}{n^2}  \Big| G'_{+}(0) \sum_{y=1 }^{\infty}    (y-x)    p(y-x)    \Big|
\end{align}
and
\begin{align} \label{eqconvneuposd}
\limsup_{\varepsilon \rightarrow 0^{+}} \limsup_{n \rightarrow \infty} \sum_{x=\varepsilon n +1}^{\infty} \Big| \frac{\Theta(n)}{n} \sum_{y=1 }^{\infty}   [G_{+}(\tfrac{y}{n}) - G_{+}( \tfrac{x}{n})  ]    p(y-x) - \frac{\kappa_{\gamma}}{n} \Delta G_{+} (\tfrac{x}{n} )     \Big|.
\end{align}
We begin by treating \eqref{eqconvneuposa}. We bound it from above by
\begin{align*}
\limsup_{\varepsilon \rightarrow 0^{+}} \limsup_{n \rightarrow \infty} \sum_{x=1}^{\varepsilon n}  \frac{\kappa_{\gamma}}{n} \| \Delta G_{+}  \|_{\infty}    = 0.
\end{align*}
Next we bound \eqref{eqconvneuposb} from above by the sum of
\begin{equation} \label{lemneu1posa}
\limsup_{\varepsilon \rightarrow 0^{+}} \limsup_{n \rightarrow \infty}   \sum_{x=1}^{\varepsilon n } \frac{\Theta(n)}{n} \Big|  \sum_{r=1-x}^{x-1} \Big[  [G_{+}( \tfrac{x+r}{n}) - G_{+}( \tfrac{x}{n}) ] - \frac{r}{n} G'_{+} (0) \Big]   p(r)  \Big| ,
\end{equation}
\begin{equation} \label{lemneu1posb}
\limsup_{\varepsilon \rightarrow 0^{+}} \limsup_{n \rightarrow \infty}  \sum_{x=1}^{\varepsilon n } \frac{\Theta(n)}{n} \Big| \sum_{r=x}^{\varepsilon n -1} \Big[  [G_{+}( \tfrac{x+r}{n}) - G_{+}( \tfrac{x}{n}) ]    - \frac{r}{n} G'_{+}(0) \Big] p(r)  \Big|  ,
\end{equation}
and
\begin{equation} \label{lemneu1posc}
\limsup_{\varepsilon \rightarrow 0^{+}} \limsup_{n \rightarrow \infty}  \sum_{x=1}^{\varepsilon n }  \frac{\Theta(n)}{n} \Big| \sum_{r= \varepsilon n}^{\infty} \Big[ [G_{+}( \tfrac{x+r}{n}) - G_{+}( \tfrac{x}{n}) ]    - \frac{r}{n} G'_{+} (0) \Big] p(r)  \Big|  .
\end{equation}
Due to the symmetry of $p$, by performing a Taylor expansion of second order on $G_{+}$ around $\frac{x}{n}$, the double limit in \eqref{lemneu1posa} is bounded from above by
\begin{align*}
& \limsup_{\varepsilon \rightarrow 0^{+}} \limsup_{n \rightarrow \infty} \frac{\Theta(n) \| \Delta G_{+} \|_{\infty}}{ n^3 }  \sum_{x=1}^{\varepsilon n } \sum_{r=1-x}^{x-1} r^2    p(r)  \lesssim \limsup_{\varepsilon \rightarrow 0^{+}} \limsup_{n \rightarrow \infty} \frac{\Theta(n)}{ n^{3}}  \sum_{x=1}^{\varepsilon n} \sum_{r=1}^{\varepsilon n} r^{1-\gamma} \\
=& \limsup_{\varepsilon \rightarrow 0^{+}} \limsup_{n \rightarrow \infty} \varepsilon \frac{\Theta(n)}{ n^{2}} \sum_{r=1}^{\varepsilon n} r^{1-\gamma} \lesssim \limsup_{\varepsilon \rightarrow 0^{+}} \varepsilon =0.
\end{align*}
From a Taylor expansion of first order on $G_{+}$, the double limit in \eqref{lemneu1posb} is bounded from above by
\begin{align*}
& \limsup_{\varepsilon \rightarrow 0^{+}} \limsup_{n \rightarrow \infty}   \frac{\Theta(n)}{  n^2}  \sum_{x=1}^{\varepsilon n } \sum_{r=x}^{\varepsilon n -1} r    p(r) \sup_{u \in  [0, 2 \varepsilon ] } | G'_{+} (u) - G'_{+} (0)  | \\
  \leq & \limsup_{\varepsilon \rightarrow 0^{+}} \limsup_{n \rightarrow \infty}   \frac{\Theta(n)}{  n^2} \sum_{x=1}^{\varepsilon n } r^2 p (r)  \sup_{ u \in  [0, 2 \varepsilon ] } |G'_{+} (u) - G'_{+} (0) |  \\
  \lesssim &  \lim_{\varepsilon \rightarrow 0^{+}}  \sup_{u \in  [0, 2 \varepsilon ] } | G'_{+} (u) -  G'_{+} (0)|  =0.
\end{align*}
The last limit is true due to the uniform continuity of $G'_{+}$,  since $G'_{+} \in \mathcal{S}(\mathbb{R})$. Above we used \eqref{thetan}. Finally \eqref{lemneu1posc} is bounded from above by a constant times
\begin{align*}
& \limsup_{\varepsilon \rightarrow 0^{+}} \limsup_{n \rightarrow \infty} \Big[  \| G \|_{\infty}  \sum_{x=1}^{\varepsilon n} \frac{\Theta(n)}{n} \sum_{r= \varepsilon n}^{\infty}   r^{-\gamma-1} + \sum_{x=1}^{\varepsilon n} \frac{\Theta(n)}{n^2} \sum_{r= \varepsilon n}^{\infty} c_{\gamma} r^{-\gamma} |  G'_{+} (0)| \Big]   \\
\lesssim & \limsup_{\varepsilon \rightarrow 0^{+}} \limsup_{n \rightarrow \infty} \varepsilon n  \frac{\Theta(n)}{n^{\gamma+1}}  \Big[    \frac{1}{n} \sum_{r= \varepsilon n}^{\infty}   \Big( \frac{r}{n} \Big)^{-\gamma-1} +     \frac{1}{n} \sum_{r= \varepsilon n}^{\infty}   \Big( \frac{r}{n} \Big)^{-\gamma} \Big] \\
\lesssim & \limsup_{\varepsilon \rightarrow 0^{+}} \limsup_{n \rightarrow \infty}  \frac{\Theta(n)}{n^{\gamma}}  \varepsilon^{1-\gamma} \big( 1 +  \varepsilon \big) =0.
\end{align*}
Above we used the fact that  that $\lim_{n \rightarrow \infty}  \Theta(n)/n^{\gamma}=0$; then we conclude that \eqref{eqconvneuposb} is equal to zero.

For $\gamma=2$, $ G'_{+} (0)=0$ and \eqref{eqconvneuposc} is zero. If $\gamma > 2$, since $p$ is symmetric, we rewrite \eqref{eqconvneuposc} as
\begin{align*}
\limsup_{\varepsilon \rightarrow 0^{+}} \limsup_{n \rightarrow \infty} \sum_{x=\varepsilon n +1}^{\infty}  \Big|  \sum_{r=x }^{\infty}   G'_{+}( 0) r    p(r)    \Big| \leq \| G'_{+} \|_{\infty} \limsup_{\varepsilon \rightarrow 0^{+}} \limsup_{n \rightarrow \infty} \sum_{x=\varepsilon n +1}^{\infty} \sum_{r=x }^{\infty} rp(r).
\end{align*}
For every $x \geq 1$, define $a_{x}:= \sum_{r=x }^{\infty} rp(r)$. Since $\gamma >2$, $\sum_{x=1}^{\infty} a_x = \sigma^2 /2 < \infty$ and we can rewrite last display as
\begin{align*}
\| G'_{+} \|_{\infty} \limsup_{\varepsilon \rightarrow 0^{+}} \limsup_{n \rightarrow \infty} \sum_{x=\varepsilon n +1}^{\infty} a_x = 0,
\end{align*}
since we deal with the tail of a convergent series. It remains to treat \eqref{eqconvneuposd}, which is bounded from above by the sum of
\begin{align} \label{eqconvneuposd1}
\limsup_{B \rightarrow \infty}  \limsup_{n \rightarrow \infty} \sum_{x= 2 B n  +1}^{\infty}  \Big| \frac{\Theta(n)}{n} \sum_{r=1-x }^{\infty}   [G_{+}(\tfrac{x+r}{n}) - G_{+}(\tfrac{x}{n})  ]    p(r) - \frac{\kappa_{\gamma}}{n} \Delta G_{+}''(\tfrac{x}{n})    \Big|,
\end{align}
\begin{align} \label{eqconvneuposd2}
\limsup_{B \rightarrow \infty} \limsup_{\varepsilon \rightarrow 0^{+}} \limsup_{n \rightarrow \infty} \sum_{x=\varepsilon n +1}^{2 B n} \frac{\Theta(n)}{n} \sum_{|r| \geq \varepsilon n } 2 \| G \|_{\infty}     p(r)   
\end{align}
and
\begin{align} \label{eqconvneuposd3}
\limsup_{B \rightarrow \infty} \limsup_{\varepsilon \rightarrow 0^{+}} \limsup_{n \rightarrow \infty} \sum_{x=\varepsilon n +1}^{2 B n}  \Big| \frac{\Theta(n)}{n} \sum_{r= 1- \varepsilon n }^{ \varepsilon n -1}   [G_{+}(\tfrac{x+r}{n}) - G_{+}( \tfrac{x}{n})  ]    p(r) - \frac{\kappa_{\gamma}}{n} \Delta G_{+} (\tfrac{x}{n} )     \Big|.
\end{align}
We bound \eqref{eqconvneuposd1} from above by the sum of
\begin{align} \label{eqconvneuposds1}
\limsup_{B \rightarrow \infty}  \limsup_{n \rightarrow \infty} \sum_{x= 2 B n  +1}^{\infty}  \Big|  \frac{\kappa_{\gamma}}{n} \Delta G_{+}(\tfrac{x}{n})    \Big|,
\end{align}
\begin{align} \label{eqconvneuposds2}
\limsup_{B \rightarrow \infty}  \limsup_{n \rightarrow \infty} \sum_{x= 2 B n  +1}^{\infty}   \frac{\Theta(n)}{n} \sum_{|r| >  x / 2  }   |G_{+}(\tfrac{x+r}{n}) - G_{+}(\tfrac{x}{n})  |    p(r)     ,
\end{align}
\begin{align} \label{eqconvneuposds3}
\limsup_{B \rightarrow \infty}  \limsup_{n \rightarrow \infty} \sum_{x= 2 B n  +1}^{\infty}  \Big| \frac{\Theta(n)}{n} \sum_{|r| \leq  x / 2 }   [G_{+}(\tfrac{x+r}{n}) - G_{+}(\tfrac{x}{n})  ]    p(r)    \Big|.
\end{align}
Since $\Delta G_{+} \in \mathcal{S}(\mathbb{R}) \subset L^1(\mathbb{R})$, \eqref{eqconvneuposds1} is bounded by a constant times $\limsup_{B \rightarrow \infty}  \int_{2B}^{\infty} | \Delta G_{+}(u)    | du =0$.
We bound \eqref{eqconvneuposds2} from above by a constant times
\begin{align*}
&  \limsup_{B \rightarrow \infty}  \limsup_{n \rightarrow \infty}  \|  G_{+} \|_{\infty} \frac{\Theta(n)}{n^{\gamma}} \frac{1}{n^2} \sum_{x= 2 B n }^{\infty}     \sum_{r =  x / 2 }^{\infty}    \Big( \frac{r}{n} \Big)^{-\gamma-1}\\
 \lesssim & \limsup_{B \rightarrow \infty}  \limsup_{n \rightarrow \infty}        \frac{\Theta(n)}{n^{\gamma}}  \int_{2B}^{\infty} \int_{  u / 2}^{\infty}     v^{-\gamma-1} dv du \lesssim  \limsup_{B \rightarrow \infty}  \limsup_{n \rightarrow \infty}   \frac{\Theta(n)}{n^{\gamma}} B^{1-\gamma} =0.
\end{align*}
From the symmetry of {$p(\cdot)$}, the expression in \eqref{eqconvneuposds3} can be rewritten as
\begin{align} 
& \limsup_{B \rightarrow \infty}  \limsup_{n \rightarrow \infty} \sum_{x= 2 B n  +1}^{\infty}  \Big| \frac{\Theta(n)}{n} \sum_{|r| \leq  x / 2 }   [G_{+}(\tfrac{x+r}{n}) - G_{+}(\tfrac{x}{n}) - rn^{-1} G'_{+}(\tfrac{x}{n})  ]    p(r)    \Big| \nonumber \\
\leq & \limsup_{B \rightarrow \infty}  \limsup_{n \rightarrow \infty} \sum_{x= 2 B n }^{\infty}   \frac{\Theta(n)}{n^3} \sum_{r \leq  x / 2 }    F_2(\tfrac{x}{n})   r^2p(r),  \label{eqconvneuposds4}
\end{align} 
where $F_2: \mathbb{R} \rightarrow \mathbb{R}$ is defined by $F_2(u):= \sup_{|v| \leq  |x| / 2} | \Delta G_{+} (u+v) |$. Since $\Delta G_{+} \in \mathcal{S}(\mathbb{R})$, we have that 
\begin{align} \label{Faux2}
\lim_{B \rightarrow \infty} \int_{2B}^{\infty} F_2(u) du  = \lim_{B \rightarrow \infty} \int_{2B}^{\infty} F_2(u) [1 + u^2] du =0.
\end{align}
If $\gamma > 2$, \eqref{eqconvneuposds4} is bounded from above by
\begin{align*}
& \limsup_{B \rightarrow \infty}  \limsup_{n \rightarrow \infty}  \frac{\Theta(n)}{n^3} \sum_{x= 2 B n }^{\infty}  F_2(\tfrac{x}{n})  \sum_{r  }       r^2p(r) = \limsup_{B \rightarrow \infty}  \limsup_{n \rightarrow \infty} \frac{\sigma^2}{n} \sum_{x= 2 B n }^{\infty}  F_2(\tfrac{x}{n})  \lesssim \lim_{B \rightarrow \infty} \int_{2B}^{\infty} F_2(u) du =0.
\end{align*}
Above we used \eqref{Faux2}. If $\gamma = 2$, the expression inside \eqref{eqconvneuposds4} is bounded from above by
\begin{align*}
&  \limsup_{B \rightarrow \infty}  \limsup_{n \rightarrow \infty}  \frac{\Theta(n)}{n^3} \sum_{x= 2 B n }^{\infty}  F_2(\tfrac{x}{n})  \sum_{r=1  }^{x}       r^2p(r) \lesssim \limsup_{B \rightarrow \infty}  \limsup_{n \rightarrow \infty}  \frac{1}{n \log(n)} \sum_{x= 2 B n }^{\infty}  F_2(\tfrac{x}{n}) [ 1 + \log( \tfrac{x}{n} ) + \log(n)] \\
  \lesssim &  \limsup_{B \rightarrow \infty}  \limsup_{n \rightarrow \infty} \frac{1}{n } \sum_{x= 2 B n }^{\infty}  F_2(\tfrac{x}{n}) [ 1 + ( \tfrac{x}{n} )^2  ]  \lesssim  \lim_{B \rightarrow \infty}     \int_{2B}^{\infty} F_2(u) [1 + u^2] du =0.
\end{align*}
Above we used \eqref{Faux2}. Next we bound \eqref{eqconvneuposd2} from above by a constant times
\begin{align*}
\limsup_{B \rightarrow \infty} \limsup_{\varepsilon \rightarrow 0^{+}} \limsup_{n \rightarrow \infty} \sum_{x=1}^{2 B n} \frac{\Theta(n)}{n^{\gamma+1}}  \frac{1}{n} \sum_{r = \varepsilon n }^{\infty}  \| G_{+} \|_{\infty}  \Big( \frac{r}{n} \Big)^{-\gamma} \lesssim \limsup_{B \rightarrow \infty} \limsup_{\varepsilon \rightarrow 0^{+}} \limsup_{n \rightarrow \infty} \frac{\Theta(n)}{n^{\gamma}}  B\varepsilon^{-\gamma}   =0. 
\end{align*}
It only remains to treat \eqref{eqconvneuposd3}. Performing a Taylor expansion of second order on $G_{+}$, for every $r,x \in \mathbb{Z}$, there exists $\bar{\xi}^n_{r,x}$ between $\frac{x+r}{n}$ and $\frac{x}{n}$ such that
\begin{align*}
 n[G_{+}( \tfrac{r+x}{n}) - G_{+}( \tfrac{x}{n}) ]  = r  G'_{+}( \tfrac{x}{n}) + \frac{r^2}{2n} \Delta  G_{+}( \bar{\xi}^n_{r,x}).
\end{align*}
Since $p(\cdot)$ is symmetric, $ G'_{+}( \tfrac{x}{n}) \sum_{|r| \leq \varepsilon n-1} r p(r) =0$ and we can rewrite \eqref{eqconvneuposd3} as
\begin{align}
&\limsup_{B \rightarrow \infty} \limsup_{\varepsilon \rightarrow 0^{+}} \limsup_{n \rightarrow \infty}   \Big| \frac{1}{n } \sum_{x=\varepsilon n+1}^{2 B n} \Big( \frac{\Theta(n)}{n^2} \sum_{r=-\varepsilon n +1}^{\varepsilon n -1} c_{\gamma} r^{1-\gamma} \frac{\Delta  G_{+}( \bar{\xi}^n_{r,x})}{2}   - \kappa_{\gamma} \Delta  G_{+} ( \tfrac{x}{n})  \Big)   \Big| \nonumber \\
\leq & \limsup_{B \rightarrow \infty} \limsup_{\varepsilon \rightarrow 0^{+}} \limsup_{n \rightarrow \infty}  \Big| \frac{1}{n} \sum_{x=\varepsilon n+1}^{2 B n}  \Delta  G_{+} ( \tfrac{x}{n})  \Big( \frac{\Theta(n)}{n^2} \sum_{r=1}^{\varepsilon n -1}   c_{\gamma} r^{1-\gamma}   - \kappa_{\gamma}   \Big) \Big| \label{eqconvneuposd31} \\
+ & \limsup_{ B \rightarrow \infty} \limsup_{\varepsilon \rightarrow 0^{+}} \limsup_{n \rightarrow \infty}  \Big| \frac{1}{n } \sum_{x=\varepsilon n+1}^{2 B n} \frac{\Theta(n)}{n^2}  \sum_{r=-\varepsilon n +1}^{\varepsilon n -1} c_{\gamma} r^{1-\gamma} \frac{\Delta  G_{+}( \bar{\xi}^n_{r,x})-  \Delta  G_{+} ( \tfrac{x}{n}) }{2}        \Big|. \label{eqconvneuposd32}
\end{align}
From \eqref{convkappagamma}, \eqref{eqconvneuposd31} is bounded from above by
\begin{align*}
\limsup_{B \rightarrow \infty} \limsup_{\varepsilon \rightarrow 0^{+}} \limsup_{n \rightarrow \infty}   \Big| \frac{1}{n} \sum_{x=1}^{2 B n} \| \Delta  G_{+} \|_{\infty} \Big|   \frac{\Theta(n)}{n^{1+\gamma}} \varepsilon^{1-\gamma} c_{\gamma} + \Big( \frac{\Theta(n)}{n^2} \sum_{r=1}^{\varepsilon n }   c_{\gamma} r^{1-\gamma}   - \kappa_{\gamma}   \Big) \Big| = 0.
\end{align*}
Finally we bound \eqref{eqconvneuposd32} from above by
\begin{align*}
 &\limsup_{B \rightarrow \infty} \limsup_{\varepsilon \rightarrow 0^{+}} \limsup_{n \rightarrow \infty}   \frac{1}{n} \sum_{x=1}^{2 B n}  \frac{\Theta(n)}{n^2} \sum_{r=1}^{\varepsilon n }   c_{\gamma} r^{1-\gamma} \sup_{ |u-v| \leq \varepsilon} |\Delta  G_{+}(v)-  \Delta  G_{+} ( u) | \\
\lesssim &\limsup_{B \rightarrow \infty}  \lim_{\varepsilon \rightarrow 0^{+}}  2 B \kappa_{\gamma} \sup_{ |u-v| \leq \varepsilon} |\Delta  G_{+}(v)-  \Delta  G_{+} (u) |  =0.
\end{align*}
In the last line, we used the uniform continuity of $\Delta  G_{+}$,  since $\Delta  G_{+} \in \mathcal{S}(\mathbb{R})$. 
\end{proof}
Next we state some auxiliary results that have been used along the article.
\begin{prop} \label{aux1}
It holds 
\begin{align} \label{sumslowbonds}
\sum_{\{x,y\} \in \mcb S} p(x-y) = 2m < \infty; \quad  \sum_{\{x,y\} \in \mcb S} |x-y|p(x-y) = \sigma^2 < \infty, \; \text{if} \; \gamma > 2.
\end{align}
Moreover, for every $C >0$ fixed, it holds
\begin{align} \label{convkappagamma}
\lim_{n \rightarrow \infty} \frac{\Theta(n)}{ n^2 } \sum_{r =1}^{C n} r^2 p(r)  = \lim_{n \rightarrow \infty} \frac{\Theta(n)}{ n^2 } \sum_{r =1}^{C n} c_{\gamma} r^{1-\gamma}  = \kappa_{\gamma}.
\end{align}
In particular, we get
\begin{align} \label{thetan}
\frac{\Theta(n)}{n^2}  \sum_{x=0}^{C n  -1} \sum_{y=- C n +1}^{-1}   (x-y) p (x-y) \leq  \frac{\Theta(n)}{n^2} \sum_{r=1}^{2 C n} r^2 p(r) \lesssim 1.
\end{align}
Finally,
\begin{equation} \label{limaux2}
\lim_{n \rightarrow \infty} \frac{\Theta(n)}{n^3} \sum_{x=0}^{C n - 1} \sum_{y= 1}^{C n-1} (x+y)^{1-\gamma} =0.
\end{equation}
\end{prop}
\begin{proof}
First we prove \eqref{sumslowbonds} using the symmetry of $p$. Since $\gamma >1$, we get
\begin{align*}
\sum_{\{x,y\} \in \mcb S} p(x-y) = 2 \sum_{x=0}^{\infty} \sum_{y=-\infty}^{-1} p (y-x) =2 \sum_{r=1}^{\infty} r p(r) = 2m < \infty.
\end{align*}
Moreover, if $\gamma >2$, we have
\begin{align*}
\sum_{\{x,y\} \in \mcb S} |y-x|p(y-x) = 2 \sum_{x=0}^{\infty} \sum_{y=-\infty}^{-1} (x-y) p (x-y) =2 \sum_{r=1}^{\infty} r^2 p(r) = \sigma^2 < \infty.
\end{align*}
Now we prove \eqref{convkappagamma}. For $\gamma > 2$, it holds
 \begin{align*}
\lim_{n \rightarrow \infty} \frac{\Theta(n)}{ n^2 } \sum_{r =1}^{C n} c_{\gamma} r^{1-\gamma} = \lim_{n \rightarrow \infty} \sum_{r =1}^{C n} r^2 p(r)  = \frac{1}{2} \sum_{r} r^2 p(r) = \frac{\sigma^2}{2} = \kappa_{\gamma}.
\end{align*}
Observe that
\begin{align*}
\log(Cn+1) = \sum_{r=1}^{Cn} \int_{r}^{r+1} u^{-1} du \leq \sum_{r=1}^{Cn} r^{-1} = \leq 1 + \sum_{r=2}^{Cn} \int_{r-1}^{r} u^{-1} du = 1+ \log(Cn).
\end{align*}
Then for $\gamma=2$ we get
\begin{align*}
c_2 \frac{\log(C n + 1)}{\log(n)} \leq \frac{\Theta(n)}{ n^2 } \sum_{r =1}^{C n} c_{\gamma} r^{1-\gamma}  =  \frac{c_{2}}{\log(n) } \sum_{r=1}^{C n}  r^{-1} \leq  c_2 \Big[ 1 + \frac{1+\log(C)}{\log(n)} \Big], \forall n \geq 1.
\end{align*}
When $n \rightarrow \infty$, the left-hand and right-hand sides of the display above both go to $c_2=\kappa_{\gamma}$, proving \eqref{convkappagamma}. It remains to prove \eqref{limaux2}. We bound it from above by
\begin{align} \label{eqdif2}
 \frac{\Theta(n)}{n^3} \sum_{x=1}^{C n }  x^{1-\gamma} + \frac{\Theta(n)}{n^3} \sum_{x=1}^{C n }   \int_{x}^{x+ Cn} u^{1-\gamma} du. 
\end{align}
Let us treat the leftmost term in \eqref{eqdif2}. For $\gamma > 2$, it is equal to
\begin{align*}
 \frac{n^2}{n^3} \sum_{x=1}^{C n }  x^{1-\gamma} \leq \frac{1}{n} \sum_{x=1}^{\infty}x^{1-\gamma} \lesssim \frac{1}{n}, 
\end{align*}
that vanishes as $n \rightarrow \infty$. For $\gamma=2$, it is equal to
\begin{align*}
 \frac{n^2 \log(n)}{n^3}  \sum_{x=0}^{C n - 1}  (x+1)^{1-2}  \leq  \frac{1}{n \log(n)} \sum_{x=1}^{C n}  x^{1-\gamma} \leq  \frac{1+ \log (C n)}{n \log(n)},
\end{align*}
and this also vanishes as $n \rightarrow \infty$. Now let us treat the rightmost term in \eqref{eqdif2}. For $\gamma >2$, it is equal to
\begin{align*}
 \frac{n^2}{n^3} \sum_{x=1}^{C n}   \int_{x}^{x+ Cn } u^{1-\gamma} \lesssim \frac{1}{n} \sum_{x=1}^{C n} x^{2-\gamma}.
\end{align*}
If $\gamma \in (2,3)$, we have
\begin{align*}
\frac{1}{n} \sum_{x=1}^{C n} x^{2-\gamma} = n^{2-\gamma} \frac{1}{n} \sum_{x=1}^{C n} \Big( \frac{x}{n} \Big)^{2-\gamma} \lesssim n^{2-\gamma} \int_0^C u^{2-\gamma} du \lesssim n^{2-\gamma}.
\end{align*}
If $\gamma =3$, it holds
\begin{align*}
\frac{1}{n} \sum_{x=1}^{C n} x^{2-\gamma} =  \frac{1}{n} \sum_{x=1}^{C n} x^{-1} \leq \frac{1+\log(C n)}{n}.
\end{align*}
And if $\gamma > 3$, we get
\begin{align*}
\frac{1}{n} \sum_{x=1}^{C n} x^{2-\gamma} \leq  \frac{1}{n} \sum_{x=1}^{\infty} x^{2-\gamma} \lesssim \frac{1}{n},
\end{align*}
and we conclude that the second term in \eqref{eqdif2} vanishes as $n \rightarrow \infty$ for $\gamma >2$. Finally, it remains to treat it for $\gamma=2$. In this case, it can be rewritten as
\begin{align*} 
 &\frac{n^2}{n^3 \log(n)} \sum_{x=1}^{C n}   \int_{x}^{x+ Cn } u^{1-2} du = \frac{1}{n \log(n)}  \sum_{x=1}^{C n} [ \log(x+Cn)   - \log(x)   ] \\
= &  \frac{1}{n \log(n)}  \sum_{x=Cn + 1}^{2 C n} \log ( \tfrac{x}{n}   ) +  \frac{1}{ n\log(n)}  \sum_{x=1}^{\varepsilon n } \log ( \tfrac{n}{x}   ) +  \frac{1}{ n\log(n)}  \sum_{x=\varepsilon n+1}^{C n} \log ( \tfrac{n}{x}   )   \\
 \lesssim&   \frac{1}{  \log(n)} \int_{C}^{2C} \log(u) du  +  \frac{1}{ n\log(n)}  \sum_{x=1}^{\varepsilon n} \log ( n  ) + \frac{C \log(\varepsilon^{-1})  }{\log(n)}, 
\end{align*}
and this vanishes as $n \rightarrow \infty$ and then $\varepsilon \rightarrow 0^+$, ending the proof.
\end{proof}
Next result was useful to prove Lemma \ref{lemconvmartterm1}. Recall from \eqref{op_Bnb} the definition of $\mathcal{B}_{n,\beta} (G)$.
\begin{prop} \label{prop2lem1convmart}
Let $(\beta,\gamma) \in R_0$. For every $G \in \mathcal{S}_{\beta,\gamma}$, we have
\begin{align*}
\lim_{n \rightarrow \infty} \mathcal{B}_{n,\beta} (G)= \mathbbm{1}_{ \{ \gamma > 2\}} \mathbbm{1}_{ \{ \beta=1 \} } \frac{2 \kappa_{\gamma}}{\hat{\alpha}}  [ \nabla_{\beta,\gamma} G(0)]^2.
\end{align*}
\end{prop}
\begin{proof}
we treat two cases separately: $\beta \in [0,1)$ and $\beta \geq 1$. 

\textbf{I.} In the case $\beta \in [0,1)$, $G \in \mathcal{S}_{\beta,\gamma}=\mathcal{S}( \mathbb{R})$ and
\begin{align*}
\mathcal{B}_{n,\beta} (G) = ( \alpha n^{-\beta} -1) \frac{\Theta(n)}{n} \sum_{   \{x,y\} \in \mcb S } p(y-x) [G( \tfrac{y}{n}) - G( \tfrac{x}{n})]^2.
\end{align*}
Since $|\alpha n^{-\beta}-1| \leq \alpha + 1$, from the symmetry of $p$ it is enough to prove that
\begin{align}
&  \frac{\Theta(n)}{n} \sum_{x=n}^{\infty} \sum_{y=-\infty}^{-1} p(y-x) |G( \tfrac{y}{n}) - G( \tfrac{x}{n})|^2 + \frac{\Theta(n)}{n} \sum_{x=0}^{n-1} \sum_{y=-\infty}^{-n} p(y-x) |G( \tfrac{y}{n}) - G( \tfrac{x}{n})|^2 \label{prop2lem1convmart1}  \\
+ &   \frac{\Theta(n)}{n}\sum_{x=0}^{n-1} \sum_{y=-n+1}^{-1} p(y-x) |G( \tfrac{y}{n}) - G( \tfrac{x}{n})|^2   \label{prop2lem1convmart2}
\end{align}
goes to zero as $n \rightarrow \infty$. Since $\gamma \geq 2 >1$, the first term in \eqref{prop2lem1convmart1} is bounded from above by a constant times
\begin{align*}
 (  \| G \|_{\infty})^{2}  \frac{\Theta(n)}{n^{\gamma}} \frac{1}{n^2} \sum_{x=n}^{\infty} \sum_{y=-\infty}^{-1} \Big( \frac{x-y}{n} \Big)^{-\gamma-1} \lesssim \frac{\Theta(n)}{n^{\gamma}} \int_{1}^{\infty} \int_{-\infty}^{0} (u-v)^{-\gamma-1} dv du \lesssim \frac{\Theta(n)}{n^{\gamma}}.
\end{align*}
An analogous procedure leads to an upper bound of same order to the second term in \eqref{prop2lem1convmart1}. From a Taylor expansion of first order on $G$, the term in \eqref{prop2lem1convmart2} is bounded from above by
\begin{align*}
\frac{\Theta(n)}{n}\sum_{x=0}^{n-1} \sum_{y=-n+1}^{-1}  p(y-x) ( \| G' \|_{\infty})^2 \Big( \frac{x-y}{n} \Big)^2  \lesssim \frac{\Theta(n)}{n^3} \sum_{x=0}^{ n - 1} \sum_{y= 1}^{ n-1} (x+y)^{1-\gamma},
\end{align*}
which goes to zero as $n \rightarrow \infty$, due to \eqref{limaux2}.

\textbf{II.} In the case $\beta\geq 1$ we have
\begin{align*}
\mathcal{B}_{n,\beta} (G) =\alpha \Theta(n) n^{-1-\beta} \sum_{   \{x,y\} \in \mcb S }    p(y-x) [G( \tfrac{y}{n}) - G( \tfrac{x}{n})]^2.
\end{align*}
First we treat the case $(\beta, \gamma) \in \big( [1, \infty) \times \{2 \} \big) \cup \big( (1,\infty) \times (2, \infty) \big)$, where $G \in \mathcal{S}_{\beta,\gamma} = \mathcal{S}_{Neu}(\mathbb{R}^{*})$ and $G'_{-}(0)=G'_{+}(0)=0$. From \eqref{sumslowbonds} $\mathcal{B}_{n,\beta} (G)$ is bounded from above by a constant times
\begin{align*}
 (  \| G \|_{\infty})^{2}  \frac{\Theta(n)}{n^{1+\beta}} \sum_{\{x, y\} \in \mcb S} p(y-x) \lesssim \frac{\Theta(n)}{n^{1+\beta}} m \lesssim \frac{\Theta(n)}{n^{1+\beta}},
\end{align*}
which goes to zero as $n \rightarrow \infty$. It remains to treat the case $(\beta, \gamma) \in \{1\} \times (2,\infty)$, where $G \in \mathcal{S}_{\beta,\gamma} = \mathcal{S}_{Rob}(\mathbb{R}^{*})$. Then $\mathcal{B}_{n,\beta} (G)$ can be rewritten as
\begin{align*}
2 \alpha  \sum_{   x=0 }^{\infty} \sum_{y=-\infty}^{-1}    p(y-x) [G( \tfrac{y}{n}) - G( \tfrac{x}{n})]^2. 
\end{align*}
Above we used the symmetry of $p$. Last display is equal to
\begin{align}
&4 \alpha [G_{-}(0) - G_{+}(0)]  \sum_{   x=0 }^{\infty} \sum_{y=-\infty}^{-1}    p(y-x) \big[ [G_{-}( \tfrac{y}{n}) - G_{-}( 0)] + [G_{+}( 0) - G_{+}( \tfrac{x}{n})] \big] \label{prop2lem1tightmarta} \\
+& 2 \alpha  \sum_{   x=0 }^{\infty} \sum_{y=-\infty}^{-1}    p(y-x) \big[ [G_{-}( \tfrac{y}{n}) - G_{-}( 0)] + [G_{+}( 0) - G_{+}( \tfrac{x}{n})] \big]^2 \label{prop2lem1tightmartb} \\
+ & 2 \alpha [G_{-}(0) - G_{+}(0)]^2  \sum_{   x=0 }^{\infty} \sum_{y=-\infty}^{-1}    p(y-x). \label{prop2lem1tightmartc}
\end{align}
We claim that \eqref{prop2lem1tightmarta} and \eqref{prop2lem1tightmartb} vanish as $n \rightarrow \infty$. First we treat the former. From Taylor expansions on $G_{-}$ and $G_{+}$, \eqref{prop2lem1tightmarta} is bounded from above by a constant times
\begin{align*}
\frac{1}{n} \sum_{   x=0 }^{\infty} \sum_{y=-\infty}^{-1}    p(y-x) \big[ (-y) \|  G'_{-} \|_{\infty}  + x \|  G'_{+}  \|_{\infty}  \big] \lesssim \frac{1}{n} \sum_{\{x,y\} \in \mcb S} |x-y| p(x-y) \lesssim n^{-1},
\end{align*}  
and \eqref{prop2lem1tightmarta} goes to zero as $n \rightarrow \infty$. In last bound we used \eqref{sumslowbonds}. Now we rewrite \eqref{prop2lem1tightmartb} as
\begin{align}
& 2 \alpha  \sum_{   x=n }^{\infty} \sum_{y=-\infty}^{-1}    p(y-x) \big[ [G_{-}( \tfrac{y}{n}) - G_{-}( 0)] + [G_{+}( 0) - G_{+}( \tfrac{x}{n})] \big]^2 \label{prop2lem1tightmartd} \\
+& 2 \alpha  \sum_{   x=0 }^{n-1} \sum_{y=-\infty}^{-n}    p(y-x) \big[ [G_{-}( \tfrac{y}{n}) - G_{-}( 0)] + [G_{+}( 0) - G_{+}( \tfrac{x}{n})] \big]^2 \label{prop2lem1tightmarte} \\
+& 2 \alpha  \sum_{   x=0 }^{n-1} \sum_{y=-n+1}^{-1}    p(y-x) \big[ [G_{-}( \tfrac{y}{n}) - G_{-}( 0)] + [G_{+}( 0) - G_{+}( \tfrac{x}{n})] \big]^2 \label{prop2lem1tightmartf}.
\end{align}
Since $G_{-},G_{+}$ are bounded, \eqref{prop2lem1tightmartd} is bounded from above by a constant times
\begin{align*}
 \sum_{   x=n }^{\infty} \sum_{y=-\infty}^{-1}    p(y-x) \lesssim n^{1-\gamma} \frac{1}{n^2} \sum_{   x=n }^{\infty} \sum_{y=-\infty}^{-1} \Big( \frac{x-y}{n} \Big)^{-\gamma-1} \lesssim n^{1-\gamma} \int_{1}^{\infty} \int_{-\infty}^{0} (u-v)^{-\gamma-1} dv du \lesssim n^{1-\gamma},
\end{align*}
and we conclude that \eqref{prop2lem1tightmartd} goes to zero as $n \rightarrow \infty$. With an analogous reasoning, we conclude that the same holds for \eqref{prop2lem1tightmarte}. In order to treat \eqref{prop2lem1tightmartf}, choose $\delta=1$ when $\gamma>3$ and $\delta=\frac{\gamma}{2}-1$ when $\gamma \in (2,3]$. In particular, $\delta \in (0,1] \cap (0, \frac{\gamma-1}{2})$. Since $G_{-},G_{+} \in \mathcal{S}(\mathbb{R})$, then they are globally Lipschitz and $\delta$- H\"older on the compact $[-1,1]$. Therefore there exists $C_{\delta}$ such that
\begin{align*}
|G_{-}(u) - G_{-}(v)| + |G_{+}(u) - G_{+}(v)| \leq C_{\delta} |u-v|^{\delta}, \forall u,v \in [-1,1].
\end{align*}
Then \eqref{prop2lem1tightmartf} is bounded from above by a constant times
\begin{align*}
n^{-2 \delta} \sum_{   x=0 }^{n-1} \sum_{y=-n+1}^{-1}    p(y-x)  \big[ |y|^{ \delta} + x^{ \delta}  \big]^2 \lesssim n^{-2 \delta} \sum_{   x=0 }^{\infty} \sum_{y=-\infty}^{-1} (x-y)^{2 \delta -1 - \gamma}.
\end{align*}
Since $\delta \in (0, \frac{\gamma-1}{2})$, the double sum in last display is convergent and we conclude that \eqref{prop2lem1tightmartf} goes to zero as $n \rightarrow \infty$. Finally from the symmetry of $p$, \eqref{sumslowbonds} and \eqref{defkdifgamma} we rewrite \eqref{prop2lem1tightmartc} as
\begin{align*}
2 \alpha m [G_{+}(0) - G_{-}(0)]^2  =  \frac{\alpha m}{\hat{\alpha}}  \frac{2}{\hat{\alpha}}  [ \nabla_{\beta,\gamma} G(0)]^2 = \frac{2 \kappa_{\gamma}}{\hat{\alpha}}  [ \nabla_{\beta,\gamma} G(0)]^2,
\end{align*}
ending the proof. Above we used the fact that $\hat{\alpha} [G_{+}(0) - G_{-}(0)] = \nabla_{\beta,\gamma} G(0)$, since $G \in \mathcal{S}_{Rob}(\mathbb{R}^{*})$.
\end{proof}
Now we present a result which was useful in the case $\beta=1$ and $\gamma>2$ in order to prove Proposition \ref{prop1lem1convmart}.
\begin{lem} \label{prop1alem1convmart}
Let $(\beta,\gamma) \in \{1\} \times (2,\infty)$ and $G \in \mathcal{S}_{\beta,\gamma}=\mathcal{S}_{Rob}(\mathbb{R}^{*})$. It holds
\begin{align*}
\lim_{n \rightarrow \infty} \Big\{  G_{+}'(0)  \sum_{x=0}^{\infty} G( \tfrac{x}{n}) \sum_{r=x+1}^{\infty} r  p(r) +  G_{-}'(0)  \sum_{x=-\infty}^{-1} G( \tfrac{x}{n}) \sum_{r=-\infty}^{x} r  p(r) \Big\} =\frac{ \kappa_{\gamma}}{\hat{\alpha}}  [ \nabla_{\beta,\gamma} G(0)]^2.
\end{align*}
\end{lem}
\begin{proof}
We begin by claiming that for every $\varepsilon >0$, it holds
\begin{align}
& \lim_{n \rightarrow \infty} \Big|  G_{+}'(0)  \sum_{x=\varepsilon n}^{\infty} G( \tfrac{x}{n}) \sum_{r=x+1}^{\infty} r  p(r) \Big| =0, \label{eqcrit1a} \\
& \lim_{n \rightarrow \infty} \Big| G_{-}'(0)  \sum_{x=-\infty}^{-\varepsilon n} G( \tfrac{x}{n}) \sum_{r=-\infty}^{x} r  p(r) \Big| =0. \label{eqcrit1b}
\end{align}
We prove only \eqref{eqcrit1a} but we observe that the proof of \eqref{eqcrit1b} is analogous. We bound the expression inside the limit in \eqref{eqcrit1a} from above by
\begin{align*}
\|G_{+}' \|_{\infty} \|G \|_{\infty}  \sum_{x=\varepsilon n}^{\infty}  \sum_{r=x+1}^{\infty} r  p(r) \lesssim  \sum_{x=\varepsilon n}^{\infty}  \sum_{r=x+1}^{\infty} r  p(r),
\end{align*}
Then \eqref{eqcrit1a} follows from the fact that last sum is the tail of a convergent series, since $\gamma>2$ and
\begin{align} \label{asterisco}
 \sum_{x=0}^{\infty} \sum_{r=x+1}^{\infty} r  p(r) = \sum_{r=1}^{\infty} r^2 p(r) = \frac{\sigma^2}{2} < \infty.
\end{align}
Now we claim that 
\begin{align}
& \limsup_{\varepsilon \rightarrow 0^{+}} \limsup_{n \rightarrow \infty} \Big|  G_{+}'(0)  \sum_{x=0}^{\varepsilon n -1} [G_{+}( 0) - G_{+}( \tfrac{x}{n})]  \sum_{r=x+1}^{\infty} r  p(r) \Big| =0, \label{eqcrit1c} \\
& \limsup_{\varepsilon \rightarrow 0^{+}} \limsup_{n \rightarrow \infty} \Big|  G_{-}'(0)  \sum_{x=- \varepsilon n+1}^{-1} [G_{-}( 0) - G_{-}( \tfrac{x}{n})] \sum_{r=-\infty}^{x} r  p(r) \Big| =0. \label{eqcrit1d}
\end{align}
We prove only \eqref{eqcrit1c} but the proof of \eqref{eqcrit1d} is analogous. We bound the expression in \eqref{eqcrit1c} from above by
\begin{align*}
 & \limsup_{\varepsilon \rightarrow 0^{+}} \limsup_{n \rightarrow \infty}  \|G_{+}^{'} \|_{\infty}  \sum_{x=0}^{\varepsilon n -1}   \sup_{u \in [0,\varepsilon]}|G_{+}( 0) - G_{+}( u)|  \sum_{r=x+1}^{\infty} r  p(r) \\
 \lesssim & \limsup_{\varepsilon \rightarrow 0^{+}}  \sup_{u \in [0,\varepsilon]}|G_{+}( 0) - G_{+}( u)| \limsup_{n \rightarrow \infty}   \sum_{x=0}^{\infty}    \sum_{r=x+1}^{\infty} r  p(r) \\
  =&\frac{\sigma^2}{2} \limsup_{\varepsilon \rightarrow 0^{+}}  \sup_{u \in [0,\varepsilon]}|G_{+}( 0) - G_{+}( u)|     =0.
\end{align*}
In first equality (resp. last equality) we used \eqref{asterisco} (resp. used the continuity of $G_{+}$). From \eqref{eqcrit1c} and \eqref{eqcrit1d}, we get
\begin{align*}
& \limsup_{\varepsilon \rightarrow 0^{+}} \limsup_{n \rightarrow \infty} \Big\{  G_{+}'(0)  \sum_{x=0}^{\varepsilon n -1} G( \tfrac{x}{n}) \sum_{r=x+1}^{\infty} r  p(r) + G_{-}'(0)  \sum_{x=- \varepsilon n+1}^{-1} G( \tfrac{x}{n}) \sum_{r=-\infty}^{x} r  p(r) \Big\} \\
=& \limsup_{\varepsilon \rightarrow 0^{+}} \limsup_{n \rightarrow \infty} \Big\{  G_{+}'(0)  \sum_{x=0}^{\varepsilon n -1} G_{+}( 0) \sum_{r=x+1}^{\infty} r  p(r) + G_{-}'(0)  \sum_{x=- \varepsilon n+1}^{-1} G_{-}( 0) \sum_{r=-\infty}^{x} r  p(r) \Big\} \\
=&  G_{+}'(0) G_{+}( 0) \frac{\sigma^2}{2} -   G_{-}'(0) G_{-}( 0) \frac{\sigma^2}{2} = \kappa_{\gamma} [ G_{+}'(0) G_{+}( 0) - G_{-}'(0) G_{-}( 0) ] =\frac{ \kappa_{\gamma}}{\hat{\alpha}}  [ \nabla_{\beta,\gamma} G(0)]^2.
\end{align*}
In last line we used \eqref{defkdifgamma} and the fact that $G \in \mathcal{S}_{Rob}(\mathbb{R}^{*})$. Combining this with \eqref{eqcrit1a} and \eqref{eqcrit1b}, we have the desired result.
\end{proof}
Next result was useful to prove Lemma \ref{lemconvmartterm1}. Recall the definition of $\mathcal{A}_{n,\beta} (G)$ from \eqref{op_Anb}.
\begin{prop} \label{prop1lem1convmart}
Let $(\beta,\gamma) \in R_0$. For every $G \in \mathcal{S}_{\beta,\gamma}$, it holds
\begin{align*} 
\lim_{n \rightarrow \infty} \mathcal{A}_{n,\beta} (G)  = 2 \kappa_{\gamma} \int_{\mathbb{R}} [\nabla_{\beta,\gamma} G(u)]^2 du + 2  \frac{ \kappa_{\gamma}}{\hat{\alpha}}  \mathbbm{1}_{\{ \gamma > 2 \} } \mathbbm{1}_{\{ \beta = 1 \} }  [ \nabla_{\beta,\gamma} G(0)]^2 = \| \nabla_{\beta,\gamma} G \|^2_{\infty} .
\end{align*} 
\end{prop}
\begin{proof}
First we treat the case $\beta \in [0,1)$. From the symmetry of $p$ and performing a change of variables, we can rewrite $\mathcal{A}_{n,\beta} (G)$ as
\begin{align}
  & \frac{\Theta(n)}{n}   \sum_{x,y} p(y-x) [G( \tfrac{x}{n})]^2 + \frac{\Theta(n)}{n}   \sum_{x,y } p(x-y) [ G( \tfrac{x}{n})]^2 - 2 \frac{\Theta(n)}{n}   \sum_{x,y } p(y-x) G( \tfrac{x}{n})  G( \tfrac{y}{n}) \nonumber  \\
 =&   - \frac{2}{n}  \sum_{x } G( \tfrac{x}{n}) \Big[ \Theta(n) \sum_{y }  p(y-x)   [G( \tfrac{y}{n}) - G( \tfrac{x}{n})] \Big] \nonumber \\
 =& - \frac{2}{n} \sum_{x } G( \tfrac{x}{n}) \Big[\Theta(n)  ( \mcb{K}_{n, \beta} G) ( \tfrac{x}{n}) - \kappa_{\gamma} \Delta_{\beta,\gamma} G( \tfrac{x}{n}) \Big] - \frac{2 \kappa_{\gamma}}{n} \sum_{x } G( \tfrac{x}{n}) \Delta_{\beta,\gamma} G( \tfrac{x}{n}). \label{eqprop1lem1convmarta}
\end{align}
In an analogous way, for $\beta \geq 1$, performing a change of variables we can rewrite $\mathcal{A}_{n,\beta} (G)$ as
\begin{align}
 &   - \frac{2}{n}  \sum_{x } G( \tfrac{x}{n}) \Big[ \Theta(n) \sum_{y: \{x,y\} \in \mcb F }  p(y-x)   [G( \tfrac{y}{n}) - G( \tfrac{x}{n})] \Big] \nonumber \\
 =& - \frac{2}{n} \sum_{x } G( \tfrac{x}{n}) \Big[\Theta(n)  ( \mcb{K}_{n, \beta} G) ( \tfrac{x}{n}) - \kappa_{\gamma} \Delta_{\beta,\gamma} G( \tfrac{x}{n}) \Big] - \frac{2 \kappa_{\gamma}}{n} \sum_{x } G( \tfrac{x}{n}) \Delta_{\beta,\gamma} G( \tfrac{x}{n}) \label{eqprop1lem1convmartb} \\
 - & 2 \Big\{ G_{+}'(0)  \sum_{x=0}^{\infty} G( \tfrac{x}{n}) \sum_{r=x+1}^{\infty} r  p(r) +  G_{-}'(0)  \sum_{x=-\infty}^{-1} G( \tfrac{x}{n}) \sum_{r=-\infty}^{x} r  p(r) \Big\}  \label{eqprop1lem1convmartc}.
\end{align}
For $\beta \geq 1$, we claim that \eqref{eqprop1lem1convmartc} converges, as $n \rightarrow \infty$, to
\begin{align} \label{eqprop1lem1convmartc2}
 - 2  \frac{ \kappa_{\gamma}}{\hat{\alpha}}  \mathbbm{1}_{\{ \gamma > 2 \} } \mathbbm{1}_{\{ \beta = 1 \} }  [ \nabla_{\beta,\gamma} G(0)]^2.
\end{align}
Indeed, if $(\beta,\gamma) \in \{1\} \times (2,\infty)$, this convergence is a direct consequence of Lemma \ref{prop1alem1convmart}. Otherwise, we get $G \in  \mathcal{S}_{Neu}(\mathbb{R}^{*})$, then $G'_{-}(0)=G'_{+}(0)=0$ and \eqref{eqprop1lem1convmartc} is equal to zero. 

Next, we bound the absolute value of the first term in both \eqref{eqprop1lem1convmarta} and \eqref{eqprop1lem1convmartb} from above by
\begin{align*}
   \frac{2}{n} \sum_{x } |G( \tfrac{x}{n})|  \big| \Theta(n) ( \mathcal{K}_{n, \beta} G) ( \tfrac{x}{n}) - \kappa_{\gamma} \Delta_{\beta,\gamma} G( \tfrac{x}{n})  \big|  \leq     \frac{2\|G\|_{\infty}}{n}  \sum_{x }   \big| \Theta(n) ( \mathcal{K}_{n, \beta} G) ( \tfrac{x}{n}) - \kappa_{\gamma}  \Delta_{\beta,\gamma} G( \tfrac{x}{n})  \big|,
\end{align*}
which goes to zero as $n \rightarrow \infty$, from Proposition \ref{convknbeta}. For $\beta <1$, it remains to treat the second term in \eqref{eqprop1lem1convmarta}; we do so by observing that $G \in \mathcal{S}_{\beta,\gamma}=\mathcal{S}(\mathbb{R})$ and performing an integration by parts:
\begin{align*}
& \lim_{n \rightarrow \infty} \Big[ -  \frac{2 \kappa_{\gamma}}{n} \sum_{x } G( \tfrac{x}{n}) \Delta_{\beta,\gamma} G( \tfrac{x}{n}) \Big] = - 2 \kappa_{\gamma} \int_{\mathbb{R}} G(u) \Delta_{\beta,\gamma} G(u) du  =  2 \kappa_{\gamma} \int_{\mathbb{R}} [\nabla_{\beta,\gamma} G(u)]^2 du.
\end{align*}
Finally, for $\beta \geq 1$, we treat the second term in \eqref{eqprop1lem1convmartb} by performing two integrations by parts:
\begin{align*}
& \lim_{n \rightarrow \infty} \Big[ -  \frac{2 \kappa_{\gamma}}{n} \sum_{x } G( \tfrac{x}{n}) \Delta_{\beta,\gamma} G( \tfrac{x}{n}) \Big] = - 2 \kappa_{\gamma} \Big[ \int_{-\infty}^{0} G_{-}(u) \Delta_{\beta,\gamma} G_{-}(u) du  + \int_{0}^{\infty} G_{+}(u) \Delta_{\beta,\gamma} G_{+}(u) du \Big] \\
=&  2 \kappa_{\gamma} \Big[ G_{+}(0) G_{+}'(0) - G_{-}(0) G_{-}'(0) +   \int_{-\infty}^{0} [G_{-}'(u)]^2  du  + \int_{0}^{\infty} [G_{+}'(u)]^2 du \Big]  \\
=&  2 \kappa_{\gamma}   \int_{\mathbb{R}} [\nabla_{\beta,\gamma} G(u)]^2 du + 2 \kappa_{\gamma} [G_{+}(0) G_{+}'(0) - G_{-}(0) G_{-}'(0)]    \\
=&  2 \kappa_{\gamma}   \int_{\mathbb{R}} [\nabla_{\beta,\gamma} G(u)]^2 du + 2  \frac{ \kappa_{\gamma}}{\hat{\alpha}}  \mathbbm{1}_{\{ \gamma > 2 \} } \mathbbm{1}_{\{ \beta = 1 \} }  [ \nabla_{\beta,\gamma} G(0)]^2.
\end{align*}
Indeed, if $\gamma >2$ and $\beta=1$ $G \in  \mathcal{S}_{Rob}(\mathbb{R}^{*})$ and the second term in the third line of last display is equal to $2 \kappa_{\gamma} [\nabla_{\beta,\gamma} G(0)]^2 / \hat{\alpha}$. Otherwise, we get $G \in  \mathcal{S}_{Neu}(\mathbb{R}^{*})$, then $G'_{-}(0)=G'_{+}(0)=0$ and the second term in the third line is equal to zero. By combining the expression in last line with \eqref{eqprop1lem1convmartc2}, the proof ends.
\end{proof}
We end this section with a result that was useful to prove Lemma \ref{lemconvmartterm2}.
\begin{prop} \label{proplemconvmart2}
For every $G \in \mathcal{S}(\mathbb{R})$, we have
\begin{align}  \label{eqproplemconvmart2a}
\lim_{n \rightarrow \infty}   \frac{[ \Theta(n) ]^2}{n^2} \Big\{ \sum_{ x,y } [ G( \tfrac{y}{n} ) - G( \tfrac{x}{n} )]^4 p^2(y-x) +  \sum_{ x } \Big[ \sum_{y  } p(y-x) [ G( \tfrac{y}{n} ) - G( \tfrac{x}{n} )]^2 \Big]^2 \Big\} =0.
\end{align}
\end{prop}
\begin{proof}
Since $(x_1+x_2)^2 \leq 2(x_1^2+x_2^2)$, \eqref{eqproplemconvmart2a} is a consequence of the next three results:
\begin{align}
 &\lim_{n \rightarrow \infty}  \frac{[ \Theta(n) ]^2}{n^2} \sum_{|x|=1 }^{4n} \Big\{ \sum_{|r|=n}^{\infty}  p^2(r)   [G( \tfrac{x+r}{n}) - G( \tfrac{x}{n})]^4  + \Big[ \sum_{|r|=n}^{\infty} p(r) [ G( \tfrac{x+r}{n} ) - G( \tfrac{x}{n} ) ]^2 \Big]^2 \Big \} =0, \label{eqproplemconvmart2a1} \\
  &\lim_{n \rightarrow \infty}  \frac{[ \Theta(n) ]^2}{n^2} \sum_{|x|=4n}^{\infty} \Big\{ \sum_{|r|=n}^{\infty}  p^2(r)   [G( \tfrac{x+r}{n}) - G( \tfrac{x}{n})]^4  + \Big[ \sum_{|r|=n}^{\infty} p(r) [ G( \tfrac{x+r}{n} ) - G( \tfrac{x}{n} ) ]^2 \Big]^2 \Big \}=0, \label{eqproplemconvmart2a2} \\
 & \lim_{n \rightarrow \infty}  \frac{[ \Theta(n) ]^2}{n^2} \sum_{x} \Big\{ \sum_{|r| = 1}^n  p^2(r) [G( \tfrac{x+r}{n}) - G( \tfrac{x}{n})]^4 + \Big[ \sum_{|r|=1}^{n} p(r) [ G( \tfrac{x+r}{n} ) - G( \tfrac{x}{n} ) ]^2 \Big]^2  \Big\} =0. \label{eqproplemconvmart2a3}
\end{align}
In order to prove \eqref{eqproplemconvmart2a1}, we bound the expression inside the limit from above by a constant times
\begin{align*}
& \frac{[ \Theta(n) ]^2}{n^2} ( \| G \|_{\infty})^{4} \sum_{|x|=1 }^{4n} \Big\{ \sum_{|r|=n}^{\infty}  |r|^{-2\gamma-2}   + \Big[ \sum_{|r|=n}^{\infty}  |r|^{-\gamma-1}  \Big]^2 \Big \} \\
\lesssim & \frac{[ \Theta(n) ]^2}{n^2}  \sum_{x=1 }^{4n} \Big\{ n^{-2 \gamma -1} \frac{1}{n} \sum_{r=n}^{\infty}  \Big( \frac{r}{n} \Big) ^{-2\gamma-2}   + \Big[ n^{-\gamma} \frac{1}{n} \sum_{r=n}^{\infty}  \Big( \frac{r}{n} \Big) ^{-\gamma-1}  \Big]^2\Big \}  \\
\lesssim & \frac{[ \Theta(n) ]^2}{n^2}  \sum_{x=1 }^{4n} \Big\{ n^{-2 \gamma -1} \int_{1}^{\infty} u^{-2 \gamma-2} du +   \Big[ n^{-\gamma} \int_{1}^{\infty} u^{- \gamma-1} du \Big]^2 \Big \} \lesssim \frac{[ \Theta(n) ]^2}{n^{2 \gamma +1}},
\end{align*}
which goes to zero as $n \rightarrow \infty$. Now we prove \eqref{eqproplemconvmart2a2}. Since $(x_1-x_2)^k \leq C_k (x_1^{k}+x_2^{k})$ for some positive constant $C_k$ depending only on $k$, the expression inside the limit is bounded from above by a constant times
\begin{align}
& \frac{[ \Theta(n) ]^2}{n^{2\gamma+1}} \int_{|u| \geq 4} [G( u)]^4 \Big\{ n^{-1}  \int_{|v| \geq 1}  |v|^{-2\gamma - 2} dv + \Big[  \int_{|v| \geq 1} |v|^{-\gamma - 1} dv \Big]^2 \Big\} du \label{eqproplemconvmart2a2a}  \\
+ & \frac{[ \Theta(n) ]^2}{n^{2\gamma+1}} \int_{|u| \geq 4}  \Big\{ n^{-1}  \int_{ |v| \geq \frac{|u|}{2}}  [G( u+v)]^4 |v|^{-2\gamma - 2} dv + \Big[  \int_{ |v| \geq \frac{|u|}{2}} [G( u+v)]^2 |v|^{-\gamma - 1} dv \Big]^2 \Big\} du \label{eqproplemconvmart2a2b} \\
+ & \frac{[ \Theta(n) ]^2}{n^{2\gamma+1}} \int_{|u| \geq 4}  \Big\{ n^{-1}  \int_{1 \leq |v| \leq \frac{|u|}{2}}  [G( u+v)]^4 |v|^{-2\gamma - 2} dv + \Big[  \int_{1 \leq |v| \leq \frac{|u|}{2}} [G( u+v)]^2 |v|^{-\gamma - 1} dv \Big]^2 \Big\} du. \label{eqproplemconvmart2a2c}
\end{align}
We bound the expression in \eqref{eqproplemconvmart2a2a} from above by a constant times
\begin{align*}
\frac{[ \Theta(n) ]^2}{n^{2\gamma+1}} \int_{|u| \geq 4} [G( u)]^4 \Big\{ n^{-1}   + 1  ]^2 \Big\} du \lesssim \frac{[ \Theta(n) ]^2}{n^{2\gamma+1}} \int_{|u| \geq 4} [G( u)]^4 du \lesssim\frac{[ \Theta(n) ]^2}{n^{2\gamma+1}} ,
\end{align*}
and this goes to zero as $n \rightarrow \infty$. Above we used $G \in \mathcal{S}(\mathbb{R}) \subset L^4(\mathbb{R})$. Finally, we bound the expression in \eqref{eqproplemconvmart2a2b} from above by
\begin{align*}
& \frac{[ \Theta(n) ]^2}{n^{2\gamma+1}} \int_{|u| \geq 4}  \Big\{ n^{-1}  \int_{ |v| \geq \frac{|u|}{2}}  [G( u+v)]^4 \Big( \frac{|u|}{2} \Big)^{-2\gamma - 2} dv + \Big[  \int_{ |v| \geq \frac{|u|}{2}} [G( u+v)]^2  \Big( \frac{|u|}{2} \Big)^{-\gamma - 1} dv \Big]^2 \Big\} du \\
\lesssim & \frac{[ \Theta(n) ]^2}{n^{2\gamma+1}} \int_{|u| \geq 4}  \Big\{ n^{-1} |u|^{-2\gamma-2} \int_{ |v| \geq \frac{|u|}{2}}  [G( u+v)]^4  dv + \Big[   |u|^{-\gamma-1} \int_{ |v| \geq \frac{|u|}{2}} [G( u+v)]^2   dv \Big]^2 \Big\} du \\
\leq & \frac{[ \Theta(n) ]^2}{n^{2\gamma+1}} \int_{|u| \geq 4}  \Big\{ n^{-1} |u|^{-2\gamma-2} \int_{ \mathbb{R}}  [G(w)]^4  dw + \Big[   |u|^{-\gamma-1} \int_{ \mathbb{R}} [G( w)]^2   dw \Big]^2 \Big\} du \\
\lesssim & \frac{[ \Theta(n) ]^2}{n^{2\gamma+1}} \int_{|u| \geq 4}  \{ n^{-1} |u|^{-2\gamma-2}  +   |u|^{-2\gamma-2}  \} du \lesssim \frac{[ \Theta(n) ]^2}{n^{2\gamma+1}} ( n^{-1} + 1) \lesssim \frac{[ \Theta(n) ]^2}{n^{2\gamma+1}}, 
\end{align*}
which goes to zero as $n \rightarrow \infty$. Above we used $G \in \mathcal{S}(\mathbb{R}) \subset  L^2(\mathbb{R}) \cap  L^4(\mathbb{R}) $. We bound the expression in \eqref{eqproplemconvmart2a2c} from above by
\begin{align*}
& \frac{[ \Theta(n) ]^2}{n^{2\gamma+1}} \int_{|u| \geq 4}  \Big\{ n^{-1}  \int_{1 \leq |v| \leq \frac{|u|}{2}}  [G( u+v)(u+v)]^4 \frac{|v|^{-2\gamma - 2}}{(u+v)^{4}}  dv + \Big[  \int_{1 \leq |v| \leq \frac{|u|}{2}} [G( u+v)(u+v)]^2 \frac{|v|^{-\gamma - 1}}{(u+v)^2}  dv \Big]^2 \Big\} du \\
\leq &\frac{[ \Theta(n) ]^2}{n^{2\gamma+1}} \int_{|u| \geq 4}  \Big\{ n^{-1}  \int_{1 \leq |v| \leq \frac{|u|}{2}}  [ \sup_{w \in \mathbb{R}} |w G(w)| ]^4 16 \frac{|v|^{-2\gamma - 2}}{(u)^{4}}  dv + \Big[  \int_{1 \leq |v| \leq \frac{|u|}{2}} [\sup_{w \in \mathbb{R}} |w G(w)|]^2 4 \frac{|v|^{-\gamma - 1}}{(u)^2}  dv \Big]^2 \Big\} du \\
\lesssim & \frac{[ \Theta(n) ]^2}{n^{2\gamma+1}} \int_{|u| \geq 4}  \Big\{ n^{-1} u^{-4}  \int_{|v| \geq 1} |v|^{-2\gamma - 2}   dv + \Big[  u^{-2} \int_{|v| \geq 1}  |v|^{-\gamma - 1}   dv \Big]^2 \Big\} du \\
\lesssim &\frac{[ \Theta(n) ]^2}{n^{2\gamma+1}} \int_{|u| \geq 4} u^{-4} (  n^{-1}  + 1 ) du \lesssim \frac{[ \Theta(n) ]^2}{n^{2\gamma+1}} (n^{-1}  +  1)  \lesssim \frac{[ \Theta(n) ]^2}{n^{2\gamma+1}}, 
\end{align*}
and this goes to zero as $n \rightarrow \infty$. This ends the proof of \eqref{eqproplemconvmart2a2}. In order to prove \eqref{eqproplemconvmart2a3}, we define $F_3: \mathbb{R} \rightarrow \mathbb{R}$ by
\begin{align*}
F_3(u):= \sup_{| v- u| \leq 1}  |G' (v )|, \forall u \in \mathbb{R}.
\end{align*}
From $G' \in \mathcal{S}( \mathbb{R})$, we get $F_3 \in L^4(\mathbb{R})$. From the Mean Value Theorem the expression in \eqref{eqproplemconvmart2a3} is bounded from above by
\begin{align*}
 & \frac{[ \Theta(n) ]^2}{n^{2}} \sum_{x } \Big\{ \sum_{|r|=1}^{n} p^2(r)   [  r n^{-1} F_3 ( \tfrac{x}{n})]^4 + \Big[ \sum_{|r|=1}^{n} p(r)  [ r n^{-1} F_3 ( \tfrac{x}{n} ) ]^2 \Big]^2  \Big\} \\
=& \frac{[ \Theta(n) ]^2}{n^{6}} \sum_{x} [   F_3 ( \tfrac{x}{n})]^4 \Big\{\sum_{|r|=1}^{n} r^{4}p^2(r) + \Big[ \sum_{|r|=1}^{n} r^2p(r)   \Big]^2   \Big\}  \\
\lesssim &  \frac{1}{n} \sum_{x} [   F_3 ( \tfrac{x}{n})]^4 \sum_{r=1}^{\infty} r^{2-2 \gamma} +n^{-1} \frac{1}{n} \sum_{x} [   F_3( \tfrac{x}{n})]^4 \Big\{ \Big[ \frac{\Theta(n)}{n^2} \sum_{r=1}^{n} r^2p(r)   \Big]^2   \\
\lesssim& \Big(\frac{[ \Theta(n) ]^2}{n^{5}} + n^{-1} \Big) \int_{\mathbb{R}} [F_3(u)]^{4} du  \lesssim n^{-1},
\end{align*}
which goes to zero as $n \rightarrow \infty$. In the third line we used \eqref{thetan}. This ends the proof.
\end{proof}

\quad

\thanks{ {\bf{Acknowledgements: }}
P.C. thanks FCT/Portugal for financial support through the project Lisbon Mathematics PhD (LisMath), during which most of this work was done. P.C. and P.G. thank  FCT/Portugal for financial support through the
projects UIDB/04459/2020 and UIDP/04459/2020.  B.J.O. thanks  Universidad Nacional de Costa Rica  for sponsoring the participation in  this article. This project has received funding from the European Research Council (ERC) under  the European Union's Horizon 2020 research and innovative programme (grant agreement   n. 715734).}  \textbf{Pedro Cardoso}:  Writing - Original Draft, Formal analysis, Investigation; \textbf{Patr\'icia Gon\c calves}: Conceptualization, Methodology, Supervision, Investigation, Writing - Review  and  Editing \textbf{Byron Jim\'enez-Oviedo}: Writing - Review  and  Editing

\bibliographystyle{plain}
\bibliography{bibliografia}

\end{document}